\documentclass[10pt]{article}
\usepackage{amssymb,amsmath,amsthm, amscd, mathrsfs, eucal}
\usepackage[colorlinks=true]{hyperref}
\usepackage{mathtools}
\usepackage{enumitem}
\usepackage {hyperref}
\usepackage[all]{xy}
\usepackage{xcolor}
\usepackage{soul}
\usepackage{faktor}
\usepackage[normalem]{ulem}
\usepackage[inkscapelatex=false]{svg}
\usepackage{subfigure}
\usepackage{cleveref}
\crefformat{section}{\S#2#1#3}
 \usepackage{import}
\usepackage{caption}
\usepackage{subcaption, pinlabel}

\theoremstyle{plain}
\newtheorem{thm}{Theorem}[section]
\newtheorem{question}{Question}

\newtheorem{lem}[thm]{Lemma}
\newtheorem{cor}[thm]{Corollary}
\newtheorem{prop}[thm]{Proposition}

\theoremstyle{definition}

\newtheorem{defn}[thm]{Definition}
\newtheorem{rmk}[thm]{Remark}
\newtheorem{rem}[thm]{Remark}

\newtheorem{example}[thm]{Example}

\newtheorem{claim}{Claim}

\newcommand{\mc}{\mathcal}
\newcommand{\geodto}{\xrightarrow{\mathrm{geod}}}
\newcommand{\CCto}{\xrightarrow{CC}}
\newcommand{\Cto}{\xrightarrow{C}}

\newcommand{\ignore}[1]{}
\newcommand{\supp}{\mathrm{supp}}

\newlist{casenv}{enumerate}{4}
\setlist[casenv]{leftmargin=*,align=left,widest={iiii}}
\setlist[casenv,1]{label={{\itshape\ \casename} \arabic*.},ref=\arabic*}
\setlist[casenv,2]{label={{\itshape\ \casename} \roman*.},ref=\roman*}
\setlist[casenv,3]{label={{\itshape\ \casename\ \alph*.}},ref=\alph*}
\setlist[casenv,4]{label={{\itshape\ \casename} \arabic*.},ref=\arabic*}

\providecommand{\casename}{Case}

\newcounter{ccomments}

\newcounter{BNcomments}

\newcommand{\rom}[1]{\text{\uppercase\expandafter{\romannumeral #1\relax}}}

\def\Z{\mathbb Z}

\def\Rcal{\mathcal{R}}
\def\winf{\widetilde{\infty}}
\def\defeq{\vcentcolon=}

\makeatletter
\newsavebox{\@brx}
\newcommand{\llangle}[1][]{\savebox{\@brx}{\(\m@th{#1\langle}\)}%
	\mathopen{\copy\@brx\kern-0.5\wd\@brx\usebox{\@brx}}}
\newcommand{\rrangle}[1][]{\savebox{\@brx}{\(\m@th{#1\rangle}\)}%
	\mathclose{\copy\@brx\kern-0.5\wd\@brx\usebox{\@brx}}}

\newcommand{\appauthor}[1]{%
  \vspace{-0.5em}\textit{By #1}\vspace{0.5em}%
}

\title{Gromov boundary of the Grand Arc Graph}
\author{Carolyn Abbott \and Assaf Bar-Natan \and Arya Vadnere \and  \large{\textit{With an appendix by} Lvzhou Chen}}
\date{}

\begin{document}

\maketitle

\begin{abstract}
We describe a dense subset of the Gromov boundary of the grand arc graph of an infinite-type surface as a space of geodesic laminations, analogous to Klarreich's description of the Gromov boundary of the curve complex. After showing that the grand arc graph satisfies a bounded geodesic image theorem, we also prove that the boundary is not compact.
\end{abstract}


\section{Introduction}


A fruitful technique for studying the mapping class group of a surface $\Sigma$ is to study its isometric action on various combinatorial models associated to $\Sigma$.  When $\Sigma$ is finite-type, that is, has finitely generated fundamental group, two common combinatorial models are the \textit{curve graph $\mathcal{C}(\Sigma)$} and the \textit{arc graph $\mathcal{A}\left(\Sigma\right)$}, both of which are infinite-diameter hyperbolic graphs.  The study of these actions has led to advances in our understanding of the mapping class group (for example, \cite{MM2, ivanov, BF}).  Gromov showed that isometries of a hyperbolic space are encoded by the diameter of their orbits and their fixed points on the \textit{Gromov boundary} of the space, or the set of points ``at infinity" \cite{Gromov:HG}.  Thus, when studying groups acting on a hyperbolic space, it is useful to understand the boundaries of these spaces.  Klarreich described the boundary $\partial \mc C(\Sigma)$ of $\mc C(\Sigma)$ as the space of \textit{ending laminations} on $\Sigma$ \cite{Klar}, that is, minimal geodesic laminations that fill the surface.  Pho-On obtained a similar description of the boundary $\partial \mc A(\Sigma)$ as the space of laminations that fill a subsurface containing all of the punctures 
\cite{Pho-On} and gives an alternate proof of Klarreich's result\footnote{Pho-On notes that Schleimer independently obtained these results as well, citing personal communication.}.  More recently,  Bowden, Hensel and Webb also described the boundary of the fine curve graph as a space of laminations  \cite{BHW_boundary}. 

When $\Sigma$ is infinite-type, on the other hand, both the arc graph and the curve graph are finite diameter.  While the study of these graphs can still be fruitful \cite{HMV, BDR}, the graphs do not have an interesting coarse geometry and, in particular, their Gromov boundary is empty. 
There are several alternative candidates for combinatorial models for infinite-type surfaces, all of which are graphs of certain arcs and/or curves. For various classes of surfaces, these graphs are infinite-diameter and hyperbolic.  Such graphs include the relative arc graph \cite{AFP}, the omnipresent arc graph \cite{F-G-M}, the non-peripheral curve graph \cite{npcurve}, the non-separating curve graph \cite{Ras_nsep}, and the ray graph \cite{Calegari, Bav}. For these graphs, only the boundary of the relative arc graph (of which the ray graph is a special case) is understood.  Bavard and Walker described the boundaries of these graphs as a space of ``high-filling rays" \cite{B-W, B-W_2}; see Section~\ref{sec:SurjectivityIsHard} for further details.

In this paper, we focus on the \textit{grand arc graph} $\mc G(\Sigma)$, which was defined by the second author and Verberne \cite{B-NV}. The vertices of this graph are isotopy classes of \textit{grand arcs}, which, very roughly speaking, are arcs between certain maximal ends of the surface, and edges correspond to disjointness.  For a large collection of surfaces, 
this graph is also infinite diameter, uniformly $\delta$-hyperbolic, and admits an isometric action of the mapping class group with finitely many orbits. 
In particular, this holds for surfaces $\Sigma$ whose \textit{grand splitting} $\mc S(\Sigma)$ contains at least three but finitely many elements; see Section~\ref{sec:background_surfaces}.   
For these surfaces, we provide the first steps toward describing the boundary $\partial \mc G(\Sigma)$ of the grand arc graph in terms of \textit{geodesic laminations} on $\Sigma$.  In what follows, $\mc{GEL}(\Sigma)$ is the space of \textit{grand ending laminations}, which are minimal geodesic laminations that fill a witness of $\mc{G}(\Sigma)$; see Definition~\ref{def:GEL}.

\begin{thm}\label{thm_main}
    Let $\Sigma$ be an infinite-type surface with $3\leq |\mc S(\Sigma)|<\infty$.  There is a $\mathrm{MCG}(\Sigma)$-equivariant map $F\colon \mathcal{GEL}(\Sigma) \to \partial \mathcal G(\Sigma)$ that is a homeomorphism onto its image, which is dense in $\partial\mathcal G(\Sigma)$.
\end{thm}

Graphs of arcs and curves are often studied via their \textit{witness subsurfaces}; see Definition~\ref{defn:witness}.  We say that a geodesic lamination $L$ \textit{fills} a witness if its support, in the sense of Defintion~\ref{def:support}, is a witness.  Our definition of grand ending laminations draws inspiration from Pho-On's definition of \textit{peripherally ending laminations} \cite{Pho-On}. We equip the space $\mc{GEL}(\Sigma)$ with the \textit{coarse Chabauty topology}, defined in Section \ref{subsec:chabauty}, which makes it Hausdorff and first-countable. 

Theorem~\ref{thm_main} does not prove that the map $F$ is a homeomorphism: it is unclear whether $F$ is surjective.  One obstruction to proving surjectivity is that geodesic laminations on infinite-type surfaces can behave quite differently than in the finite-type setting and are not well-understood in general.  See Section~\ref{sec:SurjectivityIsHard} for a further discussion.

Our approach to proving Theorem \ref{thm_main} is inspired by Pho-On, who uses \textit{unicorn arcs} to describe the boundary of the arc graph and the curve graph of a finite-type surface; unicorn arcs were also used by the second author and Verberne to prove hyperbolicity of $\mc G(\Sigma)$.  To understand the boundary $\partial \mc G(\Sigma)$, we follow Pho-On's approach, but using \textit{bicorn arcs}; see Definitions~\ref{def:bicorn} and \ref{def:bicorn_with_geodesic}.  The key observation that makes bicorn arcs better suited for studying the boundary than unicorn arcs is that a bicorn between a grand arc and \textit{any} geodesic on $\Sigma$ will again be a grand arc, while a unicorn arc may not.  
Given a grand ending lamination $L$ and a grand arc $a$, we construct an \textit{ending bicorn sequence}, which is a sequence of bicorns between $a$ and $L$.  We prove that this sequence: (1) forms a quasi-path in $\mc{G}(\Sigma)$ that converges to a point in $\partial\mc{G}(\Sigma)$; and (2) geodesically converges to a leaf of $L$. 
The map $F$ then sends $L$ to (the equivalence class of) its ending bicorn sequence.

To further understand $\mc{G}(\Sigma)$ and its boundary, we prove a \textit{bounded geodesic image theorem} for the grand arc graph, which is analogous to Masur and Minsky's bounded geodesic image theorem for the curve graph of a closed surface \cite{MM2}.  A \textit{partial witness} is a subsurface $Y$ of $\Sigma$ that is not a witness of $\mc{G}(\Sigma)$ but essentially intersects every witness of $\mc{G}(\Sigma)$; see Definition~\ref{def:partialwitness}. 


\begin{thm}[Theorem \ref{thm:BGI}]\label{thm:BGI_Intro}
    Let $\Sigma$ be an infinite-type surface with $3\leq |\mc S(\Sigma)|<\infty$. There exists a constant $B>0$ such that the following holds. If $Y\subset\Sigma$ is a partial witness for $\mathcal{G}(\Sigma)$ and $g=\left(\gamma_{i}\right)_{i\in I}$ is a geodesic in $\mathcal{G}\left(\Sigma\right)$ such that each $\gamma_{i}$ intersects $Y$ transversely, then 
\[
\mathrm{diam}_{\mathcal{A}(Y)}\left(g\right)\leq B.
\]
\end{thm}

Theorem~\ref{thm:BGI_Intro} can be used to show that $\partial \mc{G}(\Sigma)$ is not compact.

\begin{cor}[Corollary \ref{cor:non-compactness}]
    If $\Sigma$ is an infinite-type surface with $4\leq |\mc S(\Sigma)|<\infty$, then $\partial \mc G(\Sigma)$ is not compact.
\end{cor}

The description of (part of) $\partial\mc{G}(\Sigma)$ as geodesic laminations allows us to use techniques from \cite{RasWWPD} to understand the \textit{WWPD elements} of $\mathrm{MCG}(\Sigma)$ in the action on $\mc{G}(\Sigma)$; see Definition~\ref{def:WWPD}. 

\begin{thm}[Theorem \ref{thm:WWPD_classif}]
Suppose $\Sigma$ is an infinite-type surface with $3\leq |\mc S(\Sigma)|<\infty$. Suppose $\varphi\in\mathrm{MCG}(\Sigma)$ is loxodromic, and the fixed points of $\varphi$ on $\partial\mc{G}(\Sigma)$ are in the image $F(\mc{GEL}(\Sigma))$. Then $\varphi$ is WWPD if and only if $\varphi$ is pseudo-Anosov on a compact witness.
\end{thm}

This allows us to recover a result by Horbez, Qing and Rafi \cite{HQR} that the second bounded cohomology of $\mathrm{MCG(\Sigma)}$ is infinite dimensional; see Corollary~\ref{cor:BddCohom}.

\subsection{Surjectivity and comparison with Bavard--Walker}\label{sec:SurjectivityIsHard}
The main difficulty in proving surjectively of the map $F$ is that laminations on infinite-type surfaces are not well understood.  Suppose $p\in\partial\mc{G}(\Sigma)$, and let $(x_n)_{n\in\mathbb{N}}$ be a quasi-geodesic sequence of grand arcs Gromov-converging to $p$. Treating each $x_i$ as a lamination with a single leaf, by the compactness of the space of geodesic laminations with the Chabauty topology, the sequence of laminations $\{x_n\}$ converges up to subsequence in the Chabauty topology to some lamination $L$. If $L$ has a minimal sublamination $L'$ that fills a witness of $\mc{G}(\Sigma)$, then this minimal componenet would be the candidate for $F^{-1}(p)$. 

When $\Sigma$ is finite-type, any geodesic lamination $L$ on $\Sigma$ can be decomposed into minimal sublaminations connected with isolated leaves; see, for example, \cite{CEG06}. In that setting, surjectivity is proven by arguing that exactly one of the minimal components of $L$ must fill a witness.  
However, when $\Sigma$ is  infinite-type, there exist geodesic laminations that do not have any minimal sublamination, that is, it is possible to construct an infinite nested sequence of (non-empty) sublaminations $L_1\supsetneq L_2\supsetneq\dots$ such that no $L_i$ is minimal. Examples of such ``infinitely derivable" laminations were constructed by Bestvina and Rasmussen \cite[Section 8.1]{BR}. It is unclear whether such laminations can arise in this context. If they do, minimality may need to be replaced with an appropriate weaker condition to completely describe $\partial \mc G(\Sigma)$.

\begin{question}
    Suppose $(x_n)_{n\in \mathbb N}$ is a sequence of grand arcs that converges the boundary in $\mc G(\Sigma)$.  Let $L$ be the limit of (a subsequence of) the sequence in the space of geodesic laminations equipped with the Chabauty topology.  Could $L$ be infinitely derivable?
\end{question}

In the special case when $\Sigma$ is a once-punctured Cantor tree, the grand arc graph is the same as the ray graph, which was defined by Calegari and proven to by infinite-diameter and hyperbolic by Bavard \cite{Bav}.  Bavard and Walker  prove that the Gromov boundary of the ray graph is homeomorphic to the space $\mc{E}$ of \textit{cliques of high-filling rays} \cite{B-W}. They later generalized their results to describe the boundary of the relative arc graph for other surfaces with an isolated puncture \cite{B-W_2}; however, the relative arc graph for those surfaces is no longer the same as (nor quasi-isometric to) the grand arc graph. To prove their result, they construct a \textit{completed ray graph}, which is a larger graph containing the ray graph, with the property that the connected component containing the ray graph is quasi-isometric to the ray graph. The ``cliques of high-filling rays" are exactly the other connected components of this graph. A key ingredient in their arguments is infinite unicorn paths as constructed in \cite{Pho-On}.  

\begin{question}
    Is there a general combinatorial description of $\partial\mc{G}(\Sigma)$ analogous to $\mc{E}$?
\end{question}

When $\Sigma$ is the plane minus a Cantor set, so that the grand arc graph and the ray graph coincide, there is a relationship between $\mc{E}$ and $\mc{GEL}(\Sigma)$.  Let $\mc E_{\textrm{min}}\subseteq \mc E$ be the subspace of $\mc E$ consisting of high-filling rays whose limit set is a minimal lamination. Rasmussen's techniques from \cite[Section~3]{RasWWPD} work verbatim to show that $\mc E_{\textrm{min}}$ is homeomorphic to $\mc{GEL}(\Sigma)$. It is possible that $\mc E_{\textrm{min}}=\mc E$, in which case the map $F$ from Theorem~\ref{thm_main} would be a homeomorphism when $\Sigma$ is the plane minus a Cantor set.


There are several similarities between the structure of our argument to describe (part of) the boundary of the grand arc graph and Bavard--Walker's argument describing the boundary of the relative arc graph of an infinite-type surface with an isolated puncture. Thus it is reasonable to wonder whether the ideas in their proof of the surjectivity of their map could be used to prove surjectivity of the map $F$.  Unfortunately, their proof of surjectivity has a small gap, which has been filled by Lvzhou Chen in Appendix~\ref{sec:appendix}. Chen's proof heavily relies on the Bavard-Walker machinery, and, in particular, on the fact that all arcs converge to a fixed puncture.  It is precisely this piece of the proof of surjectivity that does not seem to generalise to grand arcs. 



\paragraph*{Organization of the paper.}  After gathering background on hyperbolic spaces, infinite-type surfaces, and the grand arc graph in Section~\ref{sec:background}, we define bicorn arcs between grand arcs and investigate their properties in Section~\ref{sec:finite_bicorns}.  In Section~\ref{sec:BGI}, we prove the bounded geodesic image theorem for the grand arc graph and use it to show that the boundary of the grand arc graph is not compact in Theorem~\ref{thm:BGI} and Corollary~\ref{cor:non-compactness}, respectively.  Section~\ref{sec:topology} focuses on laminations on surfaces, defining the (coarse) Chabauty topology on the space of laminations and grand ending laminations.  In Section~\ref{sec:bicorns_with_generic_geod}, we define bicorn arcs between a grand arc and an arbitrary geodesic.  In Section~\ref{sec:GEL_boundary}, we define the map $F\colon \mc{GEL}(\Sigma)\to \partial\mc G(\Sigma)$ and prove Theorem~\ref{thm_main}, which we then use to describe WWPD elements in Theorem~\ref{thm:WWPD_classif}.

\paragraph*{Acknowledgements.} The authors would like to thank Mladen Bestvina, Lvzhou Chen, George Domat, Johanna Mangahas,  and Alexander Rasmussen for guidance and helpful discussions. We would especially like to thank Kasra Rafi for suggesting Definition~\ref{def:support}, Mladen Bestvina, Federica Fanoni, Kasra Rafi, and Jing Tao for the proof of Lemma~\ref{lem:both_half_leaves_dense}, and Yvon Verberne for useful discussions at the inception of this project.  The first author was partially supported by NSF grants DMS-2106906 and DMS-2340341. The third author was partially supported by the Simons Foundation (965204, JM).

\section{Background}\label{sec:background}

Let $X$ be a metric space.  If $Y\subseteq X$, we let $N_\kappa(Y)$ denote the (closed) $\kappa$--neighbourhood of $Y$ in $X$.

A map $\gamma\colon I\to X$ is a \textit{geodesic} if it is an isometric embedding of the interval $I\subseteq R$ into $X$.  It is a \textit{$(K,C)$--quasi-geodesic} if, for all $a,b\in I$, 
        \[
        \frac1K|b-a|-c\leq d_X(\gamma(a),\gamma(b))\leq k|b-a|+C.
        \]
We often conflate a (quasi-)geodesic with its image in $X$.

Given a metric space $X$ and subsets $A,B\subseteq X$, we denote the Hausdorff distance between $A$ and $B$ by $d_{\rm{Haus}}^X(A,B)$.

\subsection{$\delta$-hyperbolic spaces and their boundaries}

We refer the reader to \cite[Section III-H]{BH} for a detailed exposition of $\delta$-hyperbolic spaces. In this subsection, we outline some of their key properties. 

\begin{defn}
    A geodesic metric space $X$ is \textit{$\delta$--hyperbolic} for some constant $\delta\geq 0$ if for every geodesic triangle $\alpha_1\cup \alpha_2\cup \alpha_3$, we have $\alpha_3\subseteq N_\delta (\alpha_1\cup \alpha_2)$.
\end{defn}

A key feature of $\delta$--hyperbolic spaces is that all quasi-geodesics with the same endpoints are contained in a uniform neighbourhood of each other.

\begin{lem}[Morse Lemma]\label{lem:Morse_lemma}
    Let $X$ be a $\delta$-hyperbolic space. Then for every $\lambda\geq 1,\epsilon \geq 0$, there exists a constant $m_0=m_0(\delta,\lambda,\epsilon)$ such that for any $(\lambda,\epsilon)$-quasigeodesic $\gamma$ in $X$ connecting $x,y\in X$, we have 
    \begin{align}
        d_\mathrm{Haus}([x,y],\gamma)<m_0. 
    \end{align} 
\end{lem}

Let $X$ be a $\delta$-hyperbolic space, and let $a\in X$ be a fixed base point. The \textit{Gromov product} between two points $x,y\in X$ is 
\[
\left(x,y\right)_{a}=\frac{1}{2}\left(d\left(a,x\right)+d\left(a,y\right)-d\left(x,y\right)\right).
\]
Intuitively, the Gromov product inversely corresponds to the ``angle" between $x$ and $y$ from the point of view of $a$.  If the Gromov product is large, then $x$ and $y$ are closer to each other than they are to $a$. More precisely, the Gromov product is coarsely equal to the distance from $a$ to a geodesic from $x$ to $y$:

\begin{lem}\label{lem:Gromov_product_vs_dist_to_geod}
    Let $X$ be a $\delta$--hyperbolic space, let $x,y,a\in X$, and let $[x,y]$ be a geodesic from $x$ to $y$.  Then
    \[
    \left|d\left(a,\left[x,y\right]\right)-\left(x,y\right)_{a}\right|\leq 2\delta.
    \]
\end{lem}

Another feature of hyperbolic spaces is the notion of its \textit{boundary at infinity}, which is quasi-isometry invariant. It is defined in terms of convergent sequences of points.

\begin{defn}
A sequence $\left(x_{n}\right)_{n\in\mathbb{N}}\subset X$ \textit{converges to infinity} in $X$ if 
\[
\liminf_{i,j\to\infty}\left(x_{i},x_{j}\right)_{a}=\infty.
\]
Two sequences $\left(x_{n}\right)$ and $\left(y_{n}\right)$ are equivalent if 
\[
\liminf_{i,j\to\infty}\left(x_{i},y_{j}\right)_{a}=\infty.
\]
The \textit{Gromov boundary} $\partial_\infty X$ of $X$ is the collection of equivalence classes of sequences $\left(x_{n}\right)$ that converge to infinity in $X$. 
\end{defn}

This is a topological space, with basic open sets:
\[
U_{R}\left(p\right)=\left\{ q\in\partial X\mid\exists\left(x_{n}\right)\in p\text{ and }\left(y_{n}\right)\in q\text{ such that }\liminf_{i,j}\left(x_{i},y_{j}\right)_{a}>R\right\} .
\]

If $X$ is a proper geodesic metric space, then the Gromov boundary also corresponds to the set of equivalence classes of geodesic rays emanating from a fixed base point,  where two rays are equivalent if they are at finite Hausdorff distance.  While this is useful for intuition, it will be more convenient in this paper to work with convergent sequences. 


We end with two technical lemmas that will be useful in this paper.  The first is well-known, and can be found, for example, in \cite[Remark~2.6]{BenakliKapovich}.

\begin{lem}\label{lem:Qgeo_to_boundary}
If $X$ is a $\delta$--hyperbolic space, then given any point $\xi\in \partial X$ and $x_0\in X$, there exists a $(1,10\delta)$--quasi-geodesic $\gamma\colon \mathbb R_{\geq 0} \to X$ with $\gamma(0)=x_0$ such that $[(\gamma(n))_{n\in \mathbb N}]=\xi$.
\end{lem}

\begin{lem}\label{lem:QpathNearQgeo}
Let $X$ be a $\delta$--hyperbolic space, let $(x_n)_{n\in \mathbb N}$ be a sequence that converges to infinity, and let $\gamma$ be a $(1,10\delta)$--quasi-geodesic such that $[(\gamma(m))_{m\in \mathbb N}]=[(x_n)_{n\in \mathbb N}]$.  If there exists $K\geq 0$ such that $(x_n)$ is a $K$--quasi-path and $d_{\mathrm{Haus}}(\{x_0,\dots, x_k\},[x_0,x_k])\leq K$ for all $k$, then there exists a constant $C=C(K,\delta)$ such that $d_{\mathrm{Haus}}((x_n)_{n\in \mathbb N},\gamma)\leq C$.
\end{lem}

\begin{proof}
    Fix $i\in \mathbb N$.  Since $x_n\to [\gamma(m)]$, for each $r\in \mathbb N$ there exists $N_r\in \mathbb N$ such that $(x_n,[\gamma(m)])_{x_0} \geq r$ for all $n\geq N_r$.   Choose $r$ large enough so that if $y\in [a,b_{N_r}]$ satisfies $d(y,x_i)\leq K$, then $d(x_0,y)\leq r-1$.  Fix $0<\varepsilon<1$.  Since $(x_{N_r},[\gamma(m)])_{x_0}\geq r$, there exists $I\in \mathbb N$ such that $(x_{N_r},\gamma(I))_{x_0}\geq r-\varepsilon$.  Let $y\in [a,b_{N_r}]$ satisfy $d(y,x_i)\leq K$, and let $z\in [x_0,\gamma(I)]$ satisfy $d(x_0,y)=d(x_0,z)$.  By our choice of $r$, we have $d(x_0,y)=d(x_0,z)\leq r-\varepsilon$, and so $d(y,z)\leq 2\delta$.  Since $\gamma$ is a $(1,10\delta)$--quasi-geodesic, Lemma~\ref{lem:Morse_lemma} and our choice of $y$ imply that $d(x_i,\gamma)\leq K+2\delta + M(1,10\delta)$.  Therefore $(x_n)_{n\in \mathbb N}\subseteq N_{K+2\delta+M(1,10\delta)}(\gamma)$.  

    We now prove the other containment.   Since $X$ is $\delta$--hyperbolic, there exists a constant $\kappa$ depending only on $\delta$ such that the nearest point projection $\pi_\gamma$  onto the $(1,10\delta)$--quasi-geodesic $\gamma$ is $\kappa$--Lipschitz.  Combining this with the fact that $(x_n)_{n\in \mathbb N}$ is a $K$--quasi-path yields that every point $z\in \gamma$ is at distance at most $\kappa K$ from a point in $\pi_\gamma((x_n)_{n\in \mathbb N})$.  by the previous paragraph, for each $i$ we hve $d(x_i,\pi_\gamma(x_i))\leq K + 2\delta + M(1,10\delta)$.  Hence there exists $i$ such that $d(z,x_i)\leq \kappa K + K +2\delta+M(1,10\delta)$.  Therefore, $(x_n)_{n\in \mathbb N}$ and $\gamma$ are at Hausdorff distance at most $\kappa K + K+2\delta+M(1,10\delta)$.
\end{proof}

\subsection{Infinite-type surfaces and grand arc graphs}\label{sec:background_surfaces}

Let $\Sigma$ be an infinite-type surface.  The \textit{space of ends} $E\left(\Sigma\right)$ of $\Sigma$ is the inverse limit $\varprojlim_{K\subseteq \Sigma} \pi_0(\Sigma\setminus K)$, where $K$ ranges over all compact subsets of $\Sigma$ and, if $K_i\subseteq K_j$, then the map in the inverse limit is the inclusion $\phi_{ij}\colon \pi_0(\Sigma\setminus K_j)\to \pi_0(\Sigma\setminus K_i)$. Endowing each $\pi_0(\Sigma\setminus K)$ with the discrete topology ensures that $E(\Sigma)$ is a totally disconnected, separable, metrisable topological space. A subset $U\subseteq E(\Sigma)$ is \textit{self-similar} if for any decomposition $U=U_1\sqcup U_2$ into disjoint clopen sets $U_1,U_2$, either $U$ is homeomorphic to $U_1$ or $U$ is homeomorphic to $U_2$. A subsurface $S\subseteq \Sigma$ \textit{separates} two ends $x,y\in E(\Sigma)$ if $x$ and $y$ are in distinct connected components of $\Sigma\setminus S$.

In \cite{Mann-Rafi}, Mann and Rafi define a partial order $\preceq$ on the space of ends of $\Sigma$ by setting $x\preceq y$ for $x,y\in E\left(\Sigma\right)$ if every clopen neighbourhood of $y$ contains a clopen set homeomorphic to a neighbourhood of $x$. Two ends $x,y$ are \textit{equivalent}, or of the \textit{same type}, if $x\preceq y$ and $y\preceq x$. With respect to this partial order, $E(\Sigma)$ contains maximal elements, which we call \textit{maximal ends} of $\Sigma$.  Mann and Rafi show that if the mapping class group of $\Sigma$ is locally coarsely bounded\footnote{A subset $X$ of a topological group $G$ is \textit{coarsely bounded} if $X$ has finite diameter in every compatible left-invariant metric on $G$.}, then the space of maximal ends of $\Sigma$ is a finite disjoint union of equivalence classes of self-similar ends. In particular, we can write the space $\mathcal{M}\left(\Sigma\right)$ of maximal ends of $\Sigma$ as 
\begin{equation}\label{eqn:GrandSplitting}
\mathcal{M}\left(\Sigma\right)=\sqcup_{i=1}^{n}E_{i} 
\end{equation}
where each $E_{i}$ is self-similar, and so is either a singleton or a Cantor set, and all $e_{i}\in E_{i}$ are equivalent. 
Verberne and the second author call the decomposition in \eqref{eqn:GrandSplitting} the \textit{grand splitting} of $\Sigma$, denoted $\mc{S}(\Sigma)$ \cite{B-NV}. Such a splitting is unique for a given surface. 

\begin{defn}
A simple arc $a$ \textit{converges to an end} $e$ if for any neighbourhood $U$ of $e$, the arc $a$ eventually never leaves this neighbourhood. A simple arc is \textit{grand} if it converges to exactly two maximal ends that lie in different elements of the grand splitting.
\end{defn}

In particular,  if $\Sigma$ has exactly one element in the grand splitting, then it has no grand arcs.  The second author and Verberne use grand arcs to define the grand arc graph \cite[Definition~2.2]{B-NV}.
\begin{defn}\label{def:grand_arcs}
The \textit{grand arc graph }$\mathcal{G}\left(\Sigma\right)$ of an infinite-type surface $\Sigma$ is the graph whose vertices are in bijection with isotopy classes of grand arcs on $\Sigma$ and with an edge between two vertices if the corresponding isotopy classes have disjoint representatives. We denote the combinatorial metric on $\mc G(\Sigma)$ by $d_{\mc G}$.
\end{defn}

A key tool for studying the geometry of graphs of arcs and curves, such as the grand arc graph, is witness subsurfaces.  

\begin{defn}\label{defn:witness}
    Let $S$ be a surface, and let $G$ be a graph whose vertices correspond to isotopy classes of simple closed curves or simple proper arcs on $S$. A subsurface $W$ is a \textit{witness} for $G$ if it essentially intersects representatives of every vertex of $G$.
\end{defn}

A subsurface $S$ of $\Sigma$ is a witness for $\mc G(\Sigma)$ exactly when $S$ separates the components of the grand splitting of $\Sigma$, that is, if, in the notation of \eqref{eqn:GrandSplitting}, the ends in $E_i$ lie in a different connected components of $\Sigma\setminus S$ than those in which the ends in $E_j$ lie for $i\neq j$.  Much like other graphs of arcs and curves, the geometry of the grand arc graph is characterized by its witnesses. Russell and Vokes \cite{RV} and Kopreski \cite{kop_multiarc} study similar results for general arc and curve graphs on finite-type surfaces.
\begin{thm}[{\cite[Theorems~1.1 \& 1.3]{B-NV}}]\label{thm:hyp_grand_arc_graph}
For an infinite-type surface $\Sigma$ with $|\mc S(\Sigma)|<\infty$, the following conditions are equivalent:
\begin{enumerate}
\item $\mathcal{G}\left(\Sigma\right)$ is $\delta$-hyperbolic for some $\delta$ independent of $\Sigma$;
\item $\Sigma$ does not have disjoint witnesses; and 
\item either $|\mc{S}(\Sigma)|\geq 3$ or $\Sigma$ is the once-punctured Cantor tree. 
\end{enumerate}
In this case, $\mathrm{MCG}\left(\Sigma\right)$ acts quasi-continuously on $\mathcal{G}\left(\Sigma\right)$, and the induced action on $\partial\mathcal{G}\left(\Sigma\right)$ is continuous.
\end{thm}
We note that the case of the once-punctured Cantor tree is addressed using different techniques and is due to Bavard \cite{Bav}. Henceforth, unless stated otherwise, we shall assume that either $3\leq |\mc{S}(\Sigma)|< \infty$ or $\Sigma$ is the once-punctured Cantor tree.

The following two lemmas will be used frequently in this paper.  
\begin{lem}[{\cite[Lemma~6.6]{B-NV}}]\label{lem:single_arc_witness}
    For any grand arc $a$, there exists a witness $S\subseteq \Sigma$ such that $a\cap S$ is connected.
\end{lem}


\begin{lem}\label{lem:grand_arc_dist_2}
Let $a,b\in\mc{G}(\Sigma)$ such that $|a\cap b|=1$. Then $d_\mc{G}(a,b)\leq 2$.
\end{lem}
\begin{proof}
    Let $a\cap b = {P}$. Since each of $a,b$ converge to two different elements of the grand splitting, we can choose one component of $a\backslash P$ and one of $b\backslash P$ which converge to distinct elements of the grand splitting. The concatenation of these two components at $P$ defines a grand arc that is disjoint from $a$ and $b$, and so $d_\mc{G}(a,b)\leq 2$.
\end{proof}

We often measure the Hausdorff distance between subsets of $\mc G(\Sigma)$.  When $\Sigma$ is clear from context, we denote this distance by $d_{\rm{Haus}}^{\mc G}$.

\subsection{Prescribed arc graphs} \label{subsec:intro_presc_arcs}

Prescribed arc graphs were introduced in \cite{B-NV} 
to study the projections of grand arcs to witnesses. Kopreski studied these graphs more generally in \cite{Kop}. 
\begin{defn}
Let $S=S_{g,n}$ be a surface with genus $g$ and $n$ boundary components. (We shall not make a distinction between marked points and boundary circles.) Let $\Gamma$ be a graph whose vertices correspond to the $n$ boundary components on $S$. A simple arc $\alpha$ on $S$ is $\Gamma$\textit{-allowed} if $\alpha$ terminates on boundary components $e_{i}$ and $e_{j}$ such that $\left(e_{i},e_{j}\right)\in E\left(\Gamma\right)$. The \textit{$\Gamma$}-\textit{prescribed arc graph} $\mathcal{A}\left(S,\Gamma\right)$ for $S$ is the graph whose vertices correspond to isotopy classes of unoriented $\Gamma$-prescribed arcs, with an edge between two vertices if the corresponding isotopy classes have disjoint representatives. 
\end{defn}

As a special case of a prescribed arc graph, we define the \textit{arc graph} of a subsurface $Y$ of a (possibly infinite-type) surface $\Sigma$.

\begin{defn}
Let $Y\subset\Sigma$ be a subsurface; we allow $\partial Y$ to contain arcs. The \textit{arc graph} $\mathcal{A}\left(Y\right)$ is the graph whose vertices correspond to isotopy classes essential simple arcs such that each endpoint is either on $\partial Y$ or converges to an end of $Y$, with an edge between two vertices if the corresponding isotopy classes can be realised disjointly. Equivalently, $\mathcal{A}(Y) = \mathcal{A}(Y,\Gamma)$ where $\Gamma$ is the complete graph with loops whose vertices correspond to the union of the ends and the boundary components of $Y$.
\end{defn}

\begin{defn}
Let $\alpha\subset\Sigma$ be a simple arc, and suppose $\alpha\cap Y\neq\emptyset$. Then $\alpha\cap Y$ consists of a pairwise disjoint collection of simple arcs, some of which may not be essential. The \textit{projection of $\alpha$ to $\mc A(Y)$}, denoted $\pi_{Y}\left(\alpha\right)\subset\mathcal{A}(Y)$, is the set of isotopy classes of essential arcs in $\alpha\cap Y$. Similarly, we can define projections to general prescribed arc graphs of subsurfaces, assuming that the original arc intersects the subsurface in prescribed arcs.
\end{defn}

If an endpoint of an arc is on a boundary component $\alpha$ of $Y$ which is itself an arc, then one can homotope the endpoint of the arc along $\alpha$ to one of its ends. Thus, up to homotopy, we can and will assume that every arc in $Y$ has endpoints on boundary circles or ends of $Y$. 

In this work, we do not study prescribed arc graphs in general, but specifically, prescribed arc graphs that arise from projections of the grand arc graph. If $S$ is a not necessarily proper subsurface of $\Sigma$,  there is a natural prescribing graph $\Gamma=\Gamma_{S}$ on $S$ coming from the grand arc graph.  When $S$ is finite-type, $\Gamma_S$ is the graph whose vertices correspond to components of $\partial S$ with an edge between boundary components $b_{1}$ and $b_{2}$ if they separate different elements of the grand splitting. In particular, if some component cuts off only non-maximal ends of $\Sigma$, then there is no edge to it in $\Gamma_{S}$. When $S$ is an infinite-type subsurface of $\Sigma$, we extend this definition by including vertices corresponding to the ends of $\Sigma$ contained in $S$, with extra edges between vertices corresponding to different elements of the grand splitting and between an element of the grand splitting and a boundary component of $S$ separating other elements of the grand splitting. Observe that in this case,
\[
\mathcal{A}\left(S,\Gamma_{S}\right)=\left\{ a\in\mathcal{A}\left(S\right)\mid\exists\alpha\in\mathcal{G}\left(\Sigma\right)\text{ such that }\pi_{S}\left(\alpha\right)=\left\{ a\right\} \right\} .
\] In particular, $\mathcal{G}(S) = \mathcal{A}\left(\Sigma,\Gamma_{\Sigma}\right)$; this will be used in Section \ref{sec:BGI}.

From here on, whenever we refer to prescribed arc graphs, we assume that $\Gamma$ has a connnected component which is a complete multipartite graph with at least $3$ (non-empty) partite sets, and every other connected component of $\Gamma$ (if any) is an isolated vertex. Theorems 1.2 and 1.3 from \cite{Kop} thus imply that $\mathcal{A}\left(S,\Gamma\right)$ is connected, infinite-diameter and hyperbolic, when $S$ is finite-type. Moreover, if $\Sigma$ has at least 3 elements in the grand splitting, the prescribed arc graphs induced from $\mathcal{G}\left(\Sigma\right)$ on its witnesses, including $\mathcal{G}\left(\Sigma\right)$ itself, all satisfy this criterion, where the partite sets exactly correspond to elements of the grand splitting.

If $W\subset S$ is a witness of $\mathcal{A}\left(S,\Gamma\right)$, then there is a prescribed arc graph $\mathcal{A}\left(W,\Gamma_{W}\right)$ where the prescribing graph $\Gamma_W$ is essentially the restriction of $\Gamma$ to $W$. More precisely, $\Gamma_{W}$ is the graph whose vertices correspond to elements of $\partial W$ and with an edge between two vertices $v,w\in\Gamma_{W}$ if there is a prescribed arc in $\mathcal{A}\left(S,\Gamma\right)$ which intersects $W$ in a single arc connecting $v$ and $w$, considered as elements of $\partial W$. The usual subsurface projection map will not project all of $\mathcal{A}\left(S,\Gamma\right)$ into $\mathcal{A}\left(W,\Gamma_{W}\right)$, as there may be arcs in $\mathcal{A}\left(S,\Gamma\right)$ which intersect $W$ in multiple components.  However, $\mathcal{A}\left(W,\Gamma_{W}\right)$ coarsely approximates the arcs in $\mathcal{A}\left(S,\Gamma\right)$ which intersect $W$ in a single component, in the following sense.

\begin{lem}\label{lem:presc_arc_proj}
Suppose $x,y\in\mathcal{A}\left(S,\Gamma\right)$, and let $W$ be a compact witness for $\mathcal{A}\left(S,\Gamma\right)$ such that $x\cap W$ and $y\cap W$ are each connected. Then 
\[
d_{W}\left(\pi_{W}(x),\pi_{W}(y)\right)\leq d_{\mathcal{A}\left(S,\Gamma\right)}\left(x,y\right)\leq d_{W}\left(\pi_{W}(x),\pi_{W}(y)\right)+3,
\]
where $\pi_{W}(x)$ is the projection of $x$ to the prescribed arc graph $\mathcal{A}\left(W,\Gamma_{W}\right)$ of $W$.  Moreover, if $\Gamma$ has at least $4$ non-empty partite sets, then the second inequality can be strengthened to 
\[
d_{\mathcal{A}\left(S,\Gamma\right)}\left(x,y\right)\leq d_{W}\left(\pi_{W}(x),\pi_{W}(y)\right)+2.
\]
\end{lem}

Before we prove this, we shall prove a useful lemma.

\begin{lem}\label{lem:agree-inside-witness}
Let $W\subset S$ be a compact witness for $\mathcal{A}(S,\Gamma)$. Let $a,b\in \mathcal{A}(S,\Gamma)$ such that $a$ and $b$ intersect $W$ in a single proper arc, and $a\cap W = b\cap W$. Then $d_{\mathcal{A}(S,\Gamma)}(a,b)\leq 3$ if $\Gamma$ has at least 3 partite sets, and $d_{\mathcal{A}(S,\Gamma)}(a,b)\leq 2$ if $\Gamma$ has at least 4 partite sets.
\end{lem}

\begin{proof}
Since $a\cap W = b\cap W$, the arcs $a$ and $b$ must intersect the same boundary components of $W$. Suppose $\Gamma$ has exactly $3$ partite sets $S_1, S_2, S_3$, and suppose $a$ converges to $v_1\in S_1$ and $w_1 \in S_2$ and $b$ converges to $v_2 \in S_1$ and $w_2 \in S_2$. As $W$ separates the three partite sets, $W$ must have at least 3 boundary components. Pick an arc $c'$, respectively $d'$, in $W$ that joins a boundary component separating $S_3$ to the boundary component separating $S_1$, respectively $S_2$, that $a$, respectively $b$, intersects. Extend $c'$ and $d'$ to arcs $c,d\in \mathcal{A}(S,\Gamma)$ in such a way that $c$ agrees with $a$ in a component of $S\backslash W$, $d$ agrees with $b$ in a different component of $S\backslash W$, and $c,d$ are disjoint.  For example, one could extend $c'$ and $d'$ by the same arc to $S_3$. Then $a-c-d-b$ is a path of length $3$ in $\mathcal{A}(S,\Gamma)$, as desired.

Suppose now that $\Gamma$ has at least 4 partite sets, so that $W$ has at least 4 boundary components. Then $\mathcal{A}\left(W,\Gamma _W\right)$ is infinite-diameter, so there is an arc in $W$ connecting boundary components corresponding to partite sets other than the ones $a,b$ converge to. Any extension of such an arc to an arc in $\mathcal{A}(S,\Gamma)$ results in an arc disjoint both from $a,b$, proving the second bound.
\end{proof}

\begin{proof}[Proof of Lemma \ref{lem:presc_arc_proj}]
The first inequality is proved in \cite{Kop} and is true for any witness $W$. For the second inequality, fix any geodesic path in $\mathcal{A}\left(S,\Gamma\right)$ between $\pi_{W}(x)$ and $\pi_{W}(y)$, which are both single arcs by hypothesis. 
For each component $C$ of $\partial W$ which is not in $\partial S$, choose an arc $\alpha_C$ joining it to the vertex of $\Gamma$ representing the corresponding partite set that it separates.  Each vertex $v$ of the  geodesic path between $\pi_{W}(x)$ and $\pi_{W}(y)$ corresponds to an arc $v$ with endpoints on  components $C^+,C^-$ of $\partial W$.  Extend $v$ by concatenating with the arcs $\alpha_{C^+}$ and $\alpha_{C^-}$ to obtain a vertex of $\mathcal{A}\left(S,\Gamma\right)$. Without loss of generality, we can assume that the arcs are chosen in such a way that $\pi_{W}(x)$ extends to $x$. Then $\pi_{W}(y)$ extends to some arc $y'\in\mathcal{A}\left(S,\Gamma\right)$ which agrees with $x$ outside $W$. By Lemma \ref{lem:agree-inside-witness}, $d_{\mathcal{A}\left(S,\Gamma\right)}\left(y',y\right)\leq3$ if $\Gamma$ has 3 partite sets, and $d_{\mathcal{A}\left(S,\Gamma\right)}\left(y',y\right)\leq2$ if $\Gamma$ has at least 4 partite sets. Thus, we have constructed a path of length at most $d_{W}\left(\pi_{W}(x),\pi_{W}(y)\right)+3$ or $d_{W}\left(\pi_{W}(x),\pi_{W}(y)\right)+2$ between $x$ and $y$ in $\mathcal{A}\left(S,\Gamma\right)$.
\end{proof}

Lemma~\ref{lem:presc_arc_proj} immediately implies the following corollary.
\begin{cor}\label{cor:quasi_geod_projection}
For any $k\geq 1$ and $c\geq 0$, there exist $k'\geq 1$ and $c'\geq 0$ such that that following holds.  If $\left(x_{i}\right)_{i\in\mathbb{N}}$ be a $(k,c)$--quasi-geodesic sequence of grand arcs such that there exists a compact witness $W\subset\Sigma$ with $x_{i}\cap W$ connected for each $i\in\mathbb{N}$, then  $\left(\pi_{W}\left(x_{i}\right)\right)_{i\in\mathbb{N}}$ is a $(k',c')$--quasi-geodesic sequence in $\mathcal{A}\left(W,\Gamma_{W}\right)$.
\end{cor}

\section{Bicorn arcs in $\mathcal{G}(\Sigma)$} \label{sec:finite_bicorns}

Given two proper geodesics $a,b$ on $\Sigma$, their intersection $a\cap b$ is a discrete (possibly infinite) collection of points.  In this section, we will use these intersection points to build new arcs and investigate their properties.

Previously, Bar-Natan--Verberne and Kopreski used \textit{unicorn arcs} to study the geometry of the grand arc graph and prescribed arc graphs, respectively \cite{B-NV, Kop}. Unicorn arcs were defined in the finite-type setting by \cite{HPW} and extended to the infinite-type setting by \cite{F-G-M}. A unicorn arc between $a$ and $b$ is a simple arc of the form $x=a'\cup b'$, where $a', b'$ are the closures of components of $a\setminus\{P\}$ and $b\setminus \{P\}$ for a point $P\in a\cap b$.  The arc $x$ is only  simple when $a'\cap b'=\{P\}$, hence not every point of intersection of $a$ and $b$ will define a unicorn arc. We denote the set of unicorn arcs between $a$ and $b$ by $\mc U(a,b)$. Note that $a,b\in \mc U(a,b)$ and $\mc U(a,b)=\mc U(b,a)$. A unicorn arc in $\mc U(a,b)\setminus\{a,b\}$ is called a \textit{proper unicorn arc}. 

Bar-Natan and Verberne prove that for any $a,b\in \mc{G}(\Sigma)$, when $|\mc{S}(\Sigma)|\geq 3$, the 1-neighbourhood of $\mc{U}(a,b)$ is a connected, unparametrised quasi-geodesic set \cite[Proposition~4.5]{B-NV}. We will prove a similar result for the collection of \textit{bicorn arcs} between two grand arcs, which we define below.

\begin{defn}\label{def:bicorn}
A \textit{bicorn arc between grand arcs $a$ and $b$} is a grand arc of the type $x=a_{1}\cup b_{1}\cup a_{2}$, where $a_{1},a_{2}\subset a$ are the closures of the unbounded (non-empty) components of $a\setminus\{P,Q\}$ and $b_{1}\subset b$ is the closure of the bounded segment in $b\setminus \{P,Q\}$, where $P,Q\in a\cap b$.   As for unicorn arcs, this is only a simple arc if $(a_1\cup a_2)\cap b_1=\{P,Q\}$.  Let $\mathcal{B}\left(a,b\right)$ be the collection of bicorns arcs between $a$ and $b$. Note that $a\in\mathcal{B}\left(a,b\right)$,  but $b\notin \mc{B}(a,b)$.  If $x\in \mc{B}(a,b)\setminus \{a\}$, then $x$ is a \textit{proper bicorn arc}.
\end{defn} 

If $x_1$ is a compact subsegment of an arc $x$ and the endpoints of $x_1$ are $P,Q$, we write $x_1\equiv (PQ)_x$, meaning that $x_1$ is formed by traveling from the point $P$ to the point $Q$ along the arc $x$.  If $x_1$ is a subsegment of $x$ starting at an end of $x$ and ending at a point $P\in x$, then we write $x_1\equiv (\dots P)_x$.  In particular, if $x\in \mc B(a,b)$, then $x\equiv (\dots, P)_a\cup (PQ)_b\cup (Q\dots)_a$ for some points $P,Q\in a\cap b$.  Note that $\mathcal{B}\left(a,b\right)$ might be an infinite set; for example, see Figure \ref{econnected}.

\begin{figure}[h]
    \centering
    \def\svgwidth{10 cm}
\begingroup%
  \makeatletter%
  \providecommand\color[2][]{%
    \errmessage{(Inkscape) Color is used for the text in Inkscape, but the package 'color.sty' is not loaded}%
    \renewcommand\color[2][]{}%
  }%
  \providecommand\transparent[1]{%
    \errmessage{(Inkscape) Transparency is used (non-zero) for the text in Inkscape, but the package 'transparent.sty' is not loaded}%
    \renewcommand\transparent[1]{}%
  }%
  \providecommand\rotatebox[2]{#2}%
  \newcommand*\fsize{\dimexpr\f@size pt\relax}%
  \newcommand*\lineheight[1]{\fontsize{\fsize}{#1\fsize}\selectfont}%
  \ifx\svgwidth\undefined%
    \setlength{\unitlength}{323.63534041bp}%
    \ifx\svgscale\undefined%
      \relax%
    \else%
      \setlength{\unitlength}{\unitlength * \real{\svgscale}}%
    \fi%
  \else%
    \setlength{\unitlength}{\svgwidth}%
  \fi%
  \global\let\svgwidth\undefined%
  \global\let\svgscale\undefined%
  \makeatother%
  \begin{picture}(1,0.40245707)%
    \lineheight{1}%
    \setlength\tabcolsep{0pt}%
    \put(0,0){\includegraphics[width=\unitlength,page=1]{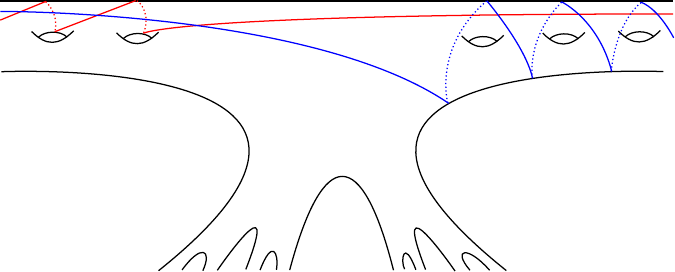}}%
    \put(0.38849746,0.3821695){\color[rgb]{1,0,0}\makebox(0,0)[lt]{\lineheight{1.25}\smash{\begin{tabular}[t]{l}$a$\end{tabular}}}}%
    \put(0.36673392,0.31949823){\color[rgb]{0,0,1}\makebox(0,0)[lt]{\lineheight{1.25}\smash{\begin{tabular}[t]{l}$b$\end{tabular}}}}%
  \end{picture}%
\endgroup%

    \captionsetup{margin=1.5 cm, justification=justified, singlelinecheck=false}
    \caption{Every bicorn arc between the grand arcs $a$ and $b$  intersects $b$ infinitely many times.}
    \label{fig:bicorns_disconnected}
\end{figure}

There are several reasons that bicorn arcs are better suited for this paper than unicorn arcs. First, in the context of the grand arc graph, when using unicorn arcs, one must be careful about the orientations on the two arcs so that the resulting unicorn arc converges to distinct ends. On the other hand, all bicorn arcs start and end like a single grand arc, so are automatically grand. The most important feature, however, is that bicorns between $a$ and $b$ only follow $b$ for a finite segment, as opposed to an infinite segment for unicorn arcs. This is the key property that allows us to generalise the definition of bicorn arcs in Section \ref{sec:bicorns_with_generic_geod} to the case when $b$ is an arbitrary geodesic, for example, a leaf of a geodesic lamination. 

For any grand arcs $a$ and $b$, we show in Corollary \ref{cor:general_unicorns_vs_bicorns} that the sets   
 $\mathcal{U}(a,b)$ and $\mc{B}(a,b)$ are within a finite Hausdorff distance of each other in $\mc{G}(\Sigma)$. We prove this via a sequence of lemmas. 

\begin{lem}\label{lem:unicorns_vs_bicorns}
Suppose $a,b\in\mathcal{G}\left(\Sigma\right)$ such that $\left|a\cap b\right|<\infty$. Then $d_{\mathrm{Haus}}^{\mathcal{G}}\left(\mathcal{B}\left(a,b\right),\mathcal{U}\left(a,b\right)\right)\leq2$.
\end{lem}

\begin{proof}
If $\left|a\cap b\right|\leq1$, then $\mathrm{diam}_{\mathcal{G}}\left(\mathcal{U}\left(a,b\right)\right)\leq2$ and $\mathcal{B}\left(a,b\right)=\left\{ a\right\}\subseteq \mc U(a,b)$, so the result holds. Suppose $|a\cap b|\geq 2$, and let $u=a_{1}\cup b_{1}\in\mathcal{U}\left(a,b\right)$, where $a_{1}\subset a$ and $b_{1}\subset b$ are sub-arcs intersecting at a point $P$. If $a\cap b_{1}=\left\{ P\right\} $, then in fact $a\cap u=\emptyset$ and $d_{\mathcal{G}}\left(a,u\right)\leq1$. Else, oriented along $a_{1}$, let $Q$ denote the last intersection of $a$ with $b_{1}$; this point $Q$ exists as $|a\cap b|\geq 2$ and finite. Then $x=a_{1}\cup\left(PQ\right)_{b}\cup\left(Q\dots\right)_{a}$ is a bicorn arc between $a,b$ disjoint from $u$, and so $\mc U(a,b)\subseteq N_1(\mc B(a,b))$. See Figure \ref{fig:bicorn_disjt_from_unicorn}.

\begin{figure}[h]
    \centering
    \includegraphics[width=10 cm]{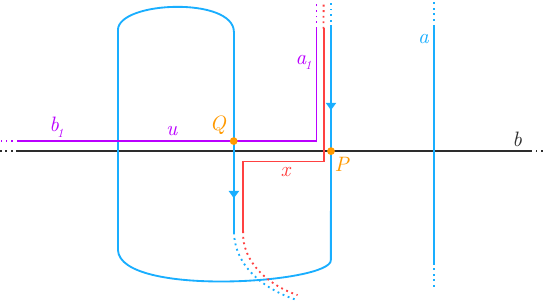}
    \caption{Constructing a bicorn arc disjoint from a given unicorn arc.}
    \label{fig:bicorn_disjt_from_unicorn}
\end{figure}

On the other hand, let $y=a_{1}'\cup b_{1}\cup a_{2}'$ be a bicorn arc between $a,b$. Orient $b$, and let $R$ denote the first point of $(a_1'\cup a_2')\cap b$ along $b$, possibly with $R=a_{1}'\cap b_{1}$; the point $R$ exists as $a\cap b$ is finite. Then $R$ defines a unicorn arc $y'\equiv(\dots R)_{b_1}\cup(R\dots)_a$ between $a$ and $b$ that is disjoint from $y$, and so $\mc B(a,b)\subseteq N_1(\mc U(a,b))$. See Figure \ref{fig:unicorn_disjt_from_bicorn}.
\end{proof}

\begin{figure}[h]
    \centering
    \includegraphics[width = 10 cm]{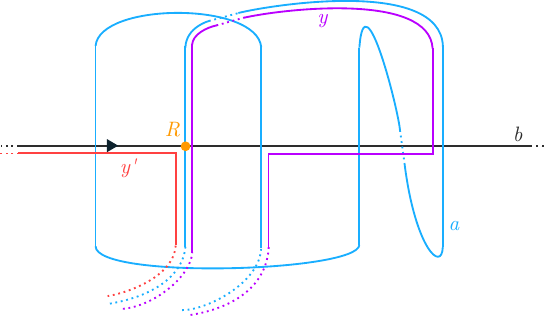}
    \caption{Constructing a unicorn arc disjoint from a given bicorn arc.}
    \label{fig:unicorn_disjt_from_bicorn}
\end{figure}

In these definitions, $(a,b)$ is an ordered pair, and a priori $\mathcal{B}(a,b)\neq \mathcal{B}(b,a)$. We show next that these two sets are not too far apart.

\begin{lem}\label{lem:bicorns_order_doesn't_matter}
If $a,b\in\mathcal{G}\left(\Sigma\right)$, then $d_{\mathrm{Haus}}^{\mathcal{G}}\left(\mathcal{B}\left(a,b\right),\mathcal{B}\left(b,a\right)\right)\leq10.$ In particular, $b\in N_{10}(\mc{B}(a,b))$.
\end{lem}

\begin{proof}
If $\#\left(a\cap b\right)\leq1$, then $\mathcal{B}\left(a,b\right)=\left\{ a\right\} $ and $\mathcal{B}\left(b,a\right)=\left\{ b\right\} $, so the result holds. Suppose $\#(a\cap b) >1$, and let $W$ be a compact witness of $\mathcal{G}\left(\Sigma\right)$ large enough such that both $\pi_{W}(a)$ and $\pi_{W}(b)$ are single arcs. Then by definition, every bicorn arc in $\mathcal{B}\left(a,b\right)$ or $\mathcal{B}\left(b,a\right)$ also projects to a single arc in $W$. 

If $a,b$ intersect outside $W$, then in particular they intersect a common component of $\Sigma\setminus W$. Then by replacing a subarc of $b\setminus W$ by a corresponding subarc of $a\setminus W$, we obtain an arc $b'\in\mathcal{G}\left(\Sigma\right)$ such that that $b'\cap W=b\cap W$ and $a$ and $b'$ are disjoint from each other outside $W$. Since $\pi_W(b)=\pi_W(b')$, we get that $\pi_W(\mc B(a,b))=\pi_W(\mc B(a,b'))$ and $\pi_W(\mc B(b,a))=\pi_W(\mc B(b',a))$. Thus Lemma~\ref{lem:presc_arc_proj} shows that $d_{\mathrm{Haus}}^{\mathcal{G}}\left(\mathcal{B}\left(a,b\right),\mathcal{B}\left(a,b'\right)\right)\leq3$ and  $d_{\mathrm{Haus}}^{\mathcal{G}}\left(\mathcal{B}\left(b,a\right),\mathcal{B}\left(b',a\right)\right)\leq3$. 

As $a,b'$ are disjoint outside $W$ and project to single arcs in $W$, the set $a\cap b'$ is finite. Thus Lemma \ref{lem:unicorns_vs_bicorns} implies that  $d_{\mathrm{Haus}}^{\mathcal{G}}\left(\mathcal{B}\left(a,b'\right),\mathcal{U}\left(a,b'\right)\right)\leq2$ and $d_{\mathrm{Haus}}^{\mathcal{G}}\left(\mathcal{B}\left(b',a\right),\mathcal{U}\left(b',a\right)\right)\leq2$. Since $\mathcal{U}\left(a,b'\right)=\mathcal{U}\left(b',a\right)$, we obtain 
\begin{align*}
d_{\mathrm{Haus}}^{\mathcal{G}}\left(\mathcal{B}\left(a,b\right),\mathcal{B}\left(b,a\right)\right) & \leq d_{\mathrm{Haus}}^{\mathcal{G}}\left(\mathcal{B}\left(a,b\right),\mathcal{B}\left(a,b'\right)\right)+d_{\mathrm{Haus}}^{\mathcal{G}}\left(\mathcal{B}\left(a,b'\right),\mathcal{B}\left(b',a\right)\right)\\ 
&+d_{\mathrm{Haus}}^{\mathcal{G}}\left(\mathcal{B}\left(b',a\right),\mathcal{B}\left(b,a\right)\right)\\
 & \leq3+d_{\mathrm{Haus}}^{\mathcal{G}}\left(\mathcal{B}\left(a,b'\right),\mathcal{U}\left(a,b'\right)\right)+d_{\mathrm{Haus}}^{\mathcal{G}}\left(\mathcal{U}\left(b',a\right),\mathcal{B}\left(b',a\right)\right)+3\\
 & \leq3+2+2+3=10.\qedhere
\end{align*}
\end{proof}

\begin{lem}\label{lem:finite_bicorn_subset}
If $a,b\in\mathcal{G}\left(\Sigma\right)$ and $x\in\mathcal{B}\left(a,b\right)$, then $\mathcal{B}\left(a,x\right)\subset\mathcal{B}\left(a,b\right)$. 
\end{lem}

\begin{proof}
Suppose $x=a_{1}\cup b_{1}\cup a_{2}$ is a bicorn arc. Any point of $a\cap x$ is a point of $a\cap b_{1}\subseteq a\cap b$, and so any bicorn arc between $a$ and $x$ is necessarily a bicorn arc between $a$ and $b$.
\end{proof}

We next show that changing the grand arcs by a finite distance only moves the bicorn set a finite Hausdorff distance. 

\begin{lem} \label{lem:bicorn_changing_first_arc}
If $a,a'\in\mathcal{G}\left(\Sigma\right)$ are disjoint and $b,b'\in\mathcal{G}\left(\Sigma\right)$ are disjoint, then
\[
d^{\mc{G}}_{\mathrm{Haus}}\left(\mathcal{B}\left(a,b\right),\mathcal{B}\left(a',b\right)\right)\leq 2 \quad \mathrm{and} \quad d_{\mathrm{Haus}}^{\mc G}(\mc B(a,b),\mc B(a,b'))\leq 12.
\]
\end{lem}

\begin{figure}[h]
    \centering
    \def\svgwidth{10 cm}
    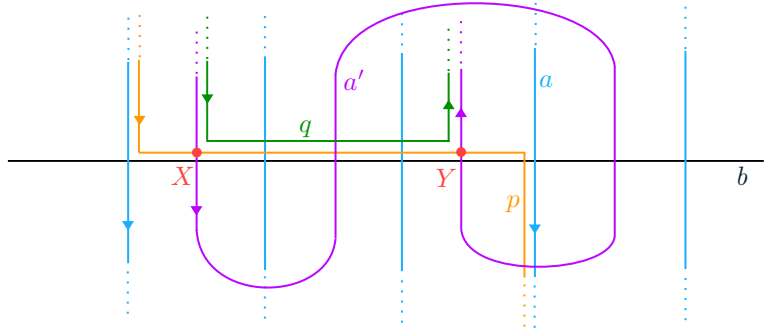
    \caption{Constructing a bicorn arc $q\in\mathcal{B}\left(a',b\right)$ intersecting $p$ in at most one point.}
    \label{fig:bicorns_first_change}
\end{figure}

\begin{proof}
We prove the first inequality.  This, along with Lemma~\ref{lem:bicorns_order_doesn't_matter}, proves the second inequality.  Let $p=a_{1}\cup b_{1}\cup a_{2}\in \mc B(a,b)$ be a bicorn arc. We will construct $q\in\mathcal{B}\left(a',b\right)$ such that $d_{\mathcal{G}}\left(p,q\right)\leq2$.  The analogous constrution with the roles of $a$ and $a'$ reversed will prove the other inclusion.

If $|a'\cap b_{1}|\leq1$, then setting $q=a'$ completes the proof by Lemma \ref{lem:grand_arc_dist_2}. If $|a'\cap b_1| \geq 2$, then pick an orientation on $a'$, and let $X,Y$ denote the first and last points of $a'\cap b_{1}$ along $a'$, respectively. Then $q=\left(\dots X\right)_{a'}\cup\left(XY\right)_{b_{1}}\cup\left(Y\dots\right)_{a'}$ is a bicorn arc in $\mathcal{B}\left(a',b\right)$ which intersects $p$ in at most one point; see Figure \ref{fig:bicorns_first_change}.  It follows that $d_{\mathcal{G}}\left(p,q\right)\leq2$ by Lemma~\ref{lem:grand_arc_dist_2}.
\end{proof}

There is an analogous statement for unicorn arcs, which is simpler as $\mc U(a,b)=\mc U(b,a)$.

\begin{lem} \label{lem:unicorns_changing_first_arc}
If $a,a'\in\mathcal{G}\left(\Sigma\right)$ are disjoint and $b\in\mathcal{G}\left(\Sigma\right)$, then \[d^{\mc{G}}_{\mathrm{Haus}}\left(\mathcal{U}\left(a,b\right),\mathcal{U}\left(a',b\right)\right)\leq1.\]
\end{lem}

\begin{proof}
Let $p=a_{1}\cup b_{1}=(\dots P)_{a}\cup (P\dots)_b$ be a unicorn arc in $\mathcal{U}\left(a,b\right)$. If $\mathcal{U}\left(a',b\right)= \{ a',b\} $, the result trivially holds, so suppose this is not the case.  We will construct a unicorn arc $q\in\mathcal{U}\left(a',b\right)$ disjoint from $p$. Any point of $a'\cap p$ is a point of $a'\cap b_{1}$. 
Let $Q\in a'\cap b_{1}$ be the closest point to $P$ along $b_{1}$ that defines a unicorn arc.  Such a point $Q$ exists as we have assumed that $\mathcal{U}\left(a',b\right) \neq \{ a',b\} $ By construction, $q=\left(\dots Q\right)_{a'}\cup\left(Q\dots\right)_{b_{1}}$ is a unicorn arc between $a'$ and $b$ that is disjoint from $p$. See Figure \ref{fig:unicorns_first_change}.
\end{proof}

\begin{figure}
    \centering
    \def\svgwidth{10 cm}
    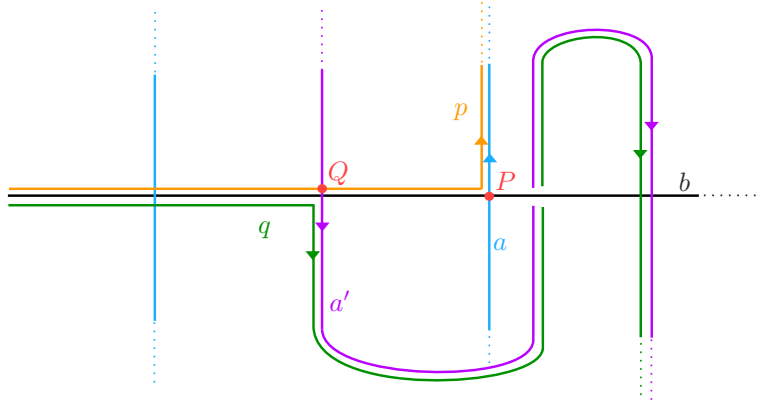
    \caption{Constructing a unicorn arc in $\mathcal{U}\left(a',b\right)$ disjoint from $p$.}
    \label{fig:unicorns_first_change}
\end{figure}

\begin{cor}
\label{cor:general_unicorns_vs_bicorns}
If $a,b\in\mathcal{\mathcal{G}}\left(\Sigma\right)$, then 
\[
d_{\mathrm{Haus}}^{\mathcal{G}}\left(\mathcal{B}\left(a,b\right),\mathcal{U}\left(a,b\right)\right)\leq11.
\]
\end{cor}

\begin{proof}
If $a,b$ intersect at most finitely many times, then this follows from Lemma \ref{lem:unicorns_vs_bicorns}. Now suppose $a\cap b$ is infinite. As in the proof of Lemma \ref{lem:bicorns_order_doesn't_matter}, we can construct a grand arc $a'$ which is disjoint from $b$ outside a compact witness (so that $a\cap b$ is finite), agrees with $a$ inside this compact witness, and $d_{\mathcal{G}}\left(a,a'\right)\leq3$. By Lemma \ref{lem:bicorn_changing_first_arc}, $d_{\mathrm{Haus}}^{\mathcal{G}}\left(\mathcal{B}\left(a,b\right),\mathcal{B}\left(a',b\right)\right)\leq2\cdot3=6.$
Moreover, by Lemma \ref{lem:unicorns_vs_bicorns} and the fact that $a'\cap b$ is finite, we have $d_{\mathrm{Haus}}^{\mathcal{G}}\left(\mathcal{B}\left(a',b\right),\mathcal{U}\left(a',b\right)\right)\leq2.$ Finally, $d_{\mathrm{Haus}}^{\mathcal{G}}\left(\mathcal{U}\left(a',b\right),\mathcal{U}\left(a,b\right)\right)\leq3$ by Lemma~\ref{lem:unicorns_changing_first_arc}. 
Combining these, we get 
\[
d_{\mathrm{Haus}}^{\mathcal{G}}\left(\mathcal{B}\left(a,b\right),\mathcal{U}\left(a,b\right)\right)\leq6+2+3=11.\qedhere
\]
\end{proof}

We conclude this section by showing that the subset of grand arcs in a uniform neighbourhood of $\mathcal{B}\left(a,b\right)$ is an unparametrised quasi-geodesic from $a$ to $b$. To see this, we invoke the Guessing Geodesics Lemma.

\begin{lem}[Guessing Geodesics Lemma, {\cite{M-S}}]
Suppose $G$ is a connected graph and for each $x,y\in V(G)$ there is an associated connected subgraph $A\left(a,b\right)$ with $a,b\in A\left(a,b\right)$. Suppose there is a uniform constant $h>0$ such that the following hold.
\begin{enumerate}
\item For any $a,b,c\in V(G)$, 
\[
A\left(a,b\right)\subset N_{h}\left(A\left(a,c\right)\cup A\left(c,b\right)\right).
\]
\item For any $a,b\in V(G)$ with $d\left(a,b\right)\leq1$, the diameter of $A\left(a,b\right)$ in $G$ is at most $h$. 
\end{enumerate}
Then $G$ is hyperbolic. Moreover, there is a uniform constant $m>0$ such that 
\[
d_{\mathrm{Haus}}\left(A\left(a,b\right),\left[a,b\right]\right)\leq m,
\]
where $\left[a,b\right]$ denotes any geodesic between $a$ and $b$ in $G$. 
\end{lem}

We will apply the Guessing Geodesics Lemma with 
\[
A(a,b):=N_{12}(\mathcal B(a,b))
\]
for grand arcs $a,b$. Note that this set contains both $a$ and $b$ and $N_1(\mc{U}(a,b))$ by Lemma~\ref{lem:bicorns_order_doesn't_matter} and Corollary~\ref{cor:general_unicorns_vs_bicorns}, respectively. As $N_1(\mc{U}(a,b))$ is connected \cite[Lemma~5.8]{B-NV}, the subgraph $A(a,b)$ is also connected.  We show  the first condition holds. 

\begin{lem}\label{lem:ApplyGG}
Suppose $\Sigma$ is an infinite-type surface with $|\mc{S}(\Sigma)|\geq 3$. If $a,b,c\in\mathcal{G}\left(\Sigma\right)$, then $A\left(a,b\right)\subseteq N_{38}\left(A\left(a,c\right)\cup A\left(b,c\right)\right)$. 
\end{lem}

\begin{proof}

For any $a,b,c\in\mathcal{G}\left(\Sigma\right)$, let $x\in N_{12}\left(\mathcal{B}\left(a,b\right)\right)$. By Corollary \ref{cor:general_unicorns_vs_bicorns}, there exists $x'\in\mathcal{U}\left(a,b\right)$ with $d_{\mathcal{G}}\left(x,x'\right)\leq11+12=23$. By \cite[Lemma~5.8]{B-NV}, there exists $y'\in N_{1}\left(\mathcal{U}\left(a,c\right)\right)\cup N_{1}\left(\mathcal{U}\left(b,c\right)\right)$ with $d_{\mathcal{G}}\left(x',y'\right)\leq3$. Without loss of generality, suppose $y'\in N_{1}\left(\mathcal{U}\left(a,c\right)\right)$. Then there exists $y\in\mathcal{U}\left(a,c\right)$ disjoint from $y'$, and hence $z\in\mathcal{B}\left(a,c\right)$ with $d_{\mathcal{G}}\left(y,z\right)\leq11$ by Corollary \ref{cor:general_unicorns_vs_bicorns}. By the triangle inequality, we have $d_{\mathcal{G}}\left(x,z\right)\leq23+3+1+11=38$, concluding the proof.
\end{proof}

We now give an alternate proof that the grand arc graph is hyperbolic if the surface has at least three maximal end types; this was originally proven in \cite{B-NV} with neighborhoods of unicorn paths playing the role of the subgraphs $A(a,b)$ in the Guessing Geodesics Lemma.

\begin{cor}\label{cor:unicorn_hausdorff_dist}
If $\Sigma$ is an infinite-type surface with $|\mc{S}(\Sigma)|\geq 3$, then the grand arc graph $\mathcal{G}\left(\Sigma\right)$ is hyperbolic. Moreover, there is a uniform constant $M_0>0$ such that for any $a,b\in\mathcal{G}\left(\Sigma\right)$ and any geodesic $\left[a,b\right]$ between $a$ and $b$ in $\mathcal{G}\left(\Sigma\right)$, we have 
\[
d_{\mathrm{Haus}}^{\mc G}\left(\mathcal{B}\left(a,b\right),\left[a,b\right]\right)\leq M_0.
\] 
\end{cor}

\begin{proof}
    We will show the Guessing Geodesics Lemma holds with $A(a,b)=N_{12}(\mathcal B(a,b))$ for grand arcs $a,b$ and $h=24$.  The first condition holds by Lemma~\ref{lem:ApplyGG}.  For the second condition, if $a\cap b=\emptyset$, then $\mathcal B(a,b)=\{a\}$, so $A(a,b)=N_{12}(\mathcal B(a,b))$ has diameter 24. 
\end{proof}



\section{Bounded Geodesic Image Theorem} \label{sec:BGI}

In this section, we prove a version of the bounded geodesic image theorem (BGI) for $\mathcal{G}\left(\Sigma\right)$.  Versions of bounded geodesic image theorems exist for the curve graph \cite{MM2}, the fine curve graph \cite{long-tan} and, more generally, hierarchically hyperbolic spaces \cite{Behrstock_2019}.


A bounded geodesic image theorem for the grand arc graph involves projections of grand arcs to  arc graphs of subsurfaces.  Given a fixed subsurface $Y\subseteq \Sigma$, it establishes a uniform upper bound on the diameter in $\mc A(Y)$ of the projection of a geodesic $g\subseteq \mc G(\Sigma)$ to $Y$.  However, 
such a theorem need not hold for \textit{every} subsurface. For example, if $Y$ is a witness, so that every vertex of $\mc G(\Sigma)$ intersects $Y$, then $\mathrm{diam}_{\mathcal{A}(Y)}(\pi_Y(g))$ may be arbitrarily large: a pseudo-Anosov homeomorphism supported on $Y$ is a loxodromic isometry of both on $\mathcal{A}(Y)$ and $\mc G(\Sigma)$. On the other hand, if $Y$ is disjoint from a witness of $G$, then the projections to $Y$ give no control over the distances in $\mc G(\Sigma)$. This motivates the following general definition.

\begin{defn}\label{def:partialwitness}
    Let $S$ be a surface, and let $G$ be a graph whose vertices correspond to isotopy classes of simple closed curves or simple proper arcs on $S$.  A \textit{partial witness} for $G$ is a subsurface $Y$ that is neither a witness for $G$ nor disjoint from a witness.
\end{defn}

\begin{example}
Suppose $S$ is a finite-type surface with complexity at least $5$ and $G=\mathcal{C}(S)$, the curve graph of $S$. The only witness for $\mathcal{C}(S)$ is $S$: every proper subsurface $Y$ has at least one component of $\partial Y$ which is an essential simple closed curve in $S$ that is disjoint from $Y$. Thus every proper subsurface is a partial witness, as each is neither $S$ nor disjoint from $S$.
\end{example}

\begin{example}
Suppose $S$ is a surface with a non-empty finite set $P$ of marked points or boundary components and $G=\mathcal{A}\left(S,P\right)$ is the relative arc graph. The witnesses for $\mathcal{A}\left(S,P\right)$ are exactly the subsurfaces which contain all elements of $P$. Thus the partial witnesses for $\mathcal{A}\left(S,P\right)$ are exactly the subsurfaces which contain a non-empty proper subset $Q\subset P$ of marked points. Indeed, these subsurfaces intersect every witness, as all witnesses contain the marked points in $Q$, but are not witnesses themselves, as they don't contain $P\setminus Q$. See Figure~\ref{fig:sep_curve_graph} (left). 
\end{example}


\begin{example}
Suppose $\Sigma$ is a closed surface of genus $g\geq 3$, and let $\mc{C}_{\mathrm{sep}}(S)$ denote the separating curve graph of $S$. This graph is not hyperbolic; see, for example, \cite{RV}. Figure~\ref{fig:sep_curve_graph} (right) shows two witnesses $W$ and $W'$ for $\mc{C}_{\mathrm{sep}}(S)$ that are disjoint and a partial witness $Y$ for $\mc{C}_{\mathrm{sep}}(S)$. 
\end{example}

\begin{figure}[h]
    \centering
    \begin{subfigure}
    \centering
    \includegraphics[width=6 cm]{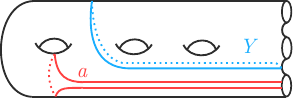}
    \end{subfigure}\hfill
    \begin{subfigure}
    \centering
    \includegraphics[width=6 cm]{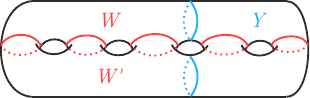}
    \end{subfigure}
    \captionsetup{margin=1.5 cm, justification=justified, singlelinecheck=false}
    \caption{Left: $Y$ is a partial witness for $\mathcal{A}(S,P)$.  Right: $W$ and $W'$ are disjoint subsurfaces of $S$, both of which are witnesses, while $Y$ is a partial witness for $S$.}
    \label{fig:sep_curve_graph}
\end{figure}

\begin{example}
Suppose $S$ is an infinite-type surface with $|\mc{S}(\Sigma)|\geq 4$ and $G=\mathcal{G}\left(\Sigma\right)$ is the grand arc graph. If $\gamma$ is a separating simple closed curve that separates two of the elements of the grand splitting from the rest, then $S\setminus\gamma$ consists of two infinite-type subsurfaces, both of which are partial witnesses. Similarly, if $W$ is a finite-type witness for $\mathcal{G}\left(\Sigma\right)$ on $\Sigma$ containing $\gamma$, then $W\setminus\gamma$ consists of two finite-type subsurfaces of $\Sigma$, both of which are partial witnesses. See Figure \ref{fig:partial_witness_grand}.

In general, a witness of $\mc{G}(\Sigma)$ must separate every element of the grand splitting, and a subsurface is disjoint from a witness if and only if it is contained in a clopen neighbourhood of some maximal end. Therefore a subsurface of $\Sigma$ is a partial witness for $\mc{G}(\Sigma)$ if and only if it there are two elements of the grand splitting that it does not separate and two other elements of the grand splitting that it either separates or limits to. Consequently, $\mc{G}(\Sigma)$ admits a \textit{compact} partial witness only when $|\mc{S}(\Sigma)|\geq 4$, though it may admit non-compact partial witnesses when $|\mc S(\Sigma)|=3$, as in Figure \ref{fig:partial_witness_grand}.
\end{example}

\begin{figure}[h]
    \centering
    \begin{subfigure}
    \centering
    \def\svgwidth{6 cm}
\begingroup%
  \makeatletter%
  \providecommand\color[2][]{%
    \errmessage{(Inkscape) Color is used for the text in Inkscape, but the package 'color.sty' is not loaded}%
    \renewcommand\color[2][]{}%
  }%
  \providecommand\transparent[1]{%
    \errmessage{(Inkscape) Transparency is used (non-zero) for the text in Inkscape, but the package 'transparent.sty' is not loaded}%
    \renewcommand\transparent[1]{}%
  }%
  \providecommand\rotatebox[2]{#2}%
  \newcommand*\fsize{\dimexpr\f@size pt\relax}%
  \newcommand*\lineheight[1]{\fontsize{\fsize}{#1\fsize}\selectfont}%
  \ifx\svgwidth\undefined%
    \setlength{\unitlength}{175.71646743bp}%
    \ifx\svgscale\undefined%
      \relax%
    \else%
      \setlength{\unitlength}{\unitlength * \real{\svgscale}}%
    \fi%
  \else%
    \setlength{\unitlength}{\svgwidth}%
  \fi%
  \global\let\svgwidth\undefined%
  \global\let\svgscale\undefined%
  \makeatother%
  \begin{picture}(1,0.64074248)%
    \lineheight{1}%
    \setlength\tabcolsep{0pt}%
    \put(0,0){\includegraphics[width=\unitlength,page=1]{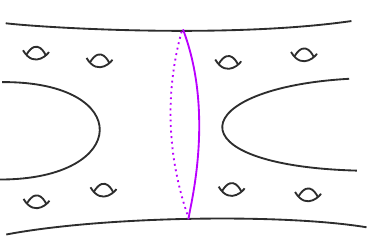}}%
    \put(0.4887255,0.59553553){\color[rgb]{0.72941176,0,1}\makebox(0,0)[lt]{\lineheight{1.25}\smash{\begin{tabular}[t]{l}\textit{$\gamma$}\end{tabular}}}}%
  \end{picture}%
\endgroup%

    \end{subfigure}\hfill
    \begin{subfigure}
    \centering
    \def\svgwidth{7 cm}
        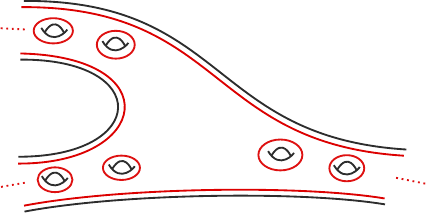
    \end{subfigure}
    \captionsetup{margin=1.5 cm, justification=justified, singlelinecheck=false}
    \caption{Left: $\gamma$ separates the surface into two partial witnesses. Right: cutting the surface along the collection of curves in red separates it into two partial witnesses, both limiting to all three maximal ends.}
    \label{fig:partial_witness_grand}
\end{figure}

We will primarily consider witnesses and partial witnesses for $\mc G(\Sigma)$ or $\mathcal{A}(S,\Gamma)$ for some $S\subseteq \Sigma$. When this is the case and no confusion is possible, we abuse terminology and do not specify ``for $\mc G(\Sigma)$" or ``for  $\mathcal{A}(S,\Gamma)$." 

\begin{thm}[BGI for the grand arc graph]\label{thm:BGI}
Let $\Sigma$ be an infinite-type surface with $3\leq |\mc{S}(\Sigma)|<\infty$.  For any $k\geq 1$ and $c\geq 0$, there exists a constant $B>0$ that satisfies the following. If $Y\subset\Sigma$ is a partial witness for $\mathcal{G}(\Sigma)$ and $g=\left(\gamma_{i}\right)_{i\in I}$ is a $(k,c)$--geodesic path in $\mathcal{G}\left(\Sigma\right)$ such that each $\gamma_{i}$ intersects $Y$ transversely, then 
\[
\mathrm{diam}_{Y}\left(g\right)\leq B.
\]
\end{thm}

Here and henceforth in this section, note that $d_Y$ refers to distance of the projections to the \textit{arc graph} of $Y$.



The following lemma is a key ingredient in the proof of Theorem \ref{thm:BGI} and describes the role of partial witnesses. For the rest of this section, we fix an infinite-type surface $\Sigma$ with $|\mc{S}(\Sigma)|\geq 3$.

\begin{lem}\label{lem:partial_witnesses}
Let $Y\subset \Sigma$ be a partial witness for $\mc{G}(\Sigma)$, and let $b\in \mc G(\Sigma)$. If there exists a grand arc $a$ that is disjoint from $Y$ and $d_{\mc G(\Sigma)}(a,b)\geq 3$, then the arc $b$ intersects $Y$ transversely. 
\end{lem}

\begin{proof}
Suppose toward a contradiction that both $a,b\in \mc{G}(\Sigma)$ are disjoint from $Y$.  First, notice that $a$ and $b$ must be in the same component $C$ of $\Sigma \setminus Y$, for if they were in different components, they would be disjoint, contradicting that $d_{\mc G(\Sigma)}(a,b)\geq 3$. Since $C$ is a subsurface disjoint from the partial witness $Y$, the component $C$ cannot be a witness. In particular, there exists $c\in \mc{G}(\Sigma)$ that is disjoint from $C$. Then $a,b$ are both disjoint from $c$, and so $d_{\mc{G}(\Sigma)}\left(a,b\right)\leq2$, which is also a contradiction. 
\end{proof}

The converse of Lemma~\ref{lem:partial_witnesses} does not hold in general.  For example, consider a surface $\Sigma$ and a witness $W$ of sufficiently low complexity so that the subgraph of arcs in $\mc{G}(\Sigma)$ which project to connected arcs in $W$ has diameter at most 2. Then there are no subarcs $a$ as in the statement of the lemma. The following lemma shows that this is the only exceptional case.

\begin{lem}\label{lem:detect_partial_witnesses}
Suppose $Y$ is a subsurface of $\Sigma$ that is disjoint from a witness. Then either there exist grand arcs $a$ and $b$ disjoint from $Y$ with $d_{\mc{G}(\Sigma)}(a,b)\geq 3$, or the only witness disjoint from $Y$ is (homeomorphic to) a pair of pants. 
\end{lem}

\begin{proof}
Since $Y$ is disjoint from a witness, some component, say $W$, of $\Sigma\setminus Y$ is a witness of $\mc{G}(\Sigma)$. Since $\mc{G}(\Sigma)$ is hyperbolic, it has no disjoint witnesses, and so $W$ is the only such component.  
If $W$ is not a pair of pants, then the prescribed arc graph $\mc{A}\left(W,\Gamma_W\right)$, where the prescribing graph $\Gamma_W$ is induced from $\mc{G}(\Sigma)$, has infinite diameter by \cite[Theorem~1.2]{Kop}.  Thus there exist prescribed arcs $a',b'$ connecting boundary components of $W$ that do not separate $Y$ with $d_{\mc{A}\left(W,\Gamma_W\right)}(a',b')\geq 3$. By Lemma \ref{lem:presc_arc_proj}, any extensions $a,b\in \mc{G}(\Sigma)$ of $a',b'$ satisfy $d_{\mc{G}(\Sigma)}(a,b)\geq 3$, and $a,b$ are disjoint from $Y$. This contradicts Lemma \ref{lem:partial_witnesses}.
\end{proof}

The proof of Theorem \ref{cor:BGI_for_presc} here closely follows the proof of the bounded geodesic image theorem for the curve graph in \cite{Webb}. The key ingredient in Webb's proof is to construct a uniform quality quasi-geodesic sequence connecting a boundary curve $\alpha$ of $Y$ with any curve $\beta$ intersecting $Y$, such that $d_{Y}\left(x,\beta\right)$ is uniformly bounded for \textit{each curve} $x$ in this quasi-geodesic except the two closest to $\alpha$. These quasi-geodesics provide shortcuts for the full geodesic in $\mathcal{C}\left(Y\right)$. Webb constructs these quasi-geodesics via certain surgeries on loops; see \cite{Jin} for an analogous proof using bicorns. 

\begin{lem}\label{lem:bicorns_project_well}
Suppose $Y$ is a partial witness for $\mc{G}(\Sigma)$ and $a$ is a grand arc disjoint from $Y$. Let $b\in \mc{G}(\Sigma)$ such that $d_{\mc{G}(\Sigma)}(a,b)\geq 3$. If a bicorn $x\in\mathcal{B}\left(b,a\right)$ intersects $Y$, then $d_{Y}\left(x,b\right)\leq2$.
\end{lem}

\begin{proof}
Let $x=b_{1}\cup a_{1}\cup b_{2}$ be a bicorn grand arc between $b$ and $a$ that intersects $Y$. Note that $x\cap b= a_{1}\cap b$. By Lemma \ref{lem:partial_witnesses}, $b$ must intersect $Y$ transversely, so $d_{Y}\left(x,b\right)$ is well-defined. While the arc $a$, and hence $a_1$, is disjoint from $Y$, it is possible that there are intersections of $a_1\cap b$ that can be isotoped into $Y$: this is the case only if there are subsegments of $a_1, b$, and a boundary component of $Y$ that bound a disk; see Figure~\ref{fig:sub-arc}. 

Any arc $\beta\in b\cap Y$ intersecting $x$ must cut through this disk and hence, intersect $\partial Y$ transversely. If there is a second point of $x\cap\beta$, the corresponding sub-arc of $b$ must again cut through this disk and, in particular, $\partial Y$. Thus, a single component of $b\cap Y$ may intersect $x$ in either zero points, one point, or two points with opposite orientations. In the first two cases, $d_{Y}\left(x,b\right)\leq2$, applying Lemma~\ref{lem:grand_arc_dist_2} in the second case. In the third case, there must be two consecutive points of $x\cap b$ along $x$ with opposite orientations along $b$, and these point define a bicorn arc between $x$ and $b$ in $Y$ that is disjoint from both $x$ and $b$.  Thus $d_{Y}\left(x,b\right)\leq2$. 
\end{proof}

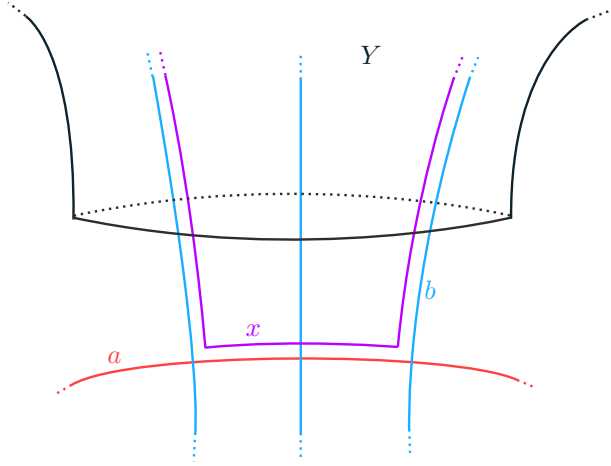
\begin{figure}
    \centering
    \def\svgwidth{8 cm}
\begingroup%
  \makeatletter%
  \providecommand\color[2][]{%
    \errmessage{(Inkscape) Color is used for the text in Inkscape, but the package 'color.sty' is not loaded}%
    \renewcommand\color[2][]{}%
  }%
  \providecommand\transparent[1]{%
    \errmessage{(Inkscape) Transparency is used (non-zero) for the text in Inkscape, but the package 'transparent.sty' is not loaded}%
    \renewcommand\transparent[1]{}%
  }%
  \providecommand\rotatebox[2]{#2}%
  \newcommand*\fsize{\dimexpr\f@size pt\relax}%
  \newcommand*\lineheight[1]{\fontsize{\fsize}{#1\fsize}\selectfont}%
  \ifx\svgwidth\undefined%
    \setlength{\unitlength}{224.8012729bp}%
    \ifx\svgscale\undefined%
      \relax%
    \else%
      \setlength{\unitlength}{\unitlength * \real{\svgscale}}%
    \fi%
  \else%
    \setlength{\unitlength}{\svgwidth}%
  \fi%
  \global\let\svgwidth\undefined%
  \global\let\svgscale\undefined%
  \makeatother%
  \begin{picture}(1,0.75927975)%
    \lineheight{1}%
    \setlength\tabcolsep{0pt}%
    \put(0,0){\includegraphics[width=\unitlength,page=1]{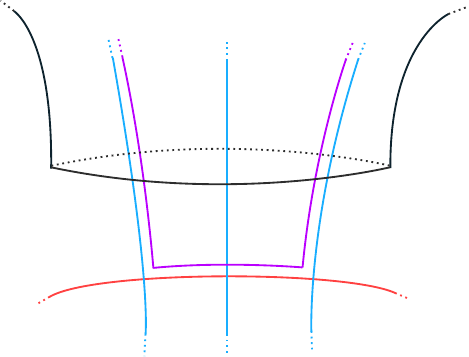}}%
    \put(0.58385898,0.65748428){\color[rgb]{0.05490196,0.1372549,0.18039216}\makebox(0,0)[lt]{\lineheight{1.25}\smash{\begin{tabular}[t]{l}\textit{$Y$}\end{tabular}}}}%
    \put(0.16549873,0.16282727){\color[rgb]{1,0.25490196,0.25490196}\makebox(0,0)[lt]{\lineheight{1.25}\smash{\begin{tabular}[t]{l}\textit{$a$}\end{tabular}}}}%
    \put(0.68940283,0.26837102){\color[rgb]{0.09803922,0.68235294,1}\makebox(0,0)[lt]{\lineheight{1.25}\smash{\begin{tabular}[t]{l}\textit{$b$}\end{tabular}}}}%
    \put(0.39438884,0.20733371){\color[rgb]{0.72941176,0,1}\makebox(0,0)[lt]{\lineheight{1.25}\smash{\begin{tabular}[t]{l}\textit{$x$}\end{tabular}}}}%
  \end{picture}%
\endgroup%

    \captionsetup{margin=1.5 cm, justification=justified, singlelinecheck=false}
    \caption{The bicorn arc $x$ and $\partial Y$ bound a disk, and so $x$ may have essential intersections with arcs of $b\cap Y$.}
    \label{fig:sub-arc}
\end{figure}

\begin{cor}\label{cor:quasigeod_project_for_BGI}
There exists a constant $D>0$, depending only on the hyperbolicity constant $\delta$ of $\mathcal{G}\left(\Sigma\right)$, such that for any partial witness $Y$, any $a\in\mathcal{G}\left(\Sigma\right)$ disjoint from $Y$ and any geodesic $a=\gamma_{0},\gamma_{1},\dots,\gamma_{n}=b$ in $\mathcal{G}\left(\Sigma\right)$ with $n\geq3$, we have $d_{Y}\left(\gamma_{i},b\right)\leq D$ for all $i\geq3$.
\end{cor}

\begin{proof}
The proof here is similar to the proof of \cite[Lemma~3.1]{Webb}, with the grand arc $a$ playing the role of ``a component $\alpha$ of $\partial Y$''. By Corollary~\ref{cor:unicorn_hausdorff_dist}, there is a uniform constant $M_0$ such that  
each $\gamma_i$ is at distance at most $M_0$ from some $x_i \in \mc{B}(b,a)$. If $x_i$ intersects $Y$, then Lemma~\ref{lem:bicorns_project_well} implies that $d_Y(x_i,b)\leq 2$, and so $d_Y(\gamma_i,b)\leq 2 + M_0$. On the other hand, if $x_i$ is disjoint from $Y$, then $d_{\mc{G}(\Sigma)}(a,x_i)\leq 2$ by Lemma~\ref{lem:partial_witnesses}. 

By the discussion in Section \ref{sec:finite_bicorns}, the set $N_{12}(\mc{B}(b,a))$ is connected and contains $a$ and $b$. Thus there exists a bicorn arc $x'_i\in\mc{B}(b,a)$ such that $3\leq d_\mc{G}(x_i',a)\leq 12$. Then $d_Y(x_i',b)\leq 2$ by Lemma~\ref{lem:bicorns_project_well}, and  $d_\mc{G}(x_i,x_i')\leq 14$. Therefore, $d_Y(\gamma_i,b)\leq M_0+12+d_Y(x_i',b)\leq M_0+14$. Setting $D=M_0+14$ completes the proof.
\end{proof}

We are now ready to prove the bounded geodesic image theorem. The proof closely follows \cite[Theorem 3.2]{Webb}, modified to handle quasi-geodesic paths.

\begin{proof}[Proof of Theorem \ref{thm:BGI}]
Let $\delta$ be the hyperbolicity constant of $\mc G(\Sigma)$, let $m_0=m_0(\delta, k,c)$ be the constant from Lemma~\ref{lem:Morse_lemma}, and let $D$ be the constant from Corollary~\ref{cor:quasigeod_project_for_BGI}.  Set $B=D+(m_0+\delta)+\left(k(2m_0+2\delta+6)+c\right)+(m_0+\delta)+D.$ Fix $i<j$.  We will show that $d_{Y}\left(\gamma_{i},\gamma_{j}\right)\leq B$. 
Fix a grand arc $a$ disjoint from $Y$, and let $I=N_{m_0+\delta+2}\left(a\right)\cap g$, where $N_{m_0+\delta+2}(a)$ denotes the $\left(m_0+\delta+2\right)$--neighbourhood of $a$ in $\mathcal{G}\left(\Sigma\right)$. Let $\gamma_{i'}$ and $\gamma_{j'}$ be the first and last point, respectively, of $g$ that is contained in $I$, respectively.  Then the subpath  $g'=\left(\gamma_{i'},\dots,\gamma_{j'}\right)$ has length at most $k(2m_0+2\delta+4)+c$. Let $P$ be a geodesic in $\mathcal{G}\left(\Sigma\right)$ from $a$ to $\gamma_{i}$ and $Q$ be a geodesic from $a$ to $\gamma_{j}$. Let $i''=\max\{i,i'-1\}$ and $j''=\min\{j,j'+1\}$. It is immediate that if grand arcs $\alpha,\beta$ both intersect $Y$ transversely and $\alpha\cap \beta=\emptyset$, then $d_Y(\alpha,\beta)\leq 1$.  From this and the fact that $g$ is a path, it follows that $d_Y(\gamma_{i''},\gamma_{j''})\leq k(2m_0+2\delta +6)+c$.

Since (quasi-)geodesic triangles are slim, either $\gamma_{i''}$ is $(m_0+\delta)$-close to $P$ and $\gamma_{j''}$ is $(m_0+\delta)$-close to $Q$, or there exist adjacent vertices of $g-g'$ one of which is $(m_0+\delta)$-close to $P$ and the other is $(m_0+\delta)$-close to $Q$. By the triangle inequality, every grand arc on either of these geodesics of length $m_0+\delta$ is distance at least 3 from $a$, and hence intersects $Y$ transversely by Lemma~\ref{lem:partial_witnesses}.  Thus each of these geodesics projects to a path of length at most $m_0+\delta$ in $\mc A(Y)$.  Combining this with Corollary~\ref{cor:quasigeod_project_for_BGI} and the triangle inequality, we have that \[d_{Y}\left(\gamma_{i},\gamma_{j}\right) \leq D+(m_0+\delta)+\left(k(2m_0+2\delta+6)+c\right)+(m_0+\delta)+D =B.\qedhere\] 
\end{proof}
%


The bounded geodesic image theorem can be used to show that the boundary of the grand arc graph is not compact.
\begin{cor}\label{cor:non-compactness}
If $\Sigma$ is an infinite-type surface with $4\leq |\mc S(\Sigma)|<\infty$, then $\partial\mathcal{G}\left(\Sigma\right)$ is not compact. 
\end{cor}

\begin{proof}
Let $a$ be a grand arc, let $W$ be a compact witness large enough such that $a\cap W$ is a single arc, and let $Y=W\setminus a$. Such a witness $W$ exists by Lemma~\ref{lem:single_arc_witness}. Up to choosing a larger witness $W$, we may assume that $Y\not\cong S_{0,3}$. By construction, $Y$ is not a witness, as $Y$ is disjoint from the grand arc $a$. If $b$ is any grand arc disjoint from $Y$, then $b\cap W = a\cap W$, and so $d_\mathcal{G}(a,b)\leq 2$ by Lemma \ref{lem:presc_arc_proj}. Thus any grand arc $b$ with $d_\mathcal{G}(a,b)\geq 3$ must intersect $Y$ transversely. Since $|\mc{S}(\Sigma)|\geq 4$, no witness of $\Sigma$ is homeomorphic to a pair of pants. By Lemma~\ref{lem:detect_partial_witnesses}, we conclude that $Y$ is a partial witness. 

Let $p\in\partial\mathcal{G}\left(\Sigma\right)$, fix a grand arc $x\in\mathcal{G}\left(\Sigma\right)$ that intersects $Y$ transversely, and let $\phi\in\mathrm{Mod}\left(Y\right)$ be pseudo-Anosov. We will show that the sequence $\left(\phi^{n}(p)\right)_{n\in\mathbb{N}}\subseteq \partial\mc G(\Sigma)$ does not have a convergent subsequence. 

For contradiction, suppose some subsequence $\left(\phi^{n_{k}}(p)\right)_{k\in\mathbb{N}}$ converges in $\partial\mathcal{G}\left(\Sigma\right)$. Then for any $R>0$, we can find a $k_{0}>0$ such that for every $k>k_{0}$, we have 
\begin{equation}\label{eqn:Gprod}
\left(\phi^{n_{k_{0}}}(x),\phi^{n_{k}}(x)\right)_{a}>R.
\end{equation}
Since $\phi$ is pseudo-Anosov on $Y$, $\phi$ is loxodromic with respect to the action on $\mathcal{A}(Y)$, so we can choose $k\gg k_{0}$ such that 
\[
d_{Y}\left(\pi_{Y}\left(\phi^{n_{k_{0}}}(x)\right),\pi_{Y}\left(\phi^{n_{k}}(x)\right)\right)>B,
\]
where $B$ is the constant from Theorem \ref{thm:BGI}. Then by Theorem \ref{thm:BGI}, every geodesics in $\mathcal{G}\left(\Sigma\right)$ from $\phi^{n_{k_{0}}}(x)$ to $\phi^{n_{k}}(x)$ must contain a grand arc $\alpha$ disjoint from $Y$. But as $W$ is a witness and $Y=W\setminus a$, it must be the case that $\pi_{W}\left(\alpha\right)=a$. There exists a prescribed arc in $W$ disjoint from $\pi_W(a)$ and connecting two boundary components of $W$ that separate elements of the grand splitting of $\Sigma$ to which $a$, and hence $\alpha$, does not converge. Any extension of this prescribed arc to a grand arc is disjoint from both $a$ and $\alpha$. Consequently, $d_{\mathcal{G}}\left(a,\alpha\right)\leq2$.

Since $\alpha$ lies on a geodesic from $\phi^{n_{k_{0}}}(x)$ to $\phi^{n_{k}}(x)$, we have
\begin{align*}
\left(\phi^{n_{k_{0}}}(x),\phi^{n_{k}}(x)\right)_{a} & =\frac{1}{2}\left(d_{\mathcal{G}}\left(\phi^{n_{k_{0}}}(x),a\right)+d_{\mathcal{G}}\left(\phi^{n_{k}}(x),a\right)-d_{\mathcal{G}}\left(\phi^{n_{k_{0}}}(x),\phi^{n_{k}}(x)\right)\right)\\
 & \leq\frac{1}{2}\left(d_{\mathcal{G}}\left(\phi^{n_{k_{0}}}(x),\alpha\right)+d_{\mathcal{G}}\left(\phi^{n_{k}}(x),\alpha\right)-d_{\mathcal{G}}\left(\phi^{n_{k_{0}}}(x),\phi^{n_{k}}(x)\right)\right)+2\\
 & =\frac12 (0) +2 = 2<R,
\end{align*}
which contradicts \eqref{eqn:Gprod}.
\end{proof}

The proof of the corollary follows analogous arguments of Wright \cite{Wright} for the curve graph of a finite-type surface, and Long--Tan \cite{long-tan} for the fine curve graph of a finite-type surface.  Both Wright and Long--Tan also prove that the boundary is \textit{linearly connected}.  It is likely that an analogous result hold for $\partial \mc G(\Sigma)$, though the proofs would require some modification. 

\begin{question}
    If $\Sigma$ is an infinite-type surface with sufficiently many elements in the grand splitting, is $\partial\mc{G}(\Sigma)$ linearly connected?
\end{question}

It will be useful in this paper to have a version of the bounded geodesic image theorem for graphs other than the grand arc graph.  We first give a version for prescribed arc graphs.  

For the remainder of the section, let $\Sigma$ be an infinite surface with at least 4 elements in the grand splitting, let $W$ be a compact witness subsurface of $\Sigma$, and let $\Gamma=\Gamma_W$ be the prescribing graph on $W$ induced by the grand arc graph.

\begin{cor}[BGI for prescribed arc graphs]
\label{cor:BGI_for_presc} There exists $B>0$ such that for any compact partial witness $Y\subset W$ of $\mathcal{A}\left(W,\Gamma_W\right)$ and any geodesic $g=\left(\gamma_{i}\right)_{i\in I}$ in $\mathcal{A}\left(W,\Gamma_W\right)$ such that each $\gamma_{i}$ intersects $Y$ transversely, we have 
\[
\mathrm{diam}_{\mathcal{A}(Y)}\left(g\right)\leq B.
\]
\end{cor}

\begin{proof}
Since $W$ is a witness of $\Sigma$ and $\Sigma$ has at least 4 elements in the grand splitting, $\mc{A}(W,\Gamma_W)$ is infinte-diameter, hyperbolic and admits partial witnesses. Let $g=\left(\gamma_{i}\right)_{i\in I}$ be any geodesic in $\mathcal{A}\left(W,\Gamma_W\right)$. Extend each $\gamma_{i}$ to a grand arc $\gamma_i'$ by choosing arcs in $\Sigma\setminus W$ converging to the corresponding maximal ends that are disjoint from each other. Then $\left(\gamma_{i}'\right)_{i\in I}\subset\mathcal{G}\left(\Sigma\right)$ a $(1,2)$-quasigeodesic by Lemma \ref{lem:presc_arc_proj}. The result then follows from Theorem \ref{thm:BGI} applied to $(1,2)$--quasigeodesics  in $\mc G(\Sigma)$ and the fact that $\pi_{Y}\left(\gamma_{i}'\right)=\pi_{Y}\left(\gamma_{i}\right)$ for every $i\in I$. 
\end{proof}

In Section~\ref{sec:GEL_boundary}, we will also need a version of the bounded geodesic image theorem for the relative arc graph $\mc{A}(W,P)$, where $P$ is the collection of boundary components of $W$ that separate elements of the grand splitting.  It follows from \cite[Corollary 2.2]{kop-asdim} that $\mc{A}(W,P)$ quasi-isometrically embeds into $\mc{A}(W,\Gamma_W)$, so quasi-geodesics in $\mc{A}(W,P)$ are also quasi-geodesics in $\mc{A}(W,\Gamma_W)$. Applying Corollary~\ref{cor:BGI_for_presc}, we obtain:

\begin{cor}[BGI for relative arc graphs]\label{cor:BGI_for_rel_arcs}
    There exists $B>0$ such that for any compact partial witness $Y\subset W$ of $\mc{A}(W,P)$, and any geodesic $g=\left(\gamma_{i}\right)_{i\in I}$ in $\mathcal{A}\left(W,P\right)$ such that each $\gamma_{i}$ intersects $Y$ transversely, we have 
\[
\mathrm{diam}_{\mathcal{A}(Y)}\left(g\right)\leq B.
\]
\end{cor}

The bounded geodesic image theorem can also be used to study dead ends in the grand arc graph.  If $X$ is a hyperbolic graph, then a vertex $b\in X$ is a \textit{dead end} with respect to a vertex $a$ if no geodesic in $X$ joining $a$ to $b$ can be extended past $b$ to a longer geodesic.  Birman and Menasco showed that the curve graph of a finite-type surface does contain dead ends, but that no non-separating curve is a dead end with respect to any other vertex of the curve graph \cite{Bir-Men}.  One can prove an analogous result for hyperbolic prescribed arc graphs and relative arc graphs which admit partial witnesses: any dead end is necessarily a separating arc. In particular, since grand arcs are never separating, the bounded geodesic image theorem also implies that the grand arc graph does not contain dead ends.



For the rest of this section, retaining the terminology and assumptions before Corollary~\ref{cor:BGI_for_presc}, let $\Gamma$ be the prescribing graph induced either by the grand arc graph or by the relative arc graph on $W$. 

\begin{lem}[No dead ends]
\label{lem:no-dead-ends}Suppose $W$ has complexity at least $5$. Fix a basepoint $o\in\mathcal{A}\left(W,\Gamma\right)$, and let $a\in\mathcal{A}\left(W,\Gamma\right)$ be a non-separating arc at distance $r$ from $o$. Then there exists $b\in\mathcal{A}\left(W,\Gamma\right)$ disjoint from $a$, at distance $r+1$ from $o$. In particular, any geodesic from $o$ to $a$ can be extended to a geodesic from $o$ to $b$. 
\end{lem}

\begin{proof}
Since $a$ is non-separating, the compact subsurface $Y=W\setminus a$ is a partial witness, whose complexity is at least $4$, and so the arc graph $\mathcal{A}(Y)$ is infinite-diameter. Pick a prescribed arc $b$ such that $b\cap Y$ is a single arc and $d_{\mathcal{A}(Y)}\left(o,b\right)>B$, where $B$ is the constant in Corollary \ref{cor:BGI_for_presc} or \ref{cor:BGI_for_rel_arcs}, depending on whether $\Gamma$ is the prescribing graph coming from the grand arc graph or the relative arc graph of $W$.  By Corollary \ref{cor:BGI_for_presc} or \ref{cor:BGI_for_rel_arcs}, any geodesic in $\mathcal{A}\left(S,\Gamma\right)$ between $o$ and $b$ must pass through a prescribed arc $a'$ that is disjoint from $Y$. Since the only arc disjoint from $Y$ is $a$, the arc $a$ lies on any geodesic from $o$ to $b$, and the conclusion follows.
\end{proof}

The fact that non-separating arcs are never dead ends can be used to show that quasi-geodesics in either the grand arc or prescribing arc graph can be extended to a uniform quality quasi-geodesic.  

\begin{cor}
\label{cor:extending-quasigeod} Any finite $(k,c)$--quasi-geodesic path $\left(x_{1},\dots,x_{m}\right)\subset \mc{A}(W,\Gamma)$ can be extended to an infinite $(\max\{k,2\},c)$--quasi-geodesic path $\left(x_{n}\right)_{n\in\mathbb{N}}$.
\end{cor}

\begin{proof}
We prove the result by induction.  Suppose we have extended the finite quasi-geodesic segment $(x_1,\dots, x_m)$ to a quasi-geodesic sequence $(x_1,\dots, x_n)$ for some $n\geq m$.  If $x_n$ is non-separating, let $x_n'=x_n$.  Otherwise, $x_n$ is separating, in which case there exists a non-separating vertex $x_n'$ disjoint from $x_n$. Lemma \ref{lem:no-dead-ends} then produces a prescribed arc $x_{n+1}$ with $d(x_1,x_{n+1})=d(x_1,x_n')+1$, so any geodesic from $x_{1}$ to $x_{n}'$ can be extended to a geodesic from $x_1$ to $x_{n+1}$. It is straightforward to check that $\left(x_{1},\dots,x_n,x_{n+1}\right)$ is a $(\max\{k,2\},c)$--quasi-geodesic. 
\end{proof}

\begin{rem}
It is likely that analogous bounded geodesic image theorems  hold for other examples hyperbolic graphs of arcs or curves. Suppose $G$ is a hyperbolic complex of arcs/curves, where edges are defined by disjointness, which does not have disjoint witnesses,  admits partial witnesses, and has a well-defined projection to a related complex of arcs/curves on subsurfaces. Then any sufficiently long geodesic in $G$ whose vertices all intersect a fixed partial witness should have a uniformly bounded projection to this partial witness. The key ingredient in proving such a theorem would involve constructing suitable quasi-geodesics between vertices of the graph, similar to the $\left(\alpha,\beta\right)$-loops in \cite{Webb} or the bicorn paths above.
\end{rem}

\section{Topology on the space of laminations}\label{sec:topology}



\subsection{Geodesic Laminations}

Fix an infinite-type surface $\Sigma$.  Whenever we consider an infinite-type surface, we will assume that we have fixed a complete, Nielsen-convex hyperbolic metric on $\Sigma$. In particular, we may choose a pants decomposition of the surface and a metric where the length of every seam is 1 and every curve in the pants decomposition has the same length. With respect to this metric, the universal cover of $\Sigma$ is $\mathbb H^2$, and we will also assume we have fixed a universal covering map $\pi\colon\mathbb{H}^2 \to \Sigma$. 

A \emph{geodesic lamination} on $\Sigma$ is a collection of pairwise disjoint geodesics on $\Sigma$ which is closed as a subspace of $\Sigma$.  Each geodesic is called a \textit{leaf} of the lamination.  It is a useful fact that, since we have fixed a Nielsen-convex metric on $\Sigma$, any geodesic lamination on $\Sigma$ is nowhere dense  \cite[Theorem~1.1]{Saric}.
A lamination is \emph{minimal} if it has no proper non-empty sublaminations, or equivalently, if every leaf is dense in the lamination. A leaf $l$ of a lamination $L$ is \textit{isolated} if every point $x\in l$ has a neighborhood $U_x$ such that $U_x\cap L$ contains a single arc, namely an arc of $l$.  A leaf $l$ of a lamination $L$ is a \textit{boundary leaf} if for every point $x\in l$, there exists a half-disk neighborhood $V_x$ of $x$ with diameter lying on $l$ such that $L\cap V_x$ is exactly the diameter of $V_x$. That is to say, there are not leaves in $L$ accumulating to $l$ from both sides of $l$. We write $L\pitchfork L'$ for the set of transverse intersections of laminations $L$ and $L'$. A geodesic $\gamma\colon \mathbb R\to \Sigma$ is \textit{asymptotic} to a lamination $L$ if there exists a half-leaf $l\in L$ such that either $\gamma|_{[0,\infty)}$ or $\gamma|_{(-\infty,0]}$ and $l$ cobound a half-strip. 

For intuition, it is useful to think of a minimal lamination  as a train track with a Cantor set cross-section. To see this, first note that any minimal lamination is either a single geodesic or has uncountably many leaves \cite[Lemma~4.2.2]{CEG06}. Indeed, suppose $L$ is a minimal lamination with least two leaves. Let $l\in L$ be a leaf, and note that $l$ cannot be a proper geodesic, else $\{l\}$ would be a sublamination of $L$. Consider a segment $a$ intersecting $l$ transversely. Since $l$ is not isolated, the points of $a\cap l$ are all accumulation points in $L$. Identifying $a$ with an interval, minimality of $L$ implies that $a\cap\overline{l}$ is a closed perfect subset of $a$.  
Since $L$ is nowhere dense in $\Sigma$, the subset $a\cap \overline l$ is homeomorphic to a Cantor set. Note that $a\cap l$ is countable, but $a\cap \overline{l}$ is the Cantor set. We obtain the train track by tracing $a$ along the leaf $l$.


The following definition was suggested to the authors by Kasra Rafi.

\begin{defn} \label{def:support}
Let $L$ be a geodesic lamination on $\Sigma$. The complement $\Sigma\setminus L$ is open and hence a union of subsurfaces of $\Sigma$. Let $\bigcup_{i\in I}S_{i}$ be the union of the essential subsurfaces in the complement. For each $i\in I$, choose geodesic representatives of $\partial S_{i}$ separating  $S_{i}$ from $L$. Each of these geodesic representatives is either a leaf of $L$ or contained in $S_i$, since the regions $S_i$ are geodesically convex. Cutting along these boundary components yields a subsurface $S_{i}'\subseteq S_{i}$ of $\Sigma$, for each $i\in I$, which is equal to $S_i$ only if its boundary consists of proper geodesics which are also leaves of $L$. The \textit{support} of $L$, denoted $\mathrm{supp}(L)$, is the subsurface $\overline{\Sigma\setminus\left(\bigcup_{i\in I}S_{i}'\right)}$. 
\end{defn}

The following lemma summarizes several basic properties of the support of a lamination.

\begin{lem}\label{lem:witness-filling}
   Let $L$ be a geodesic lamination on an infinite-type surface $\Sigma$. Then: 
        \begin{enumerate}[label=(\arabic*)]
            \item $L\subset\mathrm{supp}(L)$;
            \item if $L$ is contained in a finite-type subsurface of $\Sigma$, then $\mathrm{supp}(L)$ is the subsurface obtained as the closure of a uniform $\epsilon$-neighbourhood of $L$ for some $\epsilon> 0$;
            \item $L$ intersects every simple closed curve and every proper arc in $\mathrm{supp}(L)$; and
            \item $L$ intersects every grand arc if and only if $\mathrm{supp}(L)$ is a witness.
        \end{enumerate}
\end{lem}

\begin{proof}
    For (1),  notice that the union $\bigcup_{i\in I}S_{i}'$ has non-empty intersection with $L$ exactly when $L$ has a proper leaf which separates $L$ from $\Sigma\setminus L$. In this case, this proper leaf is a common boundary component of $\mathrm{supp}(L)$ and $\cup_{i\in I}S_{i}'$.

    For (2), note that a lamination supported on a finite-type surface will have at most finitely many complementary subsurfaces $S_i$. The collection of geodesic representatives for their boundaries must also be finite. Choosing $\epsilon >0$ smaller than the smallest collar neighbourhood of these geodesics gives us the support; any leaf of $L$ that enters the collar neighbourhood of this geodesic must remain in this collar neighbourhood, and hence is asymptotic to the proper geodesic.

    For (3), suppose there is a proper geodesic, that is, a simple closed curve or a proper arc, contained in $\mathrm{supp}(L)$ that is disjoint from $L$. Then this proper geodesic is contained in $S_{i}$ for some $i\in I$. Hence, it must also be contained in $S_{i}'$, and thus cannot be contained in $\mathrm{supp}(L)$, which is a contradiction.

    For (4), If $L$ intersects every grand arc, then so does $\mathrm{supp}(L)$ by (1), and so $\mathrm{supp}(L)$ is a witness. Conversely, if $\mathrm{supp}(L)$ is a witness, then every grand arc must intersect $\mathrm{supp}(L)$ in a non-empty collection of proper geodesics, and $L$ must intersect each of these by (3). 
\end{proof}

In light of Lemma~\ref{lem:witness-filling} (3), we say that $L$ \textit{fills} a subsurface $Y$ if $Y=\mathrm{supp}(L)$. We say a lamination $L$ is \textit{witness-filling} if $\mathrm{supp}(L)$ is a witness and \textit{connected} if $\mathrm{supp}(L)$ is connected. The support of $L$ being connected does not imply that $L$ is connected as a subset of $\Sigma$: for example, consider a lamination formed by two proper arcs converging to the same end.

\subsection{Convergence as geodesics}

Let $\Gamma\left(\mathbb{H}^{2}\right)$ denote the space of geodesics in $\mathbb{H}^{2}$, where every geodesic is identified with a pair of distinct points on $\partial\mathbb{H}^2\cong S^1$. A sequence of geodesics $\left\{ \gamma_{n}\right\} _{n\in\mathbb{N}}$ in $\mathbb{H}^{2}$ \textit{geodesically converges} to a geodesic $\gamma$ if 
the corresponding end points of $\gamma_{n}$ converge to the end points of $\gamma$ in the topology of $\mathbb S^{1}$. We denote this convergence by $\gamma_{n}\xrightarrow{\mathrm{geod}}\gamma$. With the topology this induces, note that $\Gamma\left(\mathbb{H}^{2}\right)$ is homeomorphic to an open Möbius band. 

Let $\Gamma\left(\Sigma\right)$ denote the space of simple geodesics on $\Sigma$. This space can be realised as the quotient of $\Gamma\left(\mathbb{H}^2\right)$ by the $\pi_1(\Sigma)$-action via the fixed universal cover $\pi\colon\mathbb{H}^2\to\Sigma$ defined at the start of this section, and we endow $\Gamma(\Sigma)$ with the quotient topology. The lemma below gives an alternate description of convergence in this topology, which we continue to call \textit{geodesic convergence} and denote by $\xrightarrow{\text{geod}}$.

\begin{lem}[Convergence in $\Gamma(\Sigma)$]\label{lem:convg_as_geod}
A sequence $\left\{ \gamma_{n}\right\} _{n\in\mathbb{N}}$  in $\Gamma\left(\Sigma\right)$ geodesically converges to some $\gamma_{0}\in \Gamma\left(\Sigma\right)$ if and only if for any unit-speed parametrization of $\gamma_{0}$, there exist unit-speed parametrizations of $\left\{ \gamma_{n}\right\} _{n\in\mathbb{N}}$ such that for any $R,\epsilon>0$, there exists some $N\in\mathbb{N}$ such that for any $n\geq N$ and every   $t\in\left[-R,R\right]$,
\[
d_{\Sigma}\left(\gamma_{n}(t),\gamma_{0}(t)\right)<\epsilon.
\]
\end{lem}

\begin{proof}
Choose a unit-speed parametrization of $\gamma_{0}$, or, equivalently, fix a point $P=\gamma_{0}(0)$ on $\gamma_{0}$. Fix a lift $\tilde{\gamma}_{0}$ of $\gamma_{0}$. If $ \gamma_{n}\geodto \gamma_0$, then there exists a sequence of lifts $\tilde{\gamma}_{n}$ of $\gamma_{n}$ such that $\tilde{\gamma}_{n}\geodto\tilde{\gamma}_{0}$ in $\Gamma\left(\mathbb{H}^{2}\right)$. Thus, we can parametrize $\left\{ \tilde{\gamma}_{n}\right\} _{n\in\mathbb{N}}$ appropriately so that $\tilde{\gamma}_{n}(0)\to\tilde{\gamma}_{0}(0)$  and $\tilde{\gamma}_{n}(t)\to\tilde{\gamma}_{0}(t)$ pointwise for any $t\in\mathbb{R}$. The conclusion follows.

Conversely, suppose that a sequence of geodesics $\left\{ \gamma_{n}\right\} _{n\in\mathbb{N}}$ converges to $\gamma_{0}$ as geodesics. Pick a lift $\tilde{\gamma}_{0}$ of $\gamma_{0}$, parametrized as $\gamma_{0}$. Then we can find parametrized lifts $\tilde{\gamma}_{n}$ of $\gamma_{n}$, such that for $t\in\left[-R,R\right]$, 
\[
d_{\mathbb{H}^{2}}\left(\tilde{\gamma}_{n}(t),\tilde{\gamma}_{0}(t)\right)<\epsilon.
\]
 It follows that this sequence of lifts $\left\{ \tilde{\gamma}_{n}\right\} _{n\in\mathbb{N}}$ indeed converges to $\tilde{\gamma}_{0}$ in $\Gamma\left(\mathbb{H}^{2}\right)$.
\end{proof}

Note that $\Gamma(\Sigma)$ is not $T_1$, as a non-proper geodesic in $\Sigma$ may limit to multiple geodesics not including itself.

\subsection{The (coarse) Chabauty topology} \label{subsec:chabauty}
Given a space $X$, let $C(X)$ denote the space of closed subsets of a space $X$.  The space $C(X)$ can be endowed with the \textit{Chabauty topology} (also called the \textit{geometric topology}).  We will not use the definition of this topology in this level of generality; see \cite{CEG06} for a full discussion.  Instead, we will define the topology in the context of laminations by describing convergence; see Lemma~\ref{lem:geometric_convergence}.

Let $\Lambda\left(\Sigma\right)$ denote the space of geodesic laminations on $\Sigma$. By convention, we allow the empty set to be a lamination, as well.  The lifts of these laminations are closed subsets of $\Gamma\left(\mathbb{H}^{2}\right)$ that are invariant under the $\pi_{1}\left(\Sigma\right)$-action, and we can consider the topology induced from the Chabauty topology on $C\left(\Gamma\left(\mathbb{H}^2\right)\right)$. 

Alternatively, we can consider laminations as closed subsets of the space of geodesics in $\Sigma$, that is, as elements of $C(\Gamma(\Sigma))$, and endow $\Lambda(\Sigma)$ with the subspace topology of the Chabauty topology on $C(\Gamma(\Sigma))$.  These two topologies on $\Lambda(\Sigma)$ agree:

\begin{prop}[{\cite[Proposition~4.1.7]{CEG06}}]\label{prop:Chabauty_tops_coincide}
Let $C\left(\Sigma\right)$ be the space of closed subsets of $\Sigma$ with the Chabauty topology, and let $\mathcal{L}\subset C\left(\Sigma\right)$ be defined by 
\[
\mathcal{L}=\left\{ X:X=\left|\lambda\right|\text{ for some lamination }\lambda\subset\Sigma\right\} .
\]
Then $\mathcal{L}$ is a closed subset of $C(\Sigma)$, and the map $\Lambda\left(\Sigma\right)\to \mathcal{L}$ defined by $\lambda\mapsto\left|\lambda\right|$ is a homeomorphism. In particular, the topologies induced by regarding $\Lambda\left(\Sigma\right)$ as a subset of $C\left(\Gamma\left(\mathbb{H}^{2}\right)\right)$ and $C\left(\Sigma\right)$ agree.
\end{prop}

Here, $\left|\lambda\right|$, which is notation borrowed from \cite{CEG06}, is the lamination $\lambda$ thought of as a subspace of $\Sigma$, that is,  as points in $\Sigma$, instead of geodesics on $\Sigma$.

In \cite[Proposition~4.1.6]{CEG06}, this was proven under the assumption that the surface is compact.  However, since $\Gamma\left(\mathbb{H}^2\right)$ is Hausdorff, locally compact and second countable, the space of geodesic laminations $\Lambda\left(\Sigma\right)$ with the Chabauty topology is compact, metrizable and separable \cite[Proposition~4.1.7]{CEG06}. Moreover, as noted above, any geodesic lamination on $\Sigma$ is nowhere dense. Thus, the proof of \cite[Proposition~4.1.6]{CEG06} carries over verbatim.


We now describe convergence in the Chabauty topology on $\Lambda(\Sigma)$, which we denote by $\Cto$; this can be taken as the definition of the topology.  
\begin{lem}[{Chabauty convergence, \cite[Lemma~3.1.3]{CEG06}}] \label{lem:geometric_convergence}
A sequence of geodesic laminations $\left\{ L_{n}\right\} _{n\in\mathbb{N}}$ on $\Sigma$ converges to a lamination $L$ in $\Lambda(\Sigma)$ with the Chabauty topology if and only if 
\begin{enumerate}[label=(\arabic*)]
\item for all $\gamma\in L$, there is a sequence $\gamma_{i}\in L_{i}$ such that $\gamma_{i}\xrightarrow[]{\mathrm{geod}}\gamma$, and
\item for every subsequence $\left\{ L_{n_{k}}\right\} _{k\in\mathbb{N}}$, if $\gamma_{n_{k}}\in L_{n_{k}}$ and $\gamma_{n_{k}}\xrightarrow[]{\mathrm{geod}}\gamma$ for some $\gamma\in\Gamma\left(\Sigma\right)$, then $\gamma\in L$. 
\end{enumerate}
\end{lem}

When $\Sigma$ is a compact surface, the Chabauty topology on $C(\Sigma)$, and hence, the topology on $\Lambda(\Sigma)$, corresponds to the topology induced by the Hausdorff metric. However, for infinite-type surfaces, this is not the case, as illustrated by the following example.

\begin{example}
Let $\left\{ \gamma_{i}\right\} _{i\in\mathbb{N}}$ be a sequence of essential simple closed curves on $\Sigma$ which converge to an end $e$ of $\Sigma$. For example, choose a pants decomposition for $\Sigma$, and let $\left\{ \gamma_{i}\right\} _{i\in\mathbb{N}}$ be a sequence of curves bounding the pairs of pants which converge to $e$. Then $\left\{ \gamma_{i}\right\} \xrightarrow{C}\emptyset$, the empty lamination. To see this, note that the first condition of Lemma \ref{lem:geometric_convergence} is trivially true, and the second condition is true as no subsequence of $\left\{ \gamma_{i}\right\} _{i\in\mathbb{N}}$ may converge to a geodesic. On the other hand, assuming the length of the seams of every pair of pants is bounded below, the sequence of laminations $\left\{ \left\{ \gamma_{i}\right\} \right\} _{i\in\mathbb{N}}$ does not Hausdorff-converge to any lamination, as the length of the seam provides a lower bound for the Hausdorff distance between $\left\{ \gamma_{i}\right\} $ and $\left\{ \gamma_{i+1}\right\}$. Thus, the topology on $\Lambda(\Sigma)$ induced by the Hausdorff metric is not compact, even though the Chabauty topology is compact.
\end{example}

The following lemma will be useful in the later sections.
\begin{lem}\label{lem:disjt_convg}
Suppose that $\left\{ L_{n}\right\} _{n\in\mathbb{N}}$ and $\left\{ M_{n}\right\} _{n\in\mathbb{N}}$ are sequences in $\Lambda\left(\Sigma\right)$ such that $L_{n}\xrightarrow[]{C} L$ and $M_{n} \xrightarrow[]{C} M$. If $L_{n}\pitchfork M_{n}=\emptyset$ for all $n\in\mathbb{N}$, then $L\pitchfork M=\emptyset$. 
\end{lem}

\begin{proof}
Suppose for contradiction that there exists a point point $P\in L\pitchfork M$. Then there are leaves $l\in L$ and $m\in M$ that intersect transversely such that $P\in l\cap m$. By Lemma~\ref{lem:geometric_convergence}, there are sequences of leaves $l_{i}\in L_{i}$ and $m_{i}\in M_{i}$ such that $l_{i}\to l$ and $m_{i}\to m$ in $\Gamma\left(\Sigma\right)$. But then $l_{i}\pitchfork m_{i}\neq\emptyset$ for any sufficiently large $i$, contradicting the assumption.
\end{proof}

We shall also consider the \textit{coarse Chabauty topology} on $\Lambda\left(\Sigma\right)$; see \cite[Definition 4.1.10]{CEG06}, where it is called the \textit{Thurston topology}. This topology is given by the following notion of convergence. 
\begin{defn}[Coarse Chabauty convergence]\label{def:CCconvergence}
A sequence of laminations $\left\{ L_{n}\right\} _{n\in\mathbb{N}}$ \textit{coarse Chabauty converges} to $L\in\Lambda\left(\Sigma\right)$, denoted by $L_{n}\xrightarrow{CC}L$, if for any subsequence $\left\{ L_{n_{k}}\right\} _{k\in\mathbb{N}}$ of $\left\{ L_{n}\right\} _{n\in\mathbb{N}}$ that converges in the Chabauty topology to some $L_{0}\in\Lambda\left(\Sigma\right)$, we have $L\subset L_{0}$. Equivalently, for all $\gamma\in L$, there is a sequence $\gamma_{i}\in L_{i}$ such that $\gamma_{i}\xrightarrow[]{\mathrm{geod}}\gamma$.
\end{defn}

Note that $L_n \CCto L$ if and only if the first condition in Lemma \ref{lem:geometric_convergence} holds. As its name indicates, the coarse Chabauty topology on $\Lambda\left(\Sigma\right)$ is coarser than the Chabauty topology and is not Hausdorff, since a sequence of laminations can coarsely converge to different limits.




\subsection{Grand ending laminations}
Let $\Sigma$ be an infinite-type surface with $|\mc{S}(\Sigma)|\geq 3$, and let $\Gamma_{\text{grand}}\left(\Sigma\right)\subseteq \Gamma(\Sigma)$ be the subspace of grand arcs on $\Sigma$.  We emphasize the distinction from  $\mathcal{G}\left(\Sigma\right)$, which is the grand arc graph: both have the same underlying set, but different topologies.

\begin{defn}
A connected lamination $\Sigma$ is a \textit{grand lamination} if it transversely intersects every witness subsurface in $\Sigma$.
\end{defn}

The following lemma says that any Chabauty limit of a sequence of grand arcs must be a lamination which intersects every witness of $\mc{G}(\Sigma)$, and, in particular, be non-empty.  

\begin{lem}\label{lem:grandarctograndlam}
Let $\left(x_{n}\right)_{n\in\mathbb{N}}\subset\Gamma_{\text{grand}}(\Sigma)$ such that $x_{n}\xrightarrow{C}L$ for some connected lamination $L$. Then $L$ is a grand lamination.
\end{lem}

\begin{proof}
    If $L$ is not a grand lamination, then there exists a witness, and in particular, a compact witness $W\subseteq \Sigma$ such that $L\cap W=\emptyset$. However, since each $x_n$ is a grand arc, $x_n \cap W\neq \emptyset$. By the compactness of the Hausdorff topology on $C(W)$, the sequence $(x_n \cap W)_{n\in \mathbb N}$ must converge up to subsequence to a non-empty closed set. Points in this set must be in $L$ by Lemma \ref{lem:geometric_convergence} (2), which is a contradiction.
\end{proof}

\begin{defn}\label{def:GEL}
    A \textit{grand ending lamination} on $\Sigma$ is a minimal lamination on $\Sigma$ that is witness-filling. The space of grand ending laminations, denoted $\mathcal{GEL}\left(\Sigma\right)$ and considered as a subspace of $\Lambda\left(\Sigma\right)$, is equipped with the coarse Chabauty topology.
\end{defn}

\begin{example}\label{ex:cpct_pA}
Fix $a\in \mc G(\Sigma)$, and suppose $W\subset\Sigma$ is a compact witness for $\mathcal{G}\left(\Sigma\right)$ such that $a\cap W$ is connected. Let $\varphi\in\mathrm{MCG}\left(\Sigma\right)$ be supported on $W$ and be pseudo-Anosov on $W$. Then $\varphi$ acts loxodromically on $\mathcal{G}\left(\Sigma\right)$ by \cite[Theorem 1.5]{B-NV}. Thus $\varphi$ has two fixed points $\varphi^{+}$ and $\varphi^{-}$ in $\partial\mathcal{G}\left(\Sigma\right)$, and $\varphi^{n}(a)\to\varphi^{+}$ and $\varphi^{-n}(a)\to\varphi^{-}$ in $\mc G(\Sigma)\cup \partial\mc G(\Sigma)$. 

On the other hand, $\varphi|_{W}$ is also a loxodromic isometry of $\mathcal{A}\left(W\right)$. Let $a_{W}:=a\cap W$ and $\varphi_{W}:=\varphi|_{W}\in\mathrm{MCG}(W)$. Then the sequences $\left(\varphi_{W}^{n}\left(a_{W}\right)\right)_{n\in\mathbb{N}}$ and $\left(\varphi_{W}^{-n}\left(a_{W}\right)\right)_{n\in\mathbb{N}}$ also converge to points in $\partial\mathcal{A}(W)$. Since there is a $\mathrm{MCG}\left(W\right)$-equivariant homeomorphism $G\colon\mathcal{EL}_{0}(W)\to\partial\mathcal{A}\left(W\right)$, there are two ending geodesic laminations $L^{+}$ and $L^{-}$ supported on $W$ such that $\varphi_{W}^{\pm n}\left(a_{W}\right)\to L^{\pm}$ in the coarse Hausdorff topology on $\mc{EL}_0$ (which is the same as coarse Chabauty topology since $W$ is compact). Since $\varphi$ is identity outside $W$, it follows that $\varphi^{\pm n}(a)\xrightarrow{CC}L^{\pm}$. 
In Example~\ref{ex:agrees_with_F}, we show that the map from Theorem~\ref{thm_main} sends $L^\pm$ to $\varphi^\pm$.
\end{example}

 We now construct a basis for the coarse Chabauty topology on $\mathcal{GEL}(\Sigma)$, similar to \cite{Pho-On}. For $\epsilon>0$ and $L_{0}\in\mathcal{GEL}\left(\Sigma\right)$, define 
\[
U_{\epsilon}\left(L_{0}\right)=\left\{ L\in\mathcal{GEL}\left(\Sigma\right)\mid L_{0}\subset N_{\epsilon}\left(L\right)\right\} ,
\]
where $N_{\epsilon}\left(L\right)$ is the $\epsilon$-neighbourhood around $L$ in $\Lambda\left(\Sigma\right)$, rather then the $\epsilon$-Hausdorff neighbourhood of $L$ on $\Sigma$.  Note that $N_{\epsilon}\left(L\right)$  is well-defined as $\Lambda\left(\Sigma\right)$ is metrizable. The proof of \cite[Lemma~2.2]{Pho-On} holds verbatim in this setting to show that 
\[
\mathcal{B}=\left\{ U_{\epsilon}\left(L_{0}\right)\mid L_{0}\in\mathcal{GEL}\left(\Sigma\right)\right\} 
\] is a basis for a topology on $\mathcal{GEL}\left(\Sigma\right)$. Furthermore, the proof of  \cite[Proposition~2.3]{Pho-On} also carries over, with the Hausdorff distance between laminations being replaced by distance in $\Lambda\left(\Sigma\right)$, to show that the topology on $\mathcal{GEL}\left(\Sigma\right)$ induced by $\mathcal{B}$ corresponds to the coarse Chabauty topology. In particular, this implies that $\mathcal{GEL}(\Sigma)$ is first-countable. Moreover, if  the grand arc graph of $\Sigma$ is  $\delta$-hyperbolic, then $\mathcal{GEL}\left(\Sigma\right)$ is Hausdorff. To see this, suppose a sequence $\left(L_{i}\right)_{i\in\mathbb{N}}\subset\mathcal{GEL}\left(\Sigma\right)$ coarsely converges to two different limits $M_{1}, M_{2} \in \mathcal{GEL}(\Sigma)$. By Lemma~\ref{lem:disjt_convg}, $M_{1}$ and $M_{2}$ must be disjoint minimal laminations, so the witnesses filled by $M_{1}$ and $M_{2}$ must also be disjoint. This contradicts the hyperbolicity of $\mathcal{G}\left(\Sigma\right)$.

\begin{cor}\label{cor:seq_continuity}
Let $Y$ be any topological space. A map $f\colon \mathcal{GEL}\left(\Sigma\right)\to Y$ is continuous if and only if for every sequence $\left\{ L_{n}\right\} _{n\in\mathbb{N}}\subset\mathcal{GEL}\left(\Sigma\right)$ such that $L_{n}\xrightarrow{CC}L_{0}$, we have $f\left(L_{n}\right)\to f\left(L_{0}\right)$ in $Y$.
\end{cor}

\begin{proof}
    A map from a first-countable space is continuous if and only if it is sequentially continuous, and we have shown that $\mathcal{GEL}(\Sigma)$ is first-countable.
\end{proof}

We end by showing that minimal, witness-filling lamination that does not transversely intersect a grand lamination must be a sublamination.
\begin{lem}\label{lem:sublamination}
    If $L$ is a minimal, witness-filling lamination and $M$ is a grand lamination such that $L\pitchfork M=\emptyset$, then $L$ is a sublamination of $M$.
\end{lem}

\begin{proof}
Since $M$ intersects every witness transversely, it must intersect $W=\mathrm{supp}(L)$ transversely. Since $L$ intersects every simple closed curve and proper arc in $W$ by Lemma~\ref{lem:witness-filling} (3) and $M$ is disjoint from $L$, it follows from Lemma~\ref{lem:witness-filling} (4) that no component of $M\cap W$ can be a proper geodesic. In particular, every leaf  $m\in M$ intersecting $W$ must either be a leaf of $L$ or asymptotic to $L$. The limit set of $m$, which is equal to $\overline{m}$ if $m$ is a leaf of $L$ or to $\overline{m}\setminus m$ if $m$ is asymptotic to $L$, determines a sublamination of both $L$ and $M$. By the minimality of $L$, this must be equal to $L$. Thus, $L\subset M$.
\end{proof}

\begin{cor}\label{cor:disjointarcTOsublamination}
    Let $(x_n)_{n\in \mathbb N}$ and $(y_n)_{n\in \mathbb N}$ be sequences of grand arcs with $x_n\cap y_n=\emptyset$.  If $x_n$ Chabauty converges to a lamination $M$ and $y_n$ Chabauty converges to a minimal, witness-filling lamination $L$, then $L\subseteq M$.
\end{cor}

\begin{proof}
    By Lemma~\ref{lem:grandarctograndlam}, $M$ is a grand lamination.  Since $x_n\cap y_n=\emptyset$, we have $M\pitchfork L=\emptyset$ by Lemma~\ref{lem:disjt_convg}. The result then follows by Lemma~\ref{lem:sublamination}.
\end{proof}

\section{Bicorns with generic geodesics}\label{sec:bicorns_with_generic_geod}

In this section, we construct an explicit sequence of grand arcs that converges to a grand lamination in the coarse Chabauty topology. The idea is essentially similar to the  infinite unicorn paths constructed by Pho-On \cite{Pho-On}. However, since the grand laminations could potentially have limit points in every maximal end type, there is no canonical ``starting point'' one can use to define unicorn arcs between a grand arc and a leaf of a grand lamination. Using bicorn arcs allows us to circumvent this.

\begin{defn}\label{def:bicorn_with_geodesic}
Let $a\in\mathcal{G}(\Sigma)$ and $l\in\Gamma\left(\Sigma\right)$. A \textit{bicorn arc between $a$ and $l$} is a grand arc of the form $x=a_1\cup l_1 \cup a_2$, where $a_{1},a_{2}\subset a$ are the closures of the unbounded components of $a\setminus\{P,Q\}$ and $l_{1}\subset l$ is the closure of the bounded component of $l\setminus \{P,Q\}$, where $P,Q\in a\cap l$.  We allow any of $a_i,l_1$ to be empty.  Denote the collection of bicorns between $a$ and $l$ by $\mathcal{B}\left(a,l\right)$. 
\end{defn}

If $l$ is a grand arc, then this definition agrees with the definition of a bicorn arc in Section \ref{sec:finite_bicorns}. In particular, $l\notin\mathcal{B}\left(a,l\right)$ for any $l\in\Gamma(\Sigma)$. 

\begin{lem}\label{lem:bicorn_subsets}
For any $a\in\mathcal{G}\left(\Sigma\right)$,  any $l\in\Gamma\left(\Sigma\right)$, and any $b\in\mathcal{B}\left(a,l\right)$, we have $\mathcal{B}\left(a,b\right)\subset\mathcal{B}\left(a,l\right)$.  
\end{lem}

\begin{proof}
If $b=a_{1}\cup l_{1}\cup a_{2}\in\mathcal{B}\left(a,l\right)$ is a bicorn arc, where $a_{1},a_{2}\subset a$ and $l_{1}\subset l$, then as in the proof of Lemma \ref{lem:finite_bicorn_subset}, any point of $a\cap b$ is a point of $a\cap l_{1}$. Thus, any bicorn between $a$ and $b$ is also be a bicorn between $a$ and $l$.  
\end{proof}
The proof of Lemma~\ref{lem:bicorn_changing_first_arc} carries over verbatim by replacing $b$ with $l$, as the proof does not use the assumption that $b$ is a grand arc.  Therefore, we have:

\begin{lem} \label{lem:bicorns_change_base_pt}
Let $a,a'\in\mathcal{G}\left(\Sigma\right)$ be disjoint grand arcs, and let $l\in\Gamma\left(\Sigma\right)$ be a geodesic. Then 
\[
d^{\mc{G}}_{\mathrm{Haus}}\left(\mathcal{B}\left(a,l\right),\mathcal{B}\left(a',l\right)\right)\leq2.
\]
\end{lem}

We next give a relationship between bicorn arcs with geodesics in a convergent sequence and bicorn arcs with the limiting geodesic.  

\begin{lem}\label{lem:bicorns_for_converging_leaves}
Let $a\in\mathcal{G}\left(\Sigma\right)$ and $l\in\Gamma\left(\Sigma\right)$. If $\left\{ l_{n}\right\} _{n\in\mathbb{N}}\subset\Gamma\left(\Sigma\right)$ is such that $l_{n}\geodto l$, then for any $b\in\mathcal{B}\left(a,l\right)$, there exists $N>0$ such that $b\in\mathcal{B}\left(a,l_{n}\right)$ for every $n\geq N$.
\end{lem}

\begin{proof}
Fix a unit-speed parametrization on $l$, and appropriately parametrize each $l_{n}$ at unit-speed such that $l_{n}\geodto l$ as parametrized geodesics.
Let $b=a_{1}\cup\lambda_{1}\cup a_{2}\in \mc B(a,l)$, where $a_{1},a_{2}\subset a$ and $\lambda_{1}=l\left(\left[t_{1},t_{2}\right]\right)\subset l$ for some $t_{1},t_{2}\in\mathbb{R}$.  Choose $\epsilon>0$ small enough so that the $\epsilon$--neighbourhood of $\lambda_1$ is an embedded disk. By Lemma \ref{lem:convg_as_geod}, there exists $N>0$ such that $d_{\mathrm{Haus}}^\Sigma\left(l_{n}\left(\left[t_{1},t_{2}\right]\right),l\left(t_{1},t_{2}\right)\right)<\epsilon$ for every $n\geq N$, where $d_{\mathrm{Haus}}^\Sigma$ denotes the Hausdorff distance in the fixed hyperbolic metric on $\Sigma$.  By choosing $\epsilon$ sufficiently small, we can ensure that there are intersections $l_n([t_1,t_2])\cap a$ at distance at most $\varepsilon$ on $\Sigma$ from the intersections of $\lambda_1\cap a$.  Using these intersections, we see that $l_{n}\left(\left[t_{1},t_{2}\right]\right)$ forms a bicorn with $a$ that is isotopic to $b$. Thus, $b\in\mathcal{B}\left(a,l_{n}\right)$.
\end{proof}

\subsection{Ending bicorn sequences for minimal laminations}

In this subsection, given a minimal lamination $L$, we will construct a sequence of grand arcs that coarse Chabauty converges to $L$. To do this, we construct a sequence of grand arcs that geodesically converges to a leaf of $L$, so that the sequence coarse Chabauty converges to $\overline{l}=L$. First, we show that the choice of leaf doesn't affect the set of bicorn arcs. 

\begin{lem} \label{lem:bicorns_with_lamination}
If $a\in\mathcal{G}\left(\Sigma\right)$ and $l,l'\in L$ are two leaves of a minimal lamination $L$, then $\mc{B}(a,l)=\mc{B}(a,l')$.
\end{lem}

\begin{proof}
Let $b=a_{1}\cup l_{1}\cup a_{2}\in\mathcal{B}\left(a,l\right)$, so that $a_{1},a_{2}\subset a$, with $a \cap l_1=\{P,Q\}$. We will find $b'\in\mathcal{B}\left(a,l'\right)$ homotopic to $b$.  The same argument, switching the roles of $l$ and $l'$ shows that the two sets are equal.

Since $L$ is minimal, every leaf is dense, and hence for any lift $\hat{l}$ of $l$ to the universal cover $\mathbb{H}^{2}$ of $\Sigma$, there is a sequence of lifts of $l'$ that converges to $\hat{l}$. Let $\hat{l}_{1}$ denote a (possibly unbounded) segment of $\hat{l}$ that projects to $l_{1}$. For any $\epsilon>0$, we can choose a lift of $l'$ close enough to $\hat{l}$ so that there is a segment $l'_{1}\subset l'$ such that $d_{\mathrm{Haus}}\left(l_{1},l_{1}'\right)<\epsilon$ as closed segments in $\Sigma$. In particular, by choosing $\epsilon$ small enough, there are intersections $P',Q'\in a \cap l_1'$ , each at distance at most $\epsilon$ on $\Sigma$ from $P$ and $Q$, respectively.  Moreover, we can choose $l_1'$ so that $l_{1},l_{1}'$ bound a rectangular region in $\Sigma$ with vertices $P,Q,P',Q'$. It follows that $l_{1}'$ determines a bicorn arc $b'=a_{1}'\cup l_{1}'\cup a_{2}'$ homotopic to $b$. See Figure \ref{fig:bicorns_different_leaves}.
\end{proof}

\begin{figure}[h]
    \centering
    \includegraphics[width=10 cm]{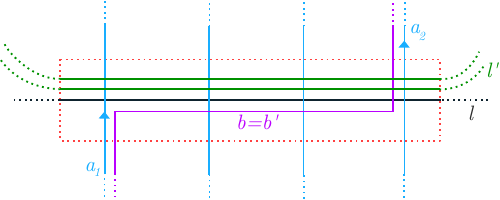}
    \captionsetup{margin=1.5 cm, justification=justified, singlelinecheck=false}
    \caption{If $l'$ is sufficienlty close to $l$, then bicorns between $a$ and $l$ are also bicorns between $a$ and $l'$.}
    \label{fig:bicorns_different_leaves}
\end{figure}

In general, if $l$ and $l'$ are disjoint geodesics, the sets $\mc{B}(a,l)$ and $\mc{B}(a,l')$ may be far apart in $\mc{G}(\Sigma)$.  For example, suppose $l$ is the leaf of a minimal lamination filling a compact witness $W$ and $l'$ is a geodesic disjoint from $W$. If $a$ is a grand arc that intersects $W$ in a single component, then every bicorn arc in $\mathcal{B}\left(a,l'\right)$ would agree with $a$ in $W$, so Lemma \ref{lem:agree-inside-witness} implies that this set is a bounded subset of $\mathcal{G}(\Sigma)$. On the other hand, we will prove in  Lemma~\ref{lem:grand_bicorns_go_to_infty} that $\mathcal{B}\left(a,l\right)$ is not bounded.

Given any lamination $L$, let 
\[
\mathcal{B}\left(a,L\right):=\underset{l\in L}{\bigcup}\mathcal{B}\left(a,l\right).
\]
If $L$ is minimal, then this set is equal to $\mathcal{B}\left(a,l\right)$ for any $l\in L$.

\begin{defn}\label{def:ending}
    Suppose $a\in \mathcal{G}(\Sigma)$, $L$ is a minimal lamination on $\Sigma$, and $l\in L$. A sequence $\left(b_n\right)_{n\in \mathbb{N}}\subset \mathcal{B}(a,l)$ 
    is an \textit{ending bicorn sequence between $a$ and $L$} if there exists $R\in \mathbb N$ such that:
        \begin{enumerate}[label=(\arabic*)]
            \item $b_0=a$,
            \item $b_n \geodto l$, 
            \item $\left(b_n\right)_{n\in \mathbb{N}}$ is an $R$--quasi-path, i.e., $d_\mathcal{G}\left(b_n , b_{n+1}\right)<R$  for all $n\in \mathbb N$, and 
            \item the $l$-part of $b_i$ is contained in the $l$-part of $b_{i+1}$.
        \end{enumerate} 
\end{defn}

In Proposition \ref{prop:bicorns_geodesic_conv} below, we shall explicitly construct ending bicorn sequences between a grand arc and a minimal lamination. To do so, we first need to choose an appropriately generic leaf of the lamination. 

The statement and proof for the lemma below were communicated to the authors by Mladen Bestvina, Federica Fanoni, Kasra Rafi, and Jing Tao.

\begin{lem}\label{lem:both_half_leaves_dense}
    Suppose $L$ is a minimal lamination on $\Sigma$ which is not a single proper geodesic. Then for any compact segment $T$ transverse to $L$, the set of points in $T\cap L$ corresponding to leaves both of whose half-leaves are dense is a dense $G_\delta$ subset of $T\cap L$. In particular, $L$ contains a non-boundary leaf both of whose half-leaves are dense in $L$.
\end{lem}

\begin{proof}
Let $T$ be a compact transversal to $L$, and let $K=T\cap L$. Note that $K$ does not have isolated points, as every leaf of $L$ is dense in $L$. Choose local orientations of leaves of $L$ at points on $K$. For each point $x\in K$, denote by $r_{x}^{+},r_{x}^{-}$  the two half-leaves based at $x$, with respect to the local orientation. 

Let $\left\{ U_{n}\right\} _{n=1}^{\infty}$ be a countable basis for the topology of $L$. Define 
\[
E_{n}^{\pm}:=\left\{ x\in K:r_{x}^{\pm}\cap U_{n}\neq\emptyset\right\} .
\]
Then each $E_{n}^{\pm}$ is relatively open in $K$: hitting a fixed relatively open set of $L$ after traveling along a leaf for a finite length is stable under small transverse perturbations in $T$. Observe that for any $x\in K$, 
\begin{enumerate}
\item since the leaf of $L$ through $x$ is dense in $L$, either $x\in E_{n}^{+}$ or $x\in E_{n}^{-}$ (possibly both);
\item $x\in E_{n}^{+}\cap E_{n}^{-}$ for every $n\in\mathbb{N}$ if and only if both half-leaves $r_{x}^{\pm}$ are dense in $L$; and
\item $x\in E_{n}^{+}$ for all $n\in\mathbb{N}$ if and only if $r_{x}^{+}$ is dense in $L$. Similarly true for $r_{x}^{-}$.
\end{enumerate}
We wish to show that 
\[
G=\bigcap_{n=1}^{\infty}\left(E_{n}^{+}\cap E_{n}^{-}\right)
\]
is a dense $G_{\delta}$ subset of the compact Baire space $K$, which, in particular, implies that $G\neq\emptyset$. To show this, it suffices to show that each $E_{n}^{\pm}$ is dense in $K$ for every $n\in\mathbb{N}$.  We will show this holds for $E_n^+$; an analogous argument holds for $E_n^-$.

Suppose for the sake of contradiction that $E_{n}^{+}$ is not dense in $K$ for some $n\in\mathbb{N}$, so that there is a non-empty relatively open set $B\subset K$ disjoint from $E_{n}^{+}$. That is to say, for every $x\in B$, we have $r_{x}^{+}\cap U_{n}\neq\emptyset$; in particular, $r_{x}^{+}$ is not dense in $L$. Since the leaf of $L$ through $x$ is dense in $L$, this implies that $r_{x}^{-}$ is dense in $L$. Moreover, since $L$ is minimal and $r_{x}^{+}$ is not dense in $L$, this implies that $r_{x}^{+}$ is isolated in $L$, so that $r_{x}^{+}\cap K$ is finite.
\begin{claim}
\label{claim:open_subset_inside_B}There is a non-empty relatively open set $B_{0}\subset B$ such that  $r_{x}^{+}\cap B_{0}=\emptyset$ for any $x\in B_{0}$.
\end{claim}

\begin{proof}[Proof of Claim \ref{claim:open_subset_inside_B}]
Let $A=\left\{ x\in B\mid r_{x}^{+}\cap B=\emptyset\right\} $. Since $r_{x}^{+}\cap B$ is finite for every $x\in B$, we can follow along $r_{x}^{+}$ to reach the last point $y$ of $r_{x}^{+}\cap B$ along $r_{x}^{+}$. Then $r_{y}^{+}\cap B=\emptyset$. So, if $\mathcal{H}$ is the countable family of positive holonomy maps from open subsets of $K$ to $K$, we get that 
\begin{equation}\label{eqn:countUnion}
B\subset\bigcup_{h\in\mathcal{H}\cup\left\{ \mathrm{id}\right\} }h^{-1}(A).
\end{equation}

If $A$ were meager in $B$, then each $h^{-1}(A)$ would be meager in $B$, as each $h\in\mathcal{H}$ is a local homeomorphism. But then \eqref{eqn:countUnion} implies that $B$ is meager in itself, which is a contradiction. 

Thus, $A$ is not meager in $B$, and so the closure $\mathrm{cl}_{B}(A)$ of $A$ in $B$ has non-empty interior. Let 
\[
B_{0}\subset\mathrm{int}_{B}\left(\mathrm{cl}_{B}(A)\right)
\]
be a non-empty relatively open set. Then $A$ is dense in $B_{0}$. We claim that $B_{0}$ satisfies the hypothesis of Claim \ref{claim:open_subset_inside_B}.

Otherwise, for some $x\in B_{0}$, a positive holonomy map $h$ would be defined near $x$ and satisfy $h(x)\in B_{0}$. Since $h^{-1}\left(B_{0}\right)$ is an open neighbourhood of $x$ in its domain and $A$ is dense in $B_{0}$, we may choose $a\in A\cap B_{0}\cap h^{-1}\left(B_{0}\right)$. Then $r_{a}^{+}$ has a later intersection with $B_{0}\subset B$, contradicting the definition of $A$.
\end{proof}
Let $B_{0}\subset B$ be the non-empty relatively open subset constructed in Claim \ref{claim:open_subset_inside_B}. Choose $x\in B_{0}$. Since $r_{x}^{-}$ is dense, there exists some $y\in r_{x}^{-}\cap B_{0}$. Since $r_{x}^{-}$ is dense in $K$, so is $r_{x'}^{-}$ for each $x'\in r_{x}^{-}$. Thus, we can further assume $y\neq x$ by choosing $y$ sufficiently far along $r_{x}^{-}$. But in that case, $x\in r_{y}^{+}$, so $r_{y}^{+}\cap B_{0}\neq\emptyset$. This contradicts the definition of $B_{0}$. Thus, $E_n ^{\pm}$ is dense in $K$ for all $n\in\mathbb{N}$, and $G$, the set of points of $T\cap L$ on leaves both of whose half-leaves are dense, is a dense $G_\delta$ subset of $K$.

In particular, for any compact transversal $T$, the collection of boundary leaves in $L$ intersects $T$ in a countable collection of points. Since a countable set cannot be a dense $G_\delta$ subset of $T\cap L$ (which is homeomorphic to a Cantor set), there exists a non-boundary leaf of $L$ with both half-leaves dense.
\end{proof}

 We will use this in the proposition below to show that there is an ending bicorn sequence which coarse Chabauty converges to a given minimal lamination.

\begin{prop} \label{prop:bicorns_geodesic_conv}
 Let $L$ be a minimal lamination on $\Sigma$ which is not a single proper geodesic, and let $l\in L$ be a non-boundary leaf such that both half-leaves are dense. If $a\in \mc G(\Sigma)$ intersects $\supp(L)$ transversely, then there exists an ending bicorn sequence sequence $\left\{ b_{n}\right\} _{n\in\mathbb{N}}\subset\mathcal{B}\left(a,l\right)$.  In particular, $b_n\CCto L$. 
\end{prop}



The choice of a non-boundary leaf with both half-leaves dense in Proposition~\ref{prop:bicorns_geodesic_conv} is important: if $l$ is a boundary leaf of the lamination, then there is a first point of $a\cap l$ along $a$. Similarly, if $l$ has a half-leaf that is not dense in $L$, then there may be a first point of intersection of $a$ and $l$ along $l$. In either case, no sequence in $\mc{B}(a,l)$ can converge to $l$.

\begin{proof}[Proof of Proposition \ref{prop:bicorns_geodesic_conv}]
Since $a$ is a grand arc and the metric on $\Sigma$ is Nielsen-convex, there exists $\epsilon>0$ such that $N_{\epsilon}\left(a\right)\subset\Sigma$ is an embedded infinite rectangular strip. It follows that any two points of $a\cap l$ are distance at least $2\epsilon$ along $l$. 

Since $l$ is not a boundary leaf, there exists a point $P_{0} \in a\cap l$ such that points of $a\cap l$ accumulate to $P_{0}$ from both sides along $a$. Label the points of intersection of $a\cap l$ along $l$ as $\left\{ P_{n}\right\} _{n\in\mathbb{\mathbb{Z}}}$ by picking an orientation on $l$. The segment $l_{1}=\left(P_{0}P_{1}\right)_{l}$ defines a bicorn arc $b_{1}=a_{1}\cup l_{1}\cup a_{1}'$, as $a\cap\left(P_{0}P_{1}\right)_{l}=\emptyset$ by construction. Thus, $|a\cap b_1 |\leq 1$.

Inductively, suppose $b_{i}=a_{i}\cup l_{i}\cup a_{i}'$ is defined, with $l_{1}\subset l_{2}\subset\dots\subset l_{i}$ and $X=a_i\cap l_i$ and $Y=a_i'\cap l_i$. Since $l$ is not a boundary leaf, both $a_{i}$ and $a_{i}'$ intersect $l\setminus l_{i}$, as $a$ does not have a first point of intersection with $l$ along $a$ by assumption. Since both the components of $l\setminus l_i$ (that is, the two half-leaves of $l$) are dense in $L$, $a_{i}$ and $a_{i}'$ cannot both intersect just one of the components of $l\setminus l_{i}$. Let $X',Y'$ be the first points of $b_{i}\cap l$ along $l$ on either side of $X$ and $Y$, so that the points appear in the order $X'-X-Y-Y'$ along $l$. There are two cases to consider, depending on where $X',Y'$ lie.
\begin{casenv}
\item Suppose $X'\in a_{i}$ and $Y'\in a_{i}'$ or $X'\in a_{i}'$ and $Y'\in a_{i}$. Then $l_{i}=\left(X'Y'\right)_{l}$ defines a bicorn arc $b_{i+1}$ that intersects $b_{i}$ at most once. Moreover, $l_{i}\subset l_{i+1}$ by construction. See Figure \ref{fig:bicorns_algo} (left).
\item Suppose $X',Y'\in a_{i}$ or $X',Y'\in a_i'$.  Then $l_{i}=\left(X'Y\right)_{l}$ or $l_{i}=\left(XY'\right)_{l}$, respectively, defines a bicorn arc $b_{i+1}$ that intersects $b_{i}$ at most once. Moreover, $l_{i}\subset l_{i+1}$ by construction. See Figure \ref{fig:bicorns_algo} (right). 
\end{casenv}

In either case, we have constructed a sequence $b_0,\dots, b_{i+1}$ that satisfies Definition~\ref{def:ending}(1) and (4).  Moreover, since $|b_i\cap b_{i+1}|\leq 1$, we have $d_{\mc G}(b_i,b_{i+1})\leq 2$ by Lemma~\ref{lem:grand_arc_dist_2}, and hence Definition~\ref{def:ending}(3) holds with $R=2$.

Since for each $i$, the subpaths $a_i$ and $a_i'$ must both intersect both components of $l\setminus l_i$, the induction step cannot always land in Case 2.  In particular, the segments $l_i$ eventually cover all of $l$.

\begin{figure}[h]
    \centering
    \begin{subfigure}
    \centering
    \def\svgwidth{6.5 cm}
        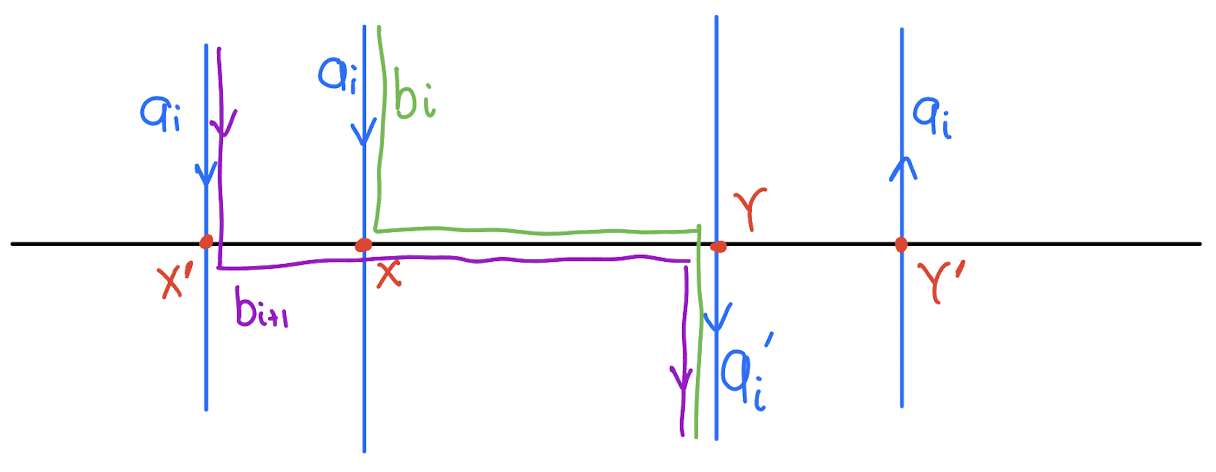
    \end{subfigure}\hfill
    \begin{subfigure}
    \centering
    \def\svgwidth{6.5 cm}
        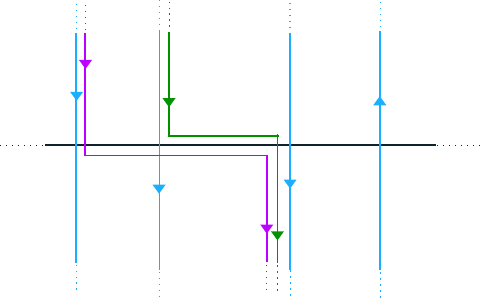
    \end{subfigure}
    
    \caption{Constructing the next bicorn $b_{i+1}$ in Case 1 (left) and Case 2 (right).}
    \label{fig:bicorns_algo}
\end{figure}

It remains to show that this sequence $(b_i)_{i\in \mathbb N}$ geodesically converges to $l$. Pick a lift $\tilde{l}$ of $l$ in the universal cover $\mathbb{H}^{2}$. Each point of $a\cap l$ corresponds to a lift of $a$ intersecting $\tilde{l}$ in $\mathbb{H}^{2}$. The segments $l_{i}$ defined above correspond to a nested sequence of segments $\tilde{l}_{i}$ along $\tilde{l}$, and the bicorn arc $b_{i}$ is obtained by joining $\tilde{l}_{i}$ to the corresponding lifts of $a$; see Figure \ref{fig:bad_angle_bicorns}.  If the lifts $\tilde{b}_{i}$ of the bicorns $b_{i}$ do not converge to $\tilde{l}$, then it must be the case that the angles of intersection of $a$ and $l$ go to $0$ along both sides of $l$.

\begin{figure}[h]
    \centering
    \def\svgwidth{8 cm}
    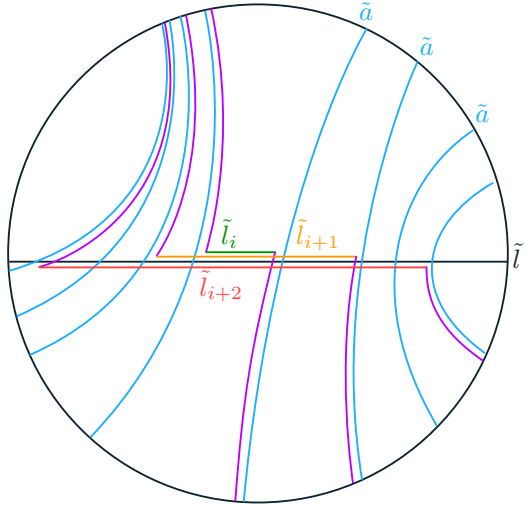
    \captionsetup{margin=1.5 cm, justification=justified, singlelinecheck=false}
    \caption{A sequence of bicorn arcs defined by lifts $\tilde{l}_i\subset \tilde{l}$ which do not converge to $\tilde{l}$, because the angles between lifts of $a$ and $\tilde{l}$ go to $0$ towards the left.}
    \label{fig:bad_angle_bicorns}
\end{figure}

Let $\theta$ be the angle of intersection of $a$ and $l$ at $P_{0}$. In a small neighbourhood of $P_{0}$ along $a$, there are infinitely many segments of $l$ that intersect $a$ at an angle at least $\theta-\kappa>0$, for some small $\kappa>0$.  Along $l$, these intersection points must limit to both endpoints of $l$, as $a\cap l$ is discrete along $l$. In particular, there are infinitely many lifts of $a$ that intersect $\tilde{l}$ at an angle at least $\theta-\epsilon>0$, and these intersections limit to both endpoints of $\tilde l$. See Figure~\ref{fig:angles_with_arc}.

\begin{figure}[h]
    \centering
    \def\svgwidth{8 cm}
    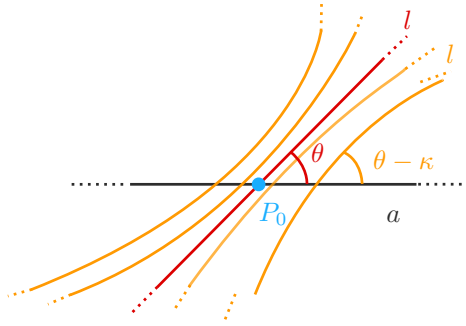
    \captionsetup{margin=1.5 cm, justification=justified, singlelinecheck=false}
    \caption{Since $l$ limits onto itself, the angles of $a\cap l$ near $P_0$ along $a$ would be very close to $\theta$.}
    \label{fig:angles_with_arc}
\end{figure}

Therefore, the angles of intersection of $a$ and $l$ cannot converge to $0$ along either side of $\tilde{l}$, and hence $\tilde b_i \geodto \tilde l$.  It follows that $b_i \geodto l$, completing the proof.
\end{proof}

\begin{rmk}
    In the construction above, if $l$ were either a boundary leaf or a leaf with a half-leaf that is not dense, then the resulting sequence of bicorn arcs would converge to a geodesic asymptotic to $l$ (and hence, $L$), but not necessarily to $l$ itself. In particular, we would still get a sequence of bicorn arcs which converges to $L$ in the coarse Chabauty topology.
\end{rmk}

\subsection{Ending bicorn sequences with grand ending laminations}

In this subsection, we show that an ending bicorn sequence between a grand arc $a$ and a leaf of a grand ending lamination $L$ is at finite Hausdorff distance from the full set of bicorns between $a$ and that leaf; see Proposition \ref{prop:map_well_def}. 

Our first goal is to show that such ending bicorn sequences are not bounded in $\mathcal{G}\left(\Sigma\right)$. The proof of is similar to \cite[Lemma~3.5]{Pho-On}.

\begin{lem}\label{lem:grand_bicorns_go_to_infty}
Suppose $a\in\mathcal{G}\left(\Sigma\right)$, $L\in\mathcal{GEL}\left(\Sigma\right)$ and $\left(b_{n}\right)_{n\in\mathbb{N}}$ is an ending bicorn sequence between $a$ and $L$. Then $\lim_{n\to\infty}d_{\mathcal{G}}\left(a,b_{n}\right)=\infty$.

\end{lem}

\begin{proof}
Suppose toward a contradiction that $\lim_{n\to\infty}d\left(a,b_{n}\right)\neq\infty$. Let $M_0>0$ be the constant in Corollary \ref{cor:unicorn_hausdorff_dist}. Corollary \ref{cor:unicorn_hausdorff_dist} and Lemma \ref{lem:bicorn_subsets} imply that $d_{\mathcal{G}}\left(a,b_{i}\right)\leq d_{\mathcal{G}}\left(a,b_{j}\right)+M_0$ for all $j\geq i$, so that $\sup d_{\mathcal{G}}\left(a,b_{n}\right)<\infty$.  In particular, there is no unbounded subsequence of $\left\{ b_{n}\right\} _{n\in\mathbb{N}}$. Therefore there exists some $N>0$ and an infinite subsequence $\left\{ b_{n}'\right\} _{n\in\mathbb{N}}$ of $\left\{ b_{n}\right\} _{n\in\mathbb{N}}$ such that $d_{\mathcal{G}}\left(a,b_{n}'\right)=N$. For each $n\in\mathbb{N}$, let $b_{n}^{1}\in\mathcal{G}\left(\Sigma\right)$ be a point on the geodesic between $a$ and $b_{n}'$ in $\mathcal{G}\left(\Sigma\right)$, with $d\left(a,b_{n}^{1}\right)=N-1$. By sequential compactness of $\Lambda\left(\Sigma\right)$, a subsequence of $\left\{ b_{n}^{1}\right\} _{n\in\mathbb{N}}$ converges to a lamination $L_{1}$.  By Corollary~\ref{cor:disjointarcTOsublamination}, we have $L\subseteq L_1$.

Repeating this process, for each $0<i<N$, we obtain a sequence of grand arcs $\left\{ b_{n}^{i}\right\} _{n\in\mathbb{N}}$ with $d_{\mathcal{G}}\left(a,b_{n}^{i}\right)=N-i$. As in the previous paragraph, a subsequence must converge to a lamination $L_{i}$ with $L\subset L_i$. However, in the final step, we have that $b_{n}^{N}=a$ for all $n\in\mathbb{N}$, which converges to the lamination $L_{N}=\left\{ a\right\} $. Clearly $L\nsubseteq\left\{ a\right\} $, which is a contradiction.
\end{proof}

We now show that ending bicorn sequences converge to a point in $\partial \mc G(\Sigma)$.
\begin{lem}\label{lem:defines_a_point}
Let $a\in\mathcal{G}\left(\Sigma\right)$, $L\in\mathcal{GEL}\left(\Sigma\right)$ and $l\in L$ be a non-boundary leaf with both half-leaves dense. An ending bicorn sequence $\left(b_{n}\right)_{n\in\mathbb{N}}\subset\mathcal{B}\left(a,l\right)$  converges to a point in $\partial\mathcal{G}\left(\Sigma\right)$.  
\end{lem}

\begin{proof}

By Lemma \ref{lem:grand_bicorns_go_to_infty}, for any $R>0$, there exists some $N>0$ such that for all $n\geq N$, $d_{\mathcal{G}}\left(a,b_{n}\right)>R$.  Let $S_{i,j}=\left\{ b_{i},b_{i+1},\dots,b_{j}\right\} \subset\left\{ b_{n}\right\} _{n\in\mathbb{N}}$. Note that $S_{i,j}\subset\mathcal{B}\left(b_{i},b_{j}\right)$ is a coarse path in $\mathcal{B}\left(b_{i},b_{j}\right)$ between $b_{i}$ and $b_{j}$. Since $\mathcal{B}\left(b_{i},b_{j}\right)$ is an unparametrized $(\lambda,\varepsilon)$--quasi-geodesic by Corollary \ref{cor:unicorn_hausdorff_dist}, so is $S_{i,j}$. Let $\delta$ be the hyperbolicity constant of $\mc G(\Sigma)$.  By  Lemma~\ref{lem:Morse_lemma}, there exists a constant $M=M\left(\delta, \lambda, \varepsilon\right)>0$ such that 
\[
d_{\mathrm{Haus}}\left(\left[b_{i},b_{j}\right],S_{i,j}\right)\leq M,
\]
where $\left[x,y\right]$ denotes the geodesic in $\mathcal{G}\left(\Sigma\right)$ between $x$ and $y$. Lemma \ref{lem:Gromov_product_vs_dist_to_geod} then implies that for all $i,j\geq N$, we have 
\begin{align*}
\left(b_{i},b_{j}\right)_{a} & \geq d_{\mathrm{Haus}}\left(a,\left[b_{i},b_{j}\right]\right)-2\delta\\
 & \geq d_{\mathrm{Haus}}\left(a,S_{i,j}\right)-M-2\delta\\
 & \geq R-M-2\delta.
\end{align*}
Since $R$ was arbitrary, $\liminf_{i,j\to\infty}(b_i,b_j)_a =\infty$, and so $\left( b_{n}\right) _{n\in\mathbb{N}}$ defines a point in $\partial\mathcal{G}\left(\Sigma\right)$.
\end{proof}



The following Proposition shows that an ending bicorn sequence coarsely describes the entire bicorn set. Here, $[x,y]$ denotes a geodesic between $x,y\in \mc{G}(\Sigma)$. 

\begin{prop}\label{prop:map_well_def}
Let $a\in\mathcal{G}\left(\Sigma\right)$, $L\in\mathcal{GEL}\left(\Sigma\right)$, $l\in L$ be a non-boundary leaf with both half-leaves dense and $\left(b_{n}\right)_{n\in\mathbb{N}}\subset\mathcal{B}\left(a,l\right)$ be an ending bicorn sequence. There exists a constant $K_0>0$, independent of $a,L$ and the ending bicorn sequence, such that for any $n\in \mathbb N$,
\[
d_{\mathrm{Haus}}^{\mc G}(\{b_0,\dots, b_n\},[a,b_n])\leq K_0.
\]
and
\[
d_{\mathrm{Haus}}^{\mathcal{G}}\left(\left(b_{n}\right)_{n\in\mathbb{N}},\mathcal{B}\left(a,l\right)\right)\leq K_0
\]
\end{prop}

\begin{proof}
Since $\mc G(\Sigma)$ is $\delta$--hyperbolic, there exists a constant $\kappa\geq 1$ such that the nearest point projection $\pi_\gamma$ to a geodesic $\gamma$ is $\kappa$--Lipschitz. 
Let $M_0$ be the constant from Corollary~\ref{cor:unicorn_hausdorff_dist}.

For the first statement, fix $n\in \mathbb N$. By Definition~\ref{def:ending}(4), we see that $\{b_0,\dots, b_n\}\subseteq \mc B(a,b_n)$.  Moreover, $d_{\text{Haus}}^{\mc G}(\mc B(a,b_n), [a,b_n])\leq M_0$ by Corollary~\ref{cor:unicorn_hausdorff_dist}.  Thus $\{b_0,\dots, b_n\}\subseteq N_{M_0}([a,b_n])$. 

Since $\{b_0,\dots, b_n\}$ is a $K$--quasi-path by Definition~\ref{def:ending}(3), every point on $[a,b_n]$ is at distance at most $\kappa K$ from a point in $\pi_{[a,b_n]}(\{b_0,\dots, b_n\})$.  In particular, if $z\in [a,b_n]$, then $d_{\mc G}(z,\{b_0,\dots, b_n\}\leq \kappa K + M_0$.  It follows that $d_{\text{Haus}}^{\mc G}(\{b_0,\dots, b_N\},[a,b_n])\leq \kappa K + M_0$, as desired.

For the second statement, let $x\in\mathcal{B}\left(a,l\right)$. Since $b_{n}\xrightarrow{\mathrm{geod}}l$, there exists $N>0$ such that $x\in\mathcal{B}\left(a,b_{N}\right)$ by Lemma \ref{lem:bicorns_for_converging_leaves}.  Combining the first statement with the fact that $d_{\text{Haus}}^{\mc G}(\mc B(a,b_N),[a,b_N])\leq M_0$ shows that there exists $z\in [a,b_N]$ such that 
\[
d_{\mc G}(x,(b_n)_{n\in \mathbb N})\leq d_{\mc G}(x,z)+d_{\mc G}(z,(b_n)_{n\in \mathbb N}) \leq M_0+\kappa K + M_0.
\]
Setting $K_0 = 2M_0+\kappa K$ completes the proof.
\end{proof}



\section{Grand ending laminations and the Gromov boundary}\label{sec:GEL_boundary}

The goal of this section is to prove Theorem~\ref{thm_main}. Let $\Sigma$ be an infinite-type surface with $3\leq |\mc{S}(\Sigma)|<\infty$, so that $\mc{G}(\Sigma)$ is an infinite-diameter hyperbolic graph. Define a map
 \[F\colon \mathcal{GEL}\left(\Sigma\right)\to\partial\mathcal{G}\left(\Sigma\right)\]  by 
\[
F\left(L\right)=\left[(b_n)_{n\in \mathbb N}\right]
\]
for some ending bicorn sequence $(b_n)_{n\in\mathbb{N}}$ between some $a\in \mc G(\Sigma)$ and $L$. We will show that $F$ is  a homeomorphism onto its image and that the image of $F$ is dense in $\partial\mc{G}(\Sigma)$.  We begin by showing well-definedness.

\begin{lem}
    The map $F$ is well-defined.
\end{lem}

\begin{proof}
    Let $a\in \mc G(\Sigma)$, and let $l\in L$ be a non-boundary leaf with both half-leaves dense.  Such a leaf exists by Lemma~\ref{lem:both_half_leaves_dense}. Let $(b_n)_{n\in \mathbb N}\subseteq \mc B(a,l)$ be an ending bicorn sequence. Since $L$ is minimal and witness-filling, Lemma~\ref{lem:defines_a_point} shows that  $(b_n)$ defines a point in $\partial\mc G(\Sigma)$. Therefore $F(L)\in \partial\mc G(\Sigma)$. 

    If $\left(b_{n}'\right)_{n\in\mathbb{N}}$ is a different ending bicorn sequence between $a$ and $L$, then using Proposition~\ref{prop:map_well_def}, the triangle inequality, and possibly Lemma \ref{lem:bicorns_with_lamination} if $\left(b_{n}\right)_{n\in\mathbb{N}}$ and $\left(b_{n}'\right)_{n\in\mathbb{N}}$ converge to different leaves, we obtain
\[
d_{\mathrm{Haus}}^{\mathcal{G}}\left(\left(b_{n}\right)_{n\in\mathbb{N}},\left(b_{n}'\right)_{n\in\mathbb{N}}\right)\leq2K+2.
\] 
Thus, any two ending bicorn sequences between $a$ and $L$ define the same point in $\partial\mathcal{G}\left(\Sigma\right)$.

Finally, Lemma \ref{lem:bicorns_change_base_pt} combined with Proposition \ref{prop:map_well_def} shows that starting with a different grand arc $a'$ changes the ending bicorn sequence by finite Hausdorff distance in $\mc{G}(\Sigma)$, so the point $F(L)$ is independent of the choice of $a\in\mc{G}(\Sigma)$.
\end{proof}

The following lemma shows that any other  sequence of grand arcs that converges to the boundary of the grand arc graph and Chabauty-converges to a given grand ending lamination will converge to the same point in $\partial\mc{G}(\Sigma)$ as defined by $F$.

\begin{lem}
\label{lem:other_convg_sequences}Suppose $L\in\mathcal{GEL}\left(\Sigma\right)$, $a\in\mathcal{G}\left(\Sigma\right)$. Suppose $\left(x_{n}\right)_{n\in\mathbb{N}}$ is a quasi-geodesic sequence with $x_0 = a$ such that:
\begin{enumerate}
\item as points of $\mathcal{G}\left(\Sigma\right)$, $x_{n}\to p$ for some $p\in\partial\mathcal{G}\left(\Sigma\right)$, and
\item as laminations, $x_{n}\xrightarrow{CC}L$. 
\end{enumerate}
Then $p=F(L)$.
\end{lem}

\begin{proof}
Let $\delta$ be the hyperbolicity constant of $\mc G(\Sigma)$. Let $\gamma$ denote the $(1,10\delta)$-quasigeodesic in $\mc{G}(\Sigma)$ starting at $x_0$ and converging to $p$, as described in Lemma \ref{lem:Qgeo_to_boundary}. Since $x_n\to p$ in $\mc G(\Sigma)\cup\partial\mc G(\Sigma)$ and since $(x_n)_{n\in\mathbb{N}}$ is a quasi-geodesic sequence, Lemma \ref{lem:QpathNearQgeo} implies that $d_\mathrm{Haus}^\mc{G}((x_n)_{n\in\mathbb{N}},\gamma)\leq C$ for some fixed $C>0$. Since finite bicorn sets are unparametrised quasi-geodesics by Corollary \ref{cor:unicorn_hausdorff_dist}, it follows from Lemma \ref{lem:QpathNearQgeo} that   $$d_{\mathrm{Haus}}^{\mc{G}}\left(\bigcup_{i\in\mathbb{N}}\mathcal{B}\left(a,x_{i}\right), \gamma\right)\leq C'.$$

On the other hand, let $l\in L$ be a non-boundary leaf with both half-leaves dense, and let $\left(b_{n}\right)_{n\in\mathbb{N}}\subset\mathcal{B}\left(a,l\right)$ be an ending bicorn sequence. Since $x_{n}\xrightarrow{CC}L$, then $x_{n}\xrightarrow{\mathrm{geod}}l$. By Lemma~\ref{lem:bicorns_for_converging_leaves}, for every $b\in\mathcal{B}\left(a,l\right)$, there exists some $N>0$ such that $b\in\mathcal{B}\left(a,x_{n}\right)$ for every $n\geq N$. In particular, $\left(b_{n}\right)_{n\in\mathbb{N}}\subset\bigcup_{i\in\mathbb{N}}\mathcal{B}\left(a,x_{i}\right)$. Thus, $\left(b_{n}\right)_{n\in\mathbb{N}}\subset N_{C''}((x_n)_{n\in\mathbb{N}})$ for some fixed $C''>0$. Since both $(b_n)$ and $(x_n)$ are quasi-geodesic sequences, with $x_n \to p$ and $b_n \to F(L)$ in $\mc G(\Sigma)\cup\partial\mc G(\Sigma)$, we have $p=F(L)$.
\end{proof}

\begin{example}\label{ex:agrees_with_F}
In Example \ref{ex:cpct_pA}, we defined points $\varphi^{\pm}\in\partial\mathcal{G}\left(\Sigma\right)$ corresponding to a pseudo-Anosov map supported on a compact witness $W$. We also described laminations $L^{\pm}\in\mathcal{GEL}\left(\Sigma\right)$ that depend only on $\varphi$, such that for any base arc $a\in\mathcal{G}\left(\Sigma\right)$, the sequences $\varphi^{\pm n}(a)\xrightarrow{CC}L^{\pm}$. Lemma \ref{lem:other_convg_sequences} then implies that $F\left(L^{\pm}\right)=\varphi^{\pm}$.
\end{example}

We now turn our attention to continuity.
\begin{lem}\label{lem:continuity}
The map $F$ is continuous.
\end{lem}

\begin{proof}
Suppose $\left( L_{i}\right) _{i\in\mathbb{N}}\subset\mathcal{GEL}\left(\Sigma\right)$ is a sequence of grand ending laminations such that $L_{i}\xrightarrow{CC}L$ for some $L\in\mathcal{GEL}\left(\Sigma\right)$, and let $a\in\mathcal{G}\left(\Sigma\right)$. By Corollary \ref{cor:seq_continuity}, it suffices to show that $F\left(L_{i}\right)\to F(L)$ in $\partial\mc G(\Sigma)$. Fix a leaf $l\in L$ and leaves $l_{i}\in L_{i}$ such that $l_{i}\geodto l$; this is possible by the definition of coarse Chabauty convergence; see Definition~\ref{def:CCconvergence}.

Let $\left(b_{n}^{i}\right)_{n\in\mathbb{N}}$ be an ending bicorn sequence between $a$ and $l_{i}\in L_{i}$, so that $F\left(L_{i}\right)=\left[\left(b_{n}^{i}\right)_{n\in\mathbb{N}}\right]$.  Similarly, let $(b_n)_{n\in \mathbb N}$ be an ending bicorn sequence between $a$ and $l\in L$, so that $F(L)=[(b_n)]$.  Since $l_{i}\to l$, for any $m$, there exists $R\geq0$ such that $\left\{ a,b_{1},\dots,b_{m}\right\} \subset\mathcal{B}\left(a,l_{r}\right)$ for every $r\geq R$ by Lemma \ref{lem:bicorns_for_converging_leaves}. It then follows from Proposition \ref{prop:map_well_def} that 
\[
\left\{ a,b_{1},\dots,b_{m}\right\} \subset N_{K}\left(\left(b_{n}^{r}\right)_{n\in\mathbb{N}}\right),
\]
where $K>0$ is a constant independent of $m$ and $R$. Since all ending bicorn sequences define a point in $\partial G(\Sigma)$ by Lemma \ref{lem:defines_a_point}, it follows that $F\left(L_{i}\right)\to F\left(L\right)$.
\end{proof}

\begin{lem}\label{lem:injective}
The map $F$ is injective.
\end{lem}

\begin{proof}
    Suppose $L,L'\in\mathcal{GEL}\left(\Sigma\right)$ are grand ending laminations such that $F\left(L\right)=F\left(L'\right)$. Let $\left( b_{n}\right) _{n\in\mathbb{N}}\subset \mathcal{B}\left(a,L\right)$ and $\left( b_{n}'\right) _{n\in\mathbb{N}}\subset \mathcal{B}\left(a,L'\right)$ be ending bicorn sequences such that $\left[\left( b_{n}\right) _{n\in\mathbb{N}}\right]=F\left(L\right)=F\left(L'\right)=\left[\left( b_{n}'\right) _{n\in\mathbb{N}}\right]$.   By Definition~\ref{def:ending} and Proposition~\ref{prop:map_well_def}, there exists a constant $K\geq 0$ such that $(b_n)$ and $(b_n')$ are $K$--quasi-paths and $\{b_0,\dots, b_k\}$ and $[a,b_k]$ are at Hausdorff distance at most $K$ for all $k\in \mathbb N$.  Let $\delta$ be the hyperbolicity constant of $\mc G(\Sigma)$, and let $C=C(K,\delta)$ be the constant provided by Lemma~\ref{lem:QpathNearQgeo}.

    Let $\gamma,\gamma'$ be $(1,10\delta)$--quasi-geodesic rays based at $a$ such that $[\gamma(m)_{m\in \mathbb N}]=[(b_n)_{n\in \mathbb N}]$ and $[\gamma'(m)_{m\in \mathbb N}]=[(b_n')_{n\in \mathbb N}]$, which exist by Lemma~\ref{lem:Qgeo_to_boundary}.  By Lemma~\ref{lem:QpathNearQgeo}, we have $d_{\text{Haus}}^{\mc G}((b_n),\gamma)\leq C$ and $d_{\text{Haus}}^{\mc G}((b_n'),\gamma')\leq C$.  By Lemma~\ref{lem:Morse_lemma} and the fact that $[\gamma(m)]=[\gamma'(m)]$, we see that $\gamma$ and $\gamma'$ are at Hausdorff distance at most $m_0(\delta, 1,10\delta)$.  Hence the triangle inequality shows that the Hausdorff distance between $(b_n)_{n\in \mathbb N}$ and $(b_n')_{n\in \mathbb N}$ is at most $C_0:=2C+m_0(\delta,1,10\delta)$.

    We proceed in a manner similar to the proof of Lemma~\ref{lem:grand_bicorns_go_to_infty}. For every $n>0$, consider the path in $\mathcal{G}\left(\Sigma\right)$ given by
\[
b_{n}=x_{n}^{0}-x_{n}^{1}-\dots-x_{n}^{C_0}=b_{n}'.
\]
A subsequence of $\left( x_{n}^{C_0-1}\right) _{n\in\mathbb{N}}$ converges to a lamination $L_{C_0-1}$, and $L'\subset L_{C_0-1}$ by Corollary~\ref{cor:disjointarcTOsublamination}. 
Since $L_{C_0-2}$ is disjoint from $L_{C_0-1}$, and hence $L'$, we can repeat the same argument to obtain that $L'\subset L_{C_0-2}$. Proceeding in this way, we find that $L'\subset L_{C_0-i}$ for all $i>0$. In particular, $L'\subset L_{0}=L$, but this can only happen if $L=L'$, since $L$ is also minimal. Therefore, $F$ is injective.
\end{proof}


By Lemma~\ref{lem:injective}, the map $F$ is a bijection onto its image and thus has an inverse $F^{-1}$ defined on the image of $F$.  The next lemma will be useful in proving that $F^{-1}$ is continuous.

Suppose some subsequence of $\left(L_{n}\right)_{n\in\mathbb{N}}$ converges in the Chabauty topology to a lamination $L$. Let $l_{n}\in L_{n}$ be leaves, and, up to choosing a subsequence, assume that $l_{n}\to l$ for some geodesic $l$. Recall that we chose a metric on $\Sigma$ by picking a pants decomposition of $\Sigma$ and setting the length of every seam of the pairs of pants to be $1$. Let $\left\{ W_{i}\right\} _{i\in\mathbb{N}}$ be an exhaustion of $\Sigma$ by compact witnesses, such that $W_{0}$ is an $n$-holed sphere (where $n=|\mc{S}(\Sigma)|$) with $\partial W_{0}$ a part of the pants decomposition, and $W_{i}$ is obtained from $W_{i+1}$ by attaching the pair of pants from the pants decomposition adjacent to each component of $\partial W_{i}$. Parametrize the leaves $l_{n}$ such that $l_{n}(0)\in W_{0}$. Then note that $l_{i}\left(\left[-n,n\right]\right)\subset W_{n}$, for each $i>0$.

\begin{lem}\label{lem:angles}
    There exists a grand arc $a$ and a constant $\theta\geq 0$ such that:
\begin{enumerate}
\item $a\cap W_{i}$ is a single arc for each $i\geq0$, and
\item the angles of intersection of $a\cap l_{n}\left[-n,n\right]$ are at least $\theta$ for each $n\geq1$. 
\end{enumerate}
\end{lem}

\begin{proof}
    By concatenating seams, there exists a grand arc $a'$ such that $a'\cap W_i$ is a single arc for each $i\geq 0$.  We will modify $a'$ to find the grand arc $a$ by Dehn twisting $a'$ along various boundary components of the pairs of pants, as described below.  The resulting grand arc will still satisfy statement (1) of the claim.

    Let $D$ be the diameter of each pair of pants, and fix some $0<\varepsilon < 1/4$.  There exists a constant $\theta_0$ such that the following holds.  Fix some $n$.  If $p$ is a transverse intersection of $a$ and $l$ inside a pair of pants $P$ and the angle of intersection at $p$ is at most $\theta_0$, then there is an rectangular strip of width $\varepsilon$ embedded in $P$ containing $a\cap P$ and $l\cap P$; see Figure \ref{fig:angle_for_F_inverse}.  In particular, $a$ and $l$ must wrap the same number of times about the boundary components of $P$.  Modify $a'$ by Dehn twisting finitely many times about the boundary components of the pairs of pants inside $W_n$ to form a new grand arc $a$, we can ensure that all transverse intersections of $a$ and $l$ have angle at least $\theta_0$.

    \begin{figure}[h]
    \centering
    \def\svgwidth{10 cm}
\begingroup%
  \makeatletter%
  \providecommand\color[2][]{%
    \errmessage{(Inkscape) Color is used for the text in Inkscape, but the package 'color.sty' is not loaded}%
    \renewcommand\color[2][]{}%
  }%
  \providecommand\transparent[1]{%
    \errmessage{(Inkscape) Transparency is used (non-zero) for the text in Inkscape, but the package 'transparent.sty' is not loaded}%
    \renewcommand\transparent[1]{}%
  }%
  \providecommand\rotatebox[2]{#2}%
  \newcommand*\fsize{\dimexpr\f@size pt\relax}%
  \newcommand*\lineheight[1]{\fontsize{\fsize}{#1\fsize}\selectfont}%
  \ifx\svgwidth\undefined%
    \setlength{\unitlength}{245.55115948bp}%
    \ifx\svgscale\undefined%
      \relax%
    \else%
      \setlength{\unitlength}{\unitlength * \real{\svgscale}}%
    \fi%
  \else%
    \setlength{\unitlength}{\svgwidth}%
  \fi%
  \global\let\svgwidth\undefined%
  \global\let\svgscale\undefined%
  \makeatother%
  \begin{picture}(1,0.16250682)%
    \lineheight{1}%
    \setlength\tabcolsep{0pt}%
    \put(0,0){\includegraphics[width=\unitlength,page=1]{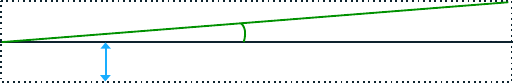}}%
    \put(0.5070826,0.0923197){\color[rgb]{0,0.56862745,0}\makebox(0,0)[lt]{\lineheight{1.25}\smash{\begin{tabular}[t]{l}\textit{$\theta_0$}\end{tabular}}}}%
    \put(0.20916454,0.03254009){\color[rgb]{0.09803922,0.68235294,1}\makebox(0,0)[lt]{\lineheight{1.25}\smash{\begin{tabular}[t]{l}\textit{$\epsilon$}\end{tabular}}}}%
  \end{picture}%
\endgroup%

    \captionsetup{margin=1.5 cm, justification=justified, singlelinecheck=false}
    \caption{Since the length of every pants seam is $1$, we can find $\theta_0$ such that for any geodesic intersecting $l$ inside the pair of pants which makes an angle at least $\theta_0$, the intersection will leave the rectangle.}
    \label{fig:angle_for_F_inverse}
\end{figure}

    Since $l_n$ geodesically converges to $l$ as $n\to\infty$, for all $t>0$, there exists $N\in \mathbb N$ such that $l_n[-t,t]$ and $l[-t,t]$ are at Hausdorff distance at most $\varepsilon/2$ for all $n\geq N$. For each $t$, there exists a constant $\beta_t>0$ such that any angle of intersection of $a$ and $l[-t,t]$ differs by at most $\beta_t$ from the corresponding angle of intersection of $a$ and $l_n[-t,t]$.  Moreover, $\beta_t\to 0$ monotonically as $t\to \infty$; see Figure \ref{fig:angles_for_converging_leaves}.

    \begin{figure}[h]
        \centering
    \def\svgwidth{10 cm}
\begingroup%
  \makeatletter%
  \providecommand\color[2][]{%
    \errmessage{(Inkscape) Color is used for the text in Inkscape, but the package 'color.sty' is not loaded}%
    \renewcommand\color[2][]{}%
  }%
  \providecommand\transparent[1]{%
    \errmessage{(Inkscape) Transparency is used (non-zero) for the text in Inkscape, but the package 'transparent.sty' is not loaded}%
    \renewcommand\transparent[1]{}%
  }%
  \providecommand\rotatebox[2]{#2}%
  \newcommand*\fsize{\dimexpr\f@size pt\relax}%
  \newcommand*\lineheight[1]{\fontsize{\fsize}{#1\fsize}\selectfont}%
  \ifx\svgwidth\undefined%
    \setlength{\unitlength}{250.10164036bp}%
    \ifx\svgscale\undefined%
      \relax%
    \else%
      \setlength{\unitlength}{\unitlength * \real{\svgscale}}%
    \fi%
  \else%
    \setlength{\unitlength}{\svgwidth}%
  \fi%
  \global\let\svgwidth\undefined%
  \global\let\svgscale\undefined%
  \makeatother%
  \begin{picture}(1,0.3263631)%
    \lineheight{1}%
    \setlength\tabcolsep{0pt}%
    \put(0,0){\includegraphics[width=\unitlength,page=1]{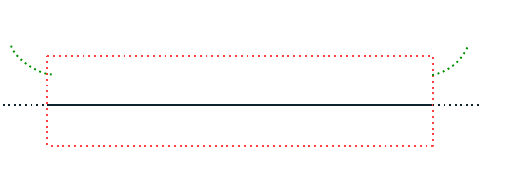}}%
    \put(0.87690993,0.08639579){\color[rgb]{0.17647059,0.17647059,0.17647059}\makebox(0,0)[lt]{\lineheight{1.25}\smash{\begin{tabular}[t]{l}\textit{$l$}\end{tabular}}}}%
    \put(0.89230949,0.17009483){\color[rgb]{0,0.56862745,0}\makebox(0,0)[lt]{\lineheight{1.25}\smash{\begin{tabular}[t]{l}\textit{$l_n$}\end{tabular}}}}%
    \put(0,0){\includegraphics[width=\unitlength,page=2]{angles_for_converging_leaves.pdf}}%
    \put(0.56020272,0.13505427){\color[rgb]{1,0.25490196,0.25490196}\makebox(0,0)[lt]{\lineheight{1.25}\smash{\begin{tabular}[t]{l}\textit{$\theta$}\end{tabular}}}}%
    \put(0.62369826,0.18519999){\color[rgb]{1,0.25490196,0.25490196}\makebox(0,0)[lt]{\lineheight{1.25}\smash{\begin{tabular}[t]{l}\textit{$\theta - \beta_t$}\end{tabular}}}}%
    \put(0.62378698,0.24893138){\color[rgb]{0.72941176,0,1}\makebox(0,0)[lt]{\lineheight{1.25}\smash{\begin{tabular}[t]{l}\textit{$a$}\end{tabular}}}}%
    \put(0,0){\includegraphics[width=\unitlength,page=3]{angles_for_converging_leaves.pdf}}%
  \end{picture}%
\endgroup%

    \captionsetup{margin=1.5 cm, justification=justified, singlelinecheck=false}
    \caption{Since $l_n[-t,t]$ and $l[-t,t]$ are close for $n\gg 0$, the angle of intersection between $a$ and $l_n$ is close to the angle of intersection between $a$ and $l$.}
    \label{fig:angles_for_converging_leaves}
    \end{figure}

 Therefore, the angles of transverse intersections of $a$ and $l_n[-t,t]$ are bounded below by $\theta:=\theta_0-\beta_1$ for all $t\geq 1$, and so (2) holds.
\end{proof}

\begin{prop} \label{prop:inverse-cont}
The map $F^{-1}$ is continuous.
\end{prop}

\begin{proof}
It suffices to show that if $\left( L_{n}\right) _{n\in\mathbb{N}}\subset\mathcal{GEL}\left(\Sigma\right)$ and $L_{0}\in\mathcal{GEL}\left(\Sigma\right)$ are such that $F\left(L_{n}\right)\to F\left(L_{0}\right)$ in $\partial\mathcal{G}\left(\Sigma\right)$, then $L_{n}\xrightarrow{CC}L_{0}$.

Let the grand arc $a$ and constant $\theta>0$ be those provided by Lemma~\ref{lem:angles}, and let $\delta$ be the hyperbolicity constant of $\mc G(\Sigma)$.
For each $n$, let $\left(b_{n,i}\right)_{n\in\mathbb{N}}\subset\mathcal{B}\left(a,L_{n}\right)$ be an ending bicorn sequence such that $F\left(L_{n}\right)=[(b_{n,i})_{i\in \mathbb N}]$. 


Since $F(L_n)\to F(L_0)$, for each $r>0$ there exists $n_r>0$ such that $(F(L_n), F(L_0))_a\geq r$ for all $n>n_r$.  Fix $\varepsilon>0$, and let $\delta$ be the hyperbolicity constant of $\mc G(\Sigma)$.  There exist indices $I,J$ such that $(b_{n_r,J},b_{0,I})_a\geq r-\varepsilon$.  Then the points $x\in [a,b_{n_r,J}]$ and $y\in [a,b_{0,I}]$ with $d_{\mc G}(a,x)=d_{\mc G}(a,y)=\lfloor r-\varepsilon\rfloor$ satisfy $d_{\mc G}(x,y)\leq 2\delta$.  Since $\{b_{n_r,1},\dots, b_{n_r,J}\}\subseteq N_{K}([a,b_{n_r,J}])$ and $\{b_{0,1},\dots, b_{0,I}\}\subseteq N_K([a,b_{0,I}])$ by Proposition~\ref{prop:map_well_def}, there exists $i_r>0$ such $b_{n_r,i_r}\in N_{2K+2\delta}((b_{0,i})_{i\in \mathbb N})$.  Moreover, $r-\varepsilon -K \leq d_{\mc G}(a,b_{n_r,i_r})\leq r-\varepsilon+K$.  In particular, the sequence $(b_{n_r,i_r})_{r\in \mathbb N}$ is is unbounded in $\mathcal{G}\left(\Sigma\right)$ and is a bounded distance from $\left(b_{0,i}\right)_{i\in\mathbb{N}}$, so it converges to the point $\left[\left(b_{n_{r},i_{r}}\right)_{r\in\mathbb{N}}\right]=\left[\left(b_{0,i}\right)_{i\in\mathbb{N}}\right]=F\left(L_{0}\right)$ in $\partial \mc{G}(\Sigma)$.

\begin{claim}
    $\left(b_{n_{r},i_{r}}\right)\xrightarrow{CC}L_{0}$ as $r\to \infty$.  
\end{claim}

\begin{proof}
    We have shown that for each $r$ there exists $m_r$ such that $d_{\mc G}(b_{n_r,i_r},b_{0,m_r})\leq C':=2K+2\delta$. Let 
\[
b_{n_r,i_r}=x_{r,0}-x_{r,1}-\dots-x_{r,C'}=b_{m_r}
\]
be a path between $b_{n_r,i_r}$ and $b_{m_r}$ in $\mathcal{G}\left(\Sigma\right)$, for each $r$. Up to choosing subsequences, we can assume that each $\left(x_{r,i}\right)_{n\in\mathbb{N}}$ geodesically converges to a lamination $L_{i}$, with $L_{C'}=L$. By Corollary~\ref{cor:disjointarcTOsublamination}, $L\subset L_{C'-1}$.
 Repeating this argument, we see that $L\subset L_{i}$ for each $i$. Thus, $L\subset M=L_{0}$, as desired.
\end{proof}

Since $a$ is a grand arc, there is an $\eta>0$ such that $N_{\eta}(a)$ is an embedded strip in $\Sigma$.  It follows that consecutive intersections of $l_n$ and $a$ must be at distance at least $2\eta$ along $l_n$.  Since $b_{n,i}$ is formed using at least the first $i$ interesections of $a$ and $l_n$, it follows that the $l_{n}$ part of $b_{n,i}$ contains at least $l_{n}\left(\left[-i\eta,i\eta\right]\right)$.  By Lemma~\ref{lem:angles}(2), the angles of intersection are bounded below by $\theta$. 

\begin{claim}\label{claim:lnbnclose}
    For any $\epsilon>0$ and $R>0$, we can choose $r>0$ large enough such that
\begin{equation}\label{eqn:lnbnclose}
d_{\Sigma}\left(b_{n_{r},i_{r}}(t),l_{n_{r}}(t)\right)<\epsilon
\end{equation}
for $t\in\left[-R,R\right]$. 
\end{claim}

\begin{proof}
    Fix a lift $\tilde l$ in $\mathbb H^2$ and lifts $\tilde l_n$ geodesically converging to $\tilde l$.  Since $b_{n,i}$ converges to $l_n$, there are lifts $\tilde a_j$ of $a$ whose intersections with $\tilde l_n$ converge to both endpoints of $\tilde l_n$.  Since $\tilde l_n$ converges to $\tilde l$, the intersection points $\tilde a_j\cap \tilde l$ also converge to both ends of $\tilde l$.  The geodesic $\tilde l$ divides $\partial\mathbb H^2$ into two components, $U^+$ and $U^-$.  Denote by $\tilde a_j^{\pm \infty}$ the endpoints of $\tilde a_j$ in $U^{\pm}$, respectively.

    Suppose $b_{n_r,i_r}$ does not converge to $\ell$.  Then it must be the case that, without loss of generality, $\tilde a_j^{+\infty}$ converges to one endpoint of $\tilde\ell$ while $\tilde a_j^{-\infty}$ does not; for example, see Figure~\ref{fig:bad_angle_bicorns}. However, this implies that the angles of intersection of $\tilde a_j$ with $\tilde l$ must go to 0 as $j\to\infty$.  This contradicts our choice of $a$.
\end{proof}

Since $\left(b_{n_{r},i_{r}}\right)_{r\in\mathbb{N}}$ coarse Chabauty converges to $L_{0}$, we have that $b_{n_{r},i_{r}}\geodto l_{0}$ for every leaf $l_{0}\in L_{0}$. By  Claim~\ref{claim:lnbnclose}, it follows that $l_{n_{r}}\geodto l_{0}$, as well, proving that $\left(L_{n}\right)_{n\in\mathbb{N}}$ coarse Chabauty converges to $L_{0}$, as desired.
\end{proof}

We have shown that $F$ is a homeomorphism onto its image.  We next prove that the image of $F$ is dense. In fact, letting $\mathcal{GEL}_{\mathrm{cpct}}\left(\Sigma\right)$ denote the subspace of $\mathcal{GEL}\left(\Sigma\right)$ consisting of minimal laminations that fill a \textit{compact} witness, we will show that $F\left(\mathcal{GEL}_{\mathrm{cpct}}\left(\Sigma\right)\right)$ is dense in $\partial\mathcal{G}\left(\Sigma\right)$.

\begin{thm}\label{thm:dense}
    The image of $F$ is dense in $\partial\mc G(\Sigma)$.
\end{thm}

\begin{proof}
    Let $\delta$ be the hyperbolicity constant of $\mc G(\Sigma)$, and let $p=[\left(x_{j}\right)_{j\in\mathbb{N}}]$ be a point in $\partial\mathcal{G}\left(\Sigma\right)$, where $\left(x_{j}\right)_{j\in\mathbb{N}}$ are the vertices on a $(1,10\delta)$--quasi-geodesic path. For each $i\in\mathbb{N}$, let $W_{i}$ be a compact witness large enough such that $x_{j}\cap W_{i}$ is connected for each $1\leq j\leq i$; such a witness exists by Lemma~\ref{lem:single_arc_witness}. Let $x_{j}^{i}=x_{j}\cap W_{i}$ for $1\leq j\leq i$. By Corollary \ref{cor:quasi_geod_projection}, $\left(x_{j}^{i}\right)_{j=1,\dots,i}$ is a finite $(k,c)$--quasi-geodesic in $\mathcal{A}\left(W_{i},\Gamma_{W_{i}}\right)$, where $k$ and $c$ are independent of $i$. 

    Let $A(W_i,P_i)$ denote the relative arc graph of $W_i$, where $P_i$ is the collection of boundary components of $W_i$ that separate elements of the grand splitting. By Corollary \ref{cor:extending-quasigeod}, we can extend $\left(x_{j}^{i}\right)_{j=1,\dots,i}$ to an infinite quasi-geodesic  $\left(x_{j}^{i}\right)_{j\in\mathbb{N}}$ in $\mathcal{A}\left(W_{i},P_i\right)$, which necessarily converges to $\partial \mc A(W_i,P_i)$. Let $q_i := [\left(x^i_j\right)_{n\in\mathbb{N}}]$. By \cite[Corollary 2.2]{kop-asdim}, we have that 
    \[\partial\mathcal{A}(W_i,P_i) \subset \partial\mathcal{A}(W_i,\Gamma_{W_i}).\] 
    Up to choosing a different sequence at Hausdorff distance 1  from $\left(x_{j}^{i}\right)_{i\in\mathbb{N}}$, we may assume that each $x_{j}^{i}$ is in $\mc{A}\left(W_i,\Gamma_{W_i}\right)$, that is, we may assume that no $x_j^i$ is a loop: if some $x_j^i$ is a loop, build a prescribed arc disjoint from $x_j^i$ by starting at a different boundary component, following any path disjoint from $x_j^i$ from this boundary component to a point on $x_j^i$, and continuing along $x_j^i$ in either direction. We can then extend each $x_{j}^{i}$ to a grand arc $y_{j}^{i}$, with $y_{j}^{i}=x_{j}$ for $j\leq i$ such that the arcs $y_j^i$ are disjoint outside $W_{i}$. Thus we can assume that $\left(x_{j}^{i}\right)_{j\in\mathbb{N}}$ is a  sequence of grand arcs whose projections to $W_{i}$ converge to a point in $\partial\mathcal{A}\left(W_{i},P_{i}\right)$. It follows from Lemma \ref{lem:presc_arc_proj} that this sequence is a $(k',c')$--quasi-geodesic in $\mc G(\Sigma)$, where $k',c'$ depend only on $k,c$, and hence converges to a point in $\partial\mathcal{G}\left(\Sigma\right)$. Since, for each $i$, the quasigeodesic sequence $(x^i_j)_{j\in \mathbb N}$ converging to $q_i$ agrees with $[\left(x_{i}\right)_{i\in\mathbb{N}}]$ for the first $i$ terms, it follows that $q_i \to p$ in $\partial \mc G(\Sigma)$ as $i\to\infty$. 

   For each $i$, it follows from \cite[Theorem~1.2]{Pho-On} that $\left(x_{j}^{i}\right)_{j\in\mathbb{N}}$ converges in the coarse Chabauty topology to a minimal lamination $L_{i}$ that fills a witness subsurface of $\mathcal{A}\left(W_{i},P_{i}\right)$ in $W_{i}$. Such a witness must contain, and hence separate, every element of $P_{i}$, and so must also be a witness of $\mathcal{G}\left(\Sigma\right)$. It follows that  $L_{i}\in\mathcal{GEL}_{\mathrm{cpct}}\left(\Sigma\right)$, and $x_{j}^{i}\xrightarrow{CC}L_{i}$ as $j\to\infty$. 
   
   \begin{claim}
       $q_{i}=F\left(L_{i}\right)$ for each $i$.
   \end{claim}

   \begin{proof}
       Let $(b_m)_{m\in \mathbb N}$ be an ending bicorn sequence in $\mc B(a,L_i)$, so that $F(L_i)=[(b_m)_{m\in \mathbb N}]$.  Pick a leaf $l\in L_i$.  By the definition of coarse Chabauty convergence, we have $x_j^i\geodto l$ as $j\to \infty$.  Thus,  Lemma~\ref{lem:bicorns_for_converging_leaves} ensures that for each $m$ there exists $n_m$ such that $b_m\in \mc B(a,x^i_{n_m})$.  By Corollary~\ref{cor:unicorn_hausdorff_dist} and the fact that $(x^i_j)_{j\in \mathbb N}$ is a $(k',c')$--quasi-geodesic with $k',c'$ independent of $i$, the sets $\mc B(a,x^i_{n_m})$ are contained in the $(M_0+m_0(\delta, k',c'))$--neighborhood of the sequence $\{x^i_j\}_{j\in \mathbb N}$, where $M_0$ is the constant from Corollary~\ref{cor:unicorn_hausdorff_dist} and $m_0(\delta,k',c')$ is the constant provided by Lemma~\ref{lem:Morse_lemma}.  Thus each $b_m$ is contained in a uniform neighborhood of $\{x^i_j\}_{j\in \mathbb N}$. 
       Since the sequence $\{b_m\}_{m\in \mathbb N}$ converges to infinity in $\mc G(\Sigma)$, it follows that $[\{b_m\}_{m\in \mathbb N}]=[\{x^i_j\}_{j\in \mathbb N}]=q_i$ in $\partial\mc G(\Sigma)$.  Therefore, $q_i=F(L_i)$.
   \end{proof}

For each $p\in \partial\mc G(\Sigma)$, we have constructed a sequence of laminations $L_{i}\in\mathcal{GEL}_{\mathrm{cpct}}\left(\Sigma\right)$ such that $F\left(L_{i}\right)\to p$ in $\partial \mc G(\Sigma)$. Therefore, the image of $F$ is dense in $\partial\mathcal{G}\left(\Sigma\right)$. 
\end{proof}

\subsection{Equivariance of $F$}

In this subsection, we show that the map $F$ is equivariant.  To do so, we first describe the action of $\mathrm{MCG}(\Sigma)$ on $\mc{GEL}(\Sigma)$.

Let $\tau\in\mathrm{MCG}\left(\Sigma\right)$, and let $\varphi\in\tau$ be a homeomorphism $\Sigma\to\Sigma$. Since the metric on $\Sigma$ is Nielsen-convex, we can lift $\varphi$ to a map $\hat{\varphi}\colon \mathbb{H}^{2}\to\mathbb{H}^{2}$ that extends to a homeomorphism $\partial\hat{\varphi}\colon \partial\mathbb{H}^{2}\to\partial\mathbb{H}^{2}$ that is invariant under homotopy and equivariant under the deck group action; see, for example, \cite[Proposition~6]{Tap}.  In particular, $\partial\hat\varphi$ depends only on the mapping class $\tau$. Thus, $\tau$ induces a continuous map on $\Gamma\left(\Sigma\right)$, and hence a continuous map on $\Lambda\left(\Sigma\right)$. If $L$ is a minimal lamination, then so is $\tau(L)$, as these two laminations are homeomorphic. Similarly, if $\mathrm{supp}(L)$ is a witness, then so is $\mathrm{supp}\left(\tau(L)\right)=\tau\left(\mathrm{supp}(L)\right)$. Thus if $L\in\mathcal{GEL}\left(\Sigma\right)$, then $\tau(L)\in\mathrm{MCG}\left(\Sigma\right)$, and so  $\mathrm{MCG}\left(\Sigma\right)$ admits a well-defined action on $\mathcal{GEL}\left(\Sigma\right)$.  On the other hand, the action of $\mathrm{MCG}\left(\Sigma\right)$ on $\mathcal{G}\left(\Sigma\right)$ by isometries extends to a continuous action on $\partial\mathcal{G}\left(\Sigma\right)$; see \cite[Theorem 1.3]{B-NV}. The following lemma shows that the map $F$ agrees with these actions.

\begin{lem}\label{lem:bicorns_preserved_homeo}
If $a\in\mathcal{G}\left(\Sigma\right)$, $l\in\Gamma\left(\Sigma\right)$ and $\tau\in\mathrm{MCG}\left(\Sigma\right)$, then  $\tau\left(\mathcal{B}\left(a,l\right)\right)=\mathcal{B}\left(\tau(a),\tau(l)\right).$
\end{lem}

\begin{proof}
Let $\varphi\in\tau$ be a homeomorphism $\Sigma\to\Sigma$. Assume that $a$ and $l$ are geodesics on $\Sigma$, so that they are in minimal position. Let $b\in\mathcal{B}\left(a,l\right)$ be a bicorn arc, so that $b=\left(\dots P\right)_{a}\cup\left(PQ\right)_{l}\cup\left(Q\dots\right)_{a}$ for some points $P,Q\in a\cap l$. Since $\varphi$ is a homeomorphism, it preserves topological information about intersections between $a$ and $l$, and the ends that $a$ converges to. Thus $\varphi(b)=\left(\dots\varphi(P)\right)_{\varphi(a)}\cup\left(\varphi(P)\varphi(Q)\right)_{\varphi(l)}\cup\left(\varphi(Q)\dots\right)_{\varphi(a)}$, and in particular, $\varphi(b)\in\mathcal{B}\left(\tau(a),\tau(l)\right)$. Thus, $\tau\left(\mathcal{B}\left(a,l\right)\right)\subset\mathcal{B}\left(\tau(a),\tau(l)\right)$. Conversely, for any bicorn arc $b'=\left(\dots P'\right)_{\varphi(a)}\cup\left(P'Q'\right)_{\varphi(l)}\cup\left(Q'\dots\right)_{\varphi(a)}$ for some points $P',Q'\in\varphi\left(a\right)\cap\varphi(l)$, we see that $\varphi^{-1}(b')\in\mathcal{B}\left(a,l\right)$. Thus $\mathcal{B}\left(\tau(a),\tau(l)\right)\subset\tau\left(\mathcal{B}\left(a,l\right)\right)$ as well. 

\end{proof}
%
\begin{cor}
The map $F\colon \mathcal{GEL}\left(\Sigma\right)\to\partial\mathcal{G}\left(\Sigma\right)$ is equivariant with respect to the action of $\mathrm{MCG}\left(\Sigma\right)$. 
\end{cor}

\begin{proof}
Let $l\in L$ be a non-boundary leaf with both half-leaves dense, and let $a\in\mathcal{G}\left(\Sigma\right)$. Let $\left(b_{n}\right)_{n\in\mathbb{N}}\subset\mathcal{B}\left(a,l\right)$ be an ending bicorn sequence. Then $b_{n}\xrightarrow{CC}L$ and $b_{n}\xrightarrow{\mathrm{Gromov}}F(L)$. Since $\mathrm{MCG}\left(\Sigma\right)$ acts on $\mathcal{G}\left(\Sigma\right)$ by automorphisms, and this action extends continuously to $\partial\mathcal{G}\left(\Sigma\right)$, we see that $\tau\left(b_{n}\right)\to\tau\left(F(L)\right)$ in $\mc G(\Sigma)\cup\partial \mc G(\Sigma)$. On the other hand, $b_{n}\xrightarrow{\mathrm{geod}}l$ by the definition of coarse Chabauty convergence. Moreover, $\tau\left(b_{n}\right)\xrightarrow{CC}\tau(L)$, as $\tau$ induces a continuous map on $\Gamma\left(\Sigma\right)$ and $\Lambda\left(\Sigma\right)$, as described at the beginning of this subsection. Further note that $\tau$ sends non-boundary leaves of $l$ to other non-boundary leaves of $l$, and if both half-leaves of $l$ are dense in $L$, then the same is true for $\tau(l)$. It follows that $\left(\tau\left(b_{n}\right)\right)_{n\in\mathbb{N}}$ is a sequence of grand arcs in $\mathcal{B}\left(\tau(a),\tau(l)\right)$ by Lemma \ref{lem:bicorns_preserved_homeo}.  In fact, $\left(\tau\left(b_{n}\right)\right)_{n\in\mathbb{N}}$ is an ending bicorn sequence between $\tau(a)$ and $\tau(l)$: Definition~\ref{def:ending}(1) is immediate, (2) was shown above, (3) holds because it holds for $(b_n)_{n\in \mathbb N}$ and $\tau$ is an automorphism of $\mc G(\Sigma)$, and (4) holds because it holds for $(b_n)_{n\in \mathbb N}$ and this property is preserved under homeomorphisms.  
The conclusion follows.
\end{proof}

\subsection{WWPD elements of $\mathrm{MCG}(\Sigma)$}\label{sec:WWPD}

The description of a dense subset of point in $\partial\mc{G}(\Sigma)$ as a space of geodesic laminations allows us to describe the WWPD elements $\mathrm{MCG}(\Sigma)$ in the action on $\mc{G}(\Sigma)$.

\begin{defn}\label{def:WWPD}
Suppose a group $G$ acts on a hyperbolic metric space $X$ by isometries. We say that $g\in G$ is WWPD if
\begin{enumerate}
\item $g$ is loxodromic, with fixed points $g^{+},g^{-}\in\partial X$, and 
\item for any sequence $h_{n}\in G$ with $h_{n}\left(g^{+}\right)\to g^{+}$ and $h_{n}\left(g^{-}\right)\to g^{-}$ in $\partial X$, there exists $N>0$ such that $h_{n}\left(g^{+}\right)=g^{+}$ and $h_{n}\left(g^{-}\right)=g^{-}$ for any $n\geq N$.
\end{enumerate}
\end{defn}

WWPD elements were first defined in \cite{HM}, and existence of WWPD elements can be used to show that the second bounded cohomology of the group is uncountably infinite dimensional. 

When $\Sigma$ is the once-punctured Cantor tree, the grand arc graph is quasi-isometric to the \textit{loop graph}, which is the graph whose vertices correspond to isotopy classes of arcs converging to the puncture at both ends and whose edges correspond to disjointness.  More generally, when $\Sigma$ has an isolated puncture, then $\mathrm{MCG}(\Sigma)$ acts on the loop graph, which is defined analogously, but this hyperbolic graph is no longer quasi-isometric to the grand arc graph. Rasmussen showed that when $\Sigma$ has an isolated puncture, the WWPD elements of $\mathrm{MCG}(\Sigma)$ in the action  on the loop graph are exactly the ones which are supported on a compact witness and are pseudo-Anosov on the witness \cite[Theorem 1.1]{RasWWPD}. This generalizes the fact that for finite-type surfaces $S$, the WWPD elements of $\mathrm{MCG}(S)$ in the action on the curve graph are exactly the pseudo-Anosov mapping classes; in fact, in this context, pseudo-Anosov mapping classes satisfy the stronger WPD property, but that is not relevant for this paper.  Bavard and Walker describe the boundary of the loop graph as a space of cliques of \textit{high-filling rays}. The proof of \cite[Theorem 1.1]{RasWWPD} involves associating to each clique of high-filling rays the geodesic lamination that is its limit set; see \cite[Section 3]{RasWWPD}. Moreover, if the rays in the clique are all supported on a finite-type witness, then its limit set is a minimal lamination filling the witness.  

The arguments in \cite[Sections 4 \& 5]{RasWWPD} generalise verbatim to $\mc{G}(\Sigma)$ when $3 \leq |\mc{S}(\Sigma)|<\infty$, using grand ending laminations instead of limit sets of cliques of high-filling rays.  In light of this, we include only an outline of the proof.

\begin{thm}\label{thm:WWPD_classif}
Let $\Sigma$ be an infinite-type surface with $3\leq |\mc S(\Sigma)|<\infty$.  Suppose $\varphi\in\mathrm{MCG}(\Sigma)$ is loxodromic in the action on $\mc G(\Sigma)$ and the fixed points of $\varphi$ in $\partial\mc{G}(\Sigma)$ are in the image $F(\mc{GEL}(\Sigma))$. Then $\varphi$ is WWPD if and only if $\varphi$ is a pseudo-Anosov on a compact witness.
\end{thm}

\begin{proof}[Outline of Proof]
    First, suppose $\varphi\in\mathrm{MCG}\left(\Sigma\right)$ is pseudo-Anosov on a compact witness. Note that the fixed points $\varphi^{\pm}$ of $\varphi$ are in $ F\left(\mathcal{GEL}\left(\Sigma\right)\right)$; see Example \ref{ex:agrees_with_F}.  Since $F$ is a $\mathrm{MCG}\left(\Sigma\right)$-equivariant embedding of $\mathcal{GEL}\left(\Sigma\right)\to\partial\mathcal{G}\left(\Sigma\right)$, for any sequence $\tau_n\in\mathrm{MCG}(\Sigma)$, we have $\tau_{n}\left(\varphi^{\pm}\right)\to\text{\ensuremath{\varphi}}^{\pm}$ in $\partial \mc G(\Sigma)$ if and only if $\tau_{n}\left(F\left(\varphi^{\pm}\right)\right)\xrightarrow{CC}F\left(\varphi^{\pm}\right)$.  By \cite[Lemma 4.2]{RasWWPD}, which holds verbatim, it follows that if $\tau_{n}(L)\xrightarrow{CC}L$ for some $L\in\mathcal{GEL}\left(\Sigma\right)$, then $\tau_{n}\left(\mathrm{supp}(L)\right)=\mathrm{supp}(L)$ for all $n\gg0$. Thus we can consider the maps $\tau_n$ and $\varphi$ restricted to $W$ as elements of $\mathrm{MCG}(W)$.  As in \cite[Theorem 4.1]{RasWWPD}, which also holds verbatim, we use the fact that pseudo-Anosov mapping classes on compact surfaces are WWPD for the action on the curve graph to conclude the same about their action on $\mc{G}(\Sigma)$.

    Now, suppose $\varphi$ is a WWPD element of the $\mathrm{MCG}\left(\Sigma\right)$-action on $\mathcal{G}\left(\Sigma\right)$, and suppose that the fixed points $\varphi^{\pm}\in\partial\mathcal{G}\left(\Sigma\right)$ are in $F\left(\mathcal{GEL}\left(\Sigma\right)\right)$. Say $\varphi^{\pm}=F\left(L^{\pm}\right)$ for some $L^{+},L^{-}\in\mathcal{GEL}\left(\Sigma\right)$.  Let $L=L^+$.  Following the proof of \cite[Theorem 5.2]{RasWWPD}, which holds with a leaf of the lamination $L$ in place of the ray $l$, the support $\mathrm{supp}(L)$ must be finite-type. Since $\varphi^{\pm}$ are fixed by $\varphi$ and $F$ is equivariant, it follows that $\varphi(L)=L$ and, in particular, $\varphi\left(\mathrm{supp}(L)\right)=\mathrm{supp}(L)$. Let $W=\mathrm{supp}(L)$. Then $\phi$ stabilizes $\mc A(W)\subseteq \mc G(\Sigma)$, and since $\mc A(W)$ quasi-isometrically embeds into $\mc G(\Sigma)$, we see that $\phi|_W\in \mathrm{MCG}(W)$ must be loxodromic with respect to the action on  $\mathcal{A}(W)$.  Thus $\varphi^{\pm n}\left(\alpha\right)\xrightarrow{CC}L^{\pm}$ for any arc $\alpha\in\mathcal{A}\left(W\right)$. If $\varphi$ were not pseudo-Anosov on $W$, then some power acts as the identity on $W$, and so there is some $n$ such that $\varphi^n$ fixes an arc on $W$, contradicting that $\varphi|_W$ is a loxodromic isometry of $\mc A(W)$.  Thus $\varphi$ is pseudo-Anosov on $W$, with attracting and repelling laminations $L^{\pm}$.
\end{proof}

Combining Theorem~\ref{thm:WWPD_classif} with \cite[Theorem~2.10]{HM}, we obtain a new proof of the following result, which is originally due to Horbez, Qing and Rafi \cite{HQR}. 

\begin{cor}\label{cor:BddCohom}
    Let $\Sigma$ be an infinite-type surface with $3\leq |\mc S(\Sigma)|<\infty$. Suppose  $H\leq \mathrm{MCG}(\Sigma)$ is a subgroup such that:
    \begin{enumerate}
        \item $H$ contains a pair of independent loxodromic elements in the action on $\mc{G}(\Sigma)$ and
        \item there exists $\varphi\in H$ which is pseudo-Anosov on a compact witness.
    \end{enumerate}
    Then $H_b ^2 (H;\mathbb{R})$ is uncountably infinite dimensional.
\end{cor}

\pagebreak

\appendix
\section{Notes on the Gromov boundary of the ray graph} \label{sec:appendix}
\appauthor{Lvzhou Chen}

The goal of this appendix is to prove Proposition~\ref{prop: main}. 
This is one step in the proof of \cite[Theorem~5.1.1]{B-W}: Bavard and Walker start from a boundary point of the ray graph, choose a quasi-geodesic representative, pass to a cover-convergent subsequence, and then need to know that the limiting long ray is high-filling. 
The original proof of this fact appears to have a small gap, so we give a detailed proof.

We follow the notation and conventions of \cite{B-W}.
Thus $S$ is the plane minus a Cantor set, with the isolated planar end denoted by
$\infty$.
The completed ray graph $\Rcal$ has as vertices
short rays, loops (based at $\infty$), and long rays, with edges given by
disjointness. 
It has a main component $\Rcal_0$, which contains all loops and all rays that are not high-filling. We know $\Rcal_0$ is quasi-isometric to the ray graph, so they have the same Gromov boundary.

Also recall that cover-convergence is convergence in the conical cover: Each oriented loop or ray has a unique lift starting at the preferred lift $\widetilde{\infty}$ of $\infty$, and a
sequence cover-converges when the endpoints of these preferred lifts converge
in the boundary circle of the conical cover. 

\begin{prop}\label{prop: main}
	Let $(x_i)$ be a quasi-geodesic representing a point $p$ in the Gromov boundary of $\Rcal_0$. Up to taking a subsequence, we suppose that $(x_i)$ cover-converges to an element $\ell$ in the conical circle. Then $\ell$ is high-filling.
\end{prop}

\subsection{About (infinite) unicorn paths}
The proof relies on some geometric understanding of finite and infinite unicorn paths in the sense of \cite[Section~3]{B-W}.
There are a few equivalent ways to define the unicorn path $P(a,b)$ between two oriented loops $a$ and $b$. The path is simply $a,b$ if they are disjoint, so we assume that $a$ and $b$ intersect and are in minimal position.
There are finitely many intersections between $a$ and $b$, and each intersection $x$ uniquely determines a (not necessarily simple) loop that is made of an initial segment $(\infty x)_b$ of $b$ and a terminal segment $(x\infty)_a$ of $a$. We only look at the simple ones, and order them (as we progress from $a$ to $b$) so that the element has more portion on $b$ and less portion on $a$. The key fact is that, given that we are looking for simple loops, taking more from $b$ is equivalent to taking less from $a$. This already yields two equivalent definitions: 
\begin{enumerate}
	\item List the intersections in the order along $b$, and collect the simple loops of the form above; or
	\item List the intersections in the order along $a$, and collect the simple loops of the form above.	
\end{enumerate}
The first description also leads to a third equivalent definition that inductively constructs the loops in the unicorn path as follows. Start with $a_0=a$, and then for each $k\ge0$, let $x$ be the first intersection along $b$ with $a_k$ and let $a_{k+1}$ be the concatenation of $(\infty x)_b$ and $(x\infty)_{a_k}$\footnote{The original definition in \cite[Definition~3.2.1]{B-W} appears to have a typo; this should be the correct version.}, which we set to be $b$ and complete the process if $a_k$ and $b$ are already disjoint. It is easy to note by induction that:
\begin{enumerate}
	\item Perturbing $(\infty x)_b$ slightly in the direction of $a_k$ at $x$ gives a representative of $a_{k+1}$ in minimal position with $b$,
	\item each $a_k$ starts with a segment on $b$ and then a segment on $a$.
\end{enumerate} 
From this, it should be clear why this is equivalent to the earlier definitions.

Now when $b$ is a (long) ray, since the intersections with $a$ are discrete along $b$, the first and third definitions still work out well, and they remain equivalent. This is the definition of the infinite unicorn path $P(a,b)$; the inductive description agrees with \cite[Definition~3.3.1]{B-W}. The following lemma can be observed from the first definition.
\begin{lem}\label{lemma: key}
	Let $a$ be a loop. If a sequence of loops or rays $\ell_i$ cover-converges to a long ray $\ell$, then for any $n$, there is $K>0$ such that for all $i\ge K$, the first $n$ elements in the unicorn path $P(a,\ell_i)$ agree with those in $P(a,\ell)$.
	
	Moreover, if $a$ and $\ell$ have infinitely many intersections, then for any (two independent) sequences of integers $i_k,j_k\to\infty$, the $j_k$-th term in the unicorn path $P(a,\ell_{i_k})$ cover-converges to $\ell$ as $k\to\infty$.
\end{lem}
\begin{proof}
	The first assertion is obvious from the first definition, as $\ell_i$ starts almost the same as $\ell$ for all $i\ge K$ with $K$ sufficiently large. 
	
	The second assertion follows from Figure~\ref{fig: unicorn}
	and the explanation below.
	Let $\winf$ be the preferred lift of the planar end $\infty$. Let $\tilde{\ell}$ be the lift of $\ell$ that starts at $\winf$. Let $\{\tilde{a}_i\}_{i=1}^\infty$ be the lifts of $a$ based at lifts of intersections along $\tilde{\ell}$. These intersections are discrete since $a$ is proper. Thus the lifts shrink down to the endpoint of $\tilde{\ell}$. So for any neighborhood $V$ of $\ell$ in the conical circle, there is $n$ such that for all $j\ge n$ the endpoints of $\tilde{a}_j$ determine a closed interval neighborhood $I_j\subset V$. For any ray or loop $\gamma$ in $I_j$, realized as a geodesic from $\winf$ to a point in $I_j$ in the upper half plane model, there is a unique intersection $x$ with $\tilde{a}_j$, and the concatenation $(\infty x)_{\gamma}\cup(x \infty)_a$ lies in $I_j$ and thus in $V$. For any $k$ large such that $j_k\ge n$ and $i_k$ is large enough so that $\ell_{i_k}\in I_n$, the $j_k$-th term in the unicorn path is constructed using the $j$-th intersection along $\ell_{i_k}$ with $a$ for some $j\ge j_k\ge n$ and hence its lift starting at $\winf$ has endpoint in $I_n\subset V$. This completes the proof.
\end{proof}

\begin{figure}
	\labellist
	\small 
	\pinlabel $\widetilde{\infty}$ at 166 227
	\pinlabel $\tilde{\ell}$ at 170 205
	\pinlabel $\tilde{\ell}_i$ at 157 205
	\pinlabel $\tilde{\ell}_2$ at 142 205
	\pinlabel $\tilde{\ell}_1$ at 112 205
	\pinlabel $\gamma$ at 205 205
	\pinlabel $\tau$ at 220 205
	\pinlabel $\bar{\tau}$ at 260 205
	\pinlabel $\tilde{a}_1$ at 200 160
	\pinlabel $\tilde{a}_2$ at 190 140
	\pinlabel $\tilde{a}_j$ at 180 100
	\pinlabel $\tilde{a}_{j'}$ at 180 40
	\pinlabel $I_j$ at 190 10
	\pinlabel $V$ at 193 -2
	\endlabellist
	\centering
	\includegraphics[scale=1]{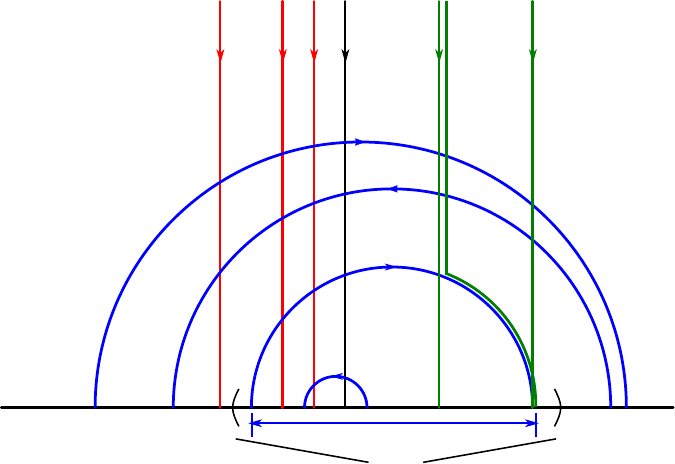}
	\vspace{5pt}
	\caption{The lifts $\tilde{a}_j$'s of $a$ shrink to the endpoint of $\tilde{\ell}$. The curve $\tau$ is the concatenation of $\gamma$ and $\tilde{a}_j$ at their unique intersection $x$. Its straightening $\bar{\tau}$ is the lift of the geodesic representative of $(\infty x)_\gamma\cup (x\infty)_a$, and it lands in $I_j\subset V$.}\label{fig: unicorn}
\end{figure}
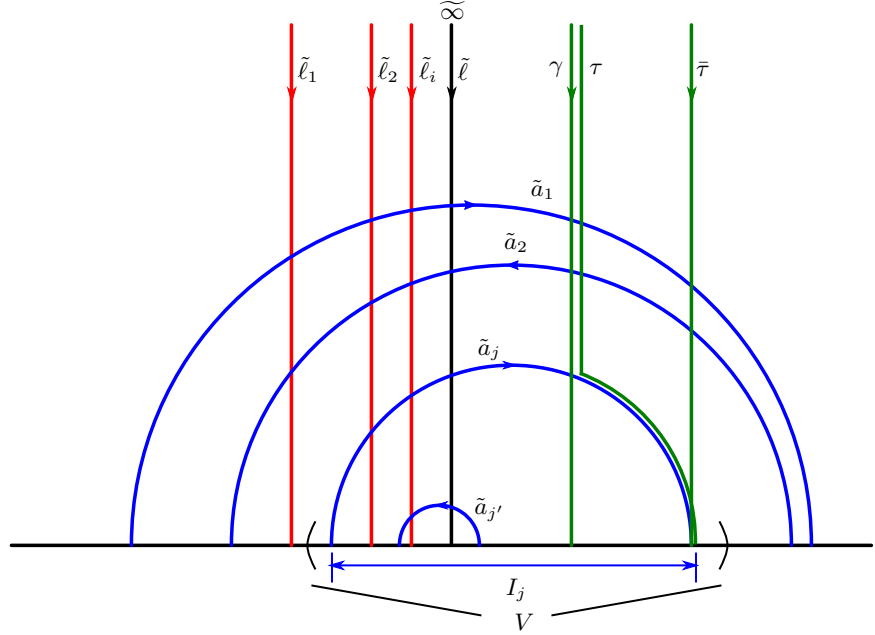

\subsection{Proof of the proposition}
Now we give a proof of Proposition~\ref{prop: main} by contradiction. 

\begin{proof}[Proof of Proposition~\ref{prop: main}]
	Suppose $\ell$ is not high-filling, or equivalently by the characterization in \cite[Lemma~3.5.4]{B-W}, 
	the infinite unicorn path $P(a,\ell)$ is bounded in $\Rcal_0$ for any oriented loop $a$.
	
	Up to adding an element to the quasi-geodesic, we assume that $x_0$ is a loop and work with $a=x_0$. It is already shown in \cite[Lemma~4.3.1]{B-W} that $\ell$ is a loop-filling ray. In particular, $\ell$ must intersect $x_0$ infinitely many times: If not, the unicorn path $P(x_0,\ell)$ is finite, and the second to last term would give rise to a loop disjoint from $\ell$.
	
	Since $(x_i)$ is a quasi-geodesic, by \cite[Lemma~3.2.6]{B-W}, there is $C\in\Z_{\ge0}$ (independent of $k$ and $n$) such that for all $n>k$, the unicorn path $P(x_0,x_k)$ is contained in the $C$-neighborhood of $P(x_0,x_n)$. This is the only property in the proof that uses the fact that $(x_i)$ is a quasi-geodesic, apart from the fact that $d(\alpha,x_i)\to\infty$ for any $\alpha\in \Rcal_0$.
	
	As we assumed $\ell$ to be not high-filling, the distance $d(\ell,x_k)$ is finite and goes to $\infty$ as $k\to\infty$.
	Fix any $k$ such that $d(\ell,x_k)> C$.
	Then for any $n>k$, there is $y_n$ on the unicorn path $P(x_0,x_n)$ with $d(x_k,y_n)\le C$. We suppress the dependence on $k$ as we will keep it fixed until the end of the argument.
	Now there is a path $z_{n,i}$ in $\Rcal_0$ with $z_{n,1}=x_k$ and $z_{n,C}=y_n$. By taking a subsequence of $n$ we may assume that each $z_{n,i}$ cover-converges to some $z_{\infty,i}$ as $n\to \infty$ fixing $1\le i\le C$. In particular, $y_n$ cover-converges to $y_\infty\defeq z_{\infty,C}$.
	Disjointness is preserved under cover-convergence, thus these limits form a path in $\Rcal_0$, and we have $d(x_k,y_\infty)\le C$.
	
	Suppose $y_n$ is the $i_n$-th element on the unicorn path $P(x_0,x_n)$. If $i_n$ is unbounded, then up to taking a subsequence, it goes to infinity. By the second assertion in Lemma~\ref{lemma: key}, we must have that $y_n$ cover-converges to $\ell=y_\infty$, which gives $d(x_k,\ell)\le C$ and we get a contradiction.
	
	Therefore, the sequence $i_n$ must be bounded. Then by taking a subsequence, it is a constant $i$, and $y_n$ is the $i$-th term on the unicorn path $P(x_0,\ell)$ for all $n$ large by the first assertion in Lemma~\ref{lemma: key}, as $x_n$ cover-converges to $\ell$. This shows that $x_k$ is in the $C$-neighborhood of $P(x_0,\ell)$, which holds for all $k$ large. However, our assumption implies that the path $P(x_0,\ell)$ is bounded and so is its $C$-neighborhood. This contradicts the fact that $x_k$ escapes every bounded subset of $\Rcal_0$. Thus $\ell$ must be high-filling as desired.
\end{proof}

\begingroup
\DeclareEmphSequence{\itshape} 
\bibliographystyle{alpha}
\bibliography{Bibliography}
\endgroup

\end{document}